\documentclass[12pt,reqno]{amsart}
\usepackage{amsfonts}
\usepackage{amsmath}
\usepackage{amssymb}
\usepackage{graphicx}
\usepackage{enumerate}
\usepackage{multicol}
\usepackage{mathrsfs}
\usepackage[usenames,dvipsnames]{pstricks}
\usepackage{epsfig}
\usepackage{pst-grad} % For gradients
\usepackage{pst-plot} % For axes
\usepackage[hmargin=3cm,vmargin=2.5cm]{geometry}
\usepackage{rotating}
\usepackage[utf8]{inputenc}
\usepackage{mathtools}
\usepackage{xypic}
\usepackage{stackrel}

\newtheorem{defi}{Definition}[section]
\newtheorem{rem}[defi]{Remark}
\newtheorem{prop}[defi]{Proposition}
\newtheorem{theorem}[defi]{Theorem}

\newtheorem{notation}[defi]{Notation}

\newtheorem{coro}[defi]{Corollary}
\newtheorem{lemma}[defi]{Lemma}

\DeclareMathOperator{\n}{\mathfrak{n}}

\DeclareMathOperator{\z}{\mathfrak{z}}
\DeclareMathOperator{\vv}{\mathfrak{v}}

\DeclareMathOperator{\la}{\langle}
\DeclareMathOperator{\ra}{\rangle}
\DeclareMathOperator{\End}{End}

\DeclareMathOperator{\A}{\mathbf{A}}
\DeclareMathOperator{\B}{\mathbf{B}}
\DeclareMathOperator{\spn}{span}
\DeclareMathOperator{\card}{card}

\DeclareMathOperator{\ad}{ad}
\DeclareMathOperator{\Id}{Id}
\DeclareMathOperator{\Cl}{\rm{Cl}}

\DeclareMathOperator{\sign}{sign}

\title{Classification of pseudo $H$-type algebras}
\author[Christian Autenried, Kenro Furutani, Irina Markina]{ Christian Autenried,  Kenro Furutani, Irina Markina}

\address{Department of Mathematics, University of Bergen, Norway.}
\email{christian.autenried@math.uib.no}

\address{K.~Furutani:  Department of Mathematics, Faculty of Science and Technology, Science University of Tokyo, 2641 Yamazaki, Noda, Chiba (278-8510), Japan}
\email{furutani\_kenro@ma.noda.tus.ac.jp}

\address{Department of Mathematics, University of Bergen, Norway.}
\email{irina.markina@math.uib.no}

\thanks{The authors are partially supported by the NFR-FRINAT grants \#204726/V30 and \#213440/BG.}


\begin{document}\title{Classification of pseudo $H$-type algebras}
\maketitle

%\tableofcontents

%%%%%%%%%%%%%%%%%%%%%%%%%%%%%%%%%%%%%%%%%%%%%%%%

\begin{section}{Introduction}

%%%%%%%%%%%%%%%%%%%%%%%%%%%%%%%%%%%%%%%%%%%%%%%%

A. Kaplan introduced the Lie groups of Heisenberg type or shortly $H$-type groups in 1980~\cite{Ka1} and studied in details, for instance in~\cite{Ka2,Ka3}. The $H$-type Lie algebras of the $H$-type Lie groups, constructed in~\cite{Ka1}, used the presence of an inner product on the Lie algebra. Later this approach was extended by exploiting an arbitrary indefinite non-degenerate scalar product in~\cite{Ciatti, GodoyKorolkoMarkina} and the introduced Lie algebras received the name pseudo $H$-type algebras. This construction is closely related to the existence of a special scalar product on the representation space of Clifford algebras. Namely the Clifford algebras $\Cl_{r,0}$ generated by a positive definite scalar product space $(\mathbb R^{r,0},\langle\cdot\,,\cdot\rangle_{r,0})$, which lead to the $H$-type Lie algebras $\n_{r,0}$ introduced by A.~Kaplan and the Clifford algebras $\Cl_{r,s}$ generated by indefinite non-degenerate scalar product  spaces $(\mathbb R^{r,s},\langle\cdot\,,\cdot\rangle_{r,s})$ creating pseudo $H$-type Lie algebras~$\n_{r,s}$.

In the present work we study the isomorphism properties of the Lie algebras $\n_{r,s}$, that were constructed as pseudo $H$-type Lie algebras. Thus, we neglect the presence of the scalar product $\langle\cdot\,,\cdot\rangle_{r,s}$ on the Lie algebra $\n_{r,s}$ and study isomorphisms of Lie algebras as themselves. The isomorphism of Lie algebras defines the isomorphism of the corresponding Lie groups. In the present paper we mostly concentrate on the minimal admissible modules. We showed that the Lie algebras $\n_{r,s}$ can not be isomorphic to $\n_{u,t}$ unless $r=t$ and $s=u$ or $r=u$ and $s=t$. The question of existence of an isomorphism between $\n_{r,s}$ and $\n_{s,r}$ is much more complicated. We proved that if $s=0$, then the Lie algebras $\n_{r,0}$ and $\n_{0,r}$ are isomorphic if the dimensions of the centers coincide. If $r,s\neq 0$, then we present examples of both cases: the isomorphic pairs and non isomorphic pairs, having equal dimensions. Some of the Lie algebras, that we call of block type, allow to use the Bott periodicity of underlying Clifford algebras and obtain more isomorphic pairs, see Theorems~\ref{th:14} and \ref{th:15}.

We stress an interesting feature, that there are no direct relations between the isomorphisms of the underlying Clifford algebras and the isomorphisms of generated pseudo $H$-type Lie algebras. In some cases the isomorphic Clifford algebras lead to isomorphic Lie algebras, in other cases not. For instance, in spite of the isomorphism of the Clifford algebras $\Cl_{8,0}$, $\Cl_{0,8}$, and $\Cl_{4,4}$, the corresponding Lie algebras $\n_{8,0}$, $\n_{0,8}$ are isomorphic, but not isomorphic to the Lie algebra $\n_{4,4}$. 

The structure of the article is the following. After this Introduction we give all necessary definitions and notations in Section~\ref{sec:prelim}. Furthermore, the relation between Clifford algebras and pseudo $H$-type Lie algebras is presented. In Section~\ref{sec:80_08}, we discuss a necessary condition for isomorphisms of pseudo $H$-type Lie algebras $\n_{r,s}$, which shows that the only possible algebra which is isomorphic to $\n_{r,s}$ is the pseudo $H$-type Lie algebra $\n_{s,r}$. Section~\ref{sec:low_dim} is devoted to the complete classification of pseudo $H$-type Lie algebras $\n_{r,0}$ and $\n_{0,r}$. Sections~\ref{New} and~\ref{sec:nonisom} study two different situations revealing that the Lie algebras $\n_{r,s}$ and $\n_{s,r}$ can be both isomorphic and non-isomorphic. In Section~\ref{strbracksec}, we exhibit the strongly bracket generating property of the pseudo $H$-type Lie algebras $\n_{r,0}$ and $\n_{0,s}$ and the non-existence of this property for pseudo $H$-type algebras $\n_{r,s}$ with $r,s \not=0$.  Furthermore, we introduce an equivalent definition for pseudo $H$-type Lie algebras, that was introduced in~\cite{GodoyKorolkoMarkina}, and explain the equivalence of these definitions in details in subsection~\ref{sec:equivPsGen}. In Section~\ref{sec7}, we briefly discuss the isomorphism of pseudo $H$-type Lie algebras related to non-equivalent irreducible Clifford modules. Finally Appendix~\ref{appendix}, gives the commutator tables of $\n_{8,0}$, $\n_{0,8}$, $\n_{4,4}$ and a table of permutations for a basis of $\n_{8,0}$.
\end{section}

%%%%%%%%%%%%%%%%%%%%%%%%%%%%%%%%%%%%%%%%%%%%%%%%

\begin{section}{Preliminaries}\label{sec:prelim}

%%%%%%%%%%%%%%%%%%%%%%%%%%%%%%%%%%%%%%%%%%%%%%%%

\subsection{Clifford algebras and their representations}\label{sec:Cliff}

%%%%%%%%%%%%%%%%%%%%%%%%%%%%%%%%%%%%%%%%%%%%%%%%

Throughout this paper we assume all scalar products to be non-degenerate besides otherwise stated. We denote by $\n$ a nilpotent $2$-step Lie algebra endowed with a scalar product $\langle  \cdot\,,\cdot  \rangle_{\n}$ of signature $(r,s)$, $r,s\in\mathbb N$, $r+s=n$: that means that there exists a basis $\{U_1, \dotso ,U_n\}$ of $\n$ which satisfies 
$$\langle U_i \,, U_j \rangle_{\n}=\epsilon_i(r,s)\delta_{ij},\quad\text{where}\quad 
\epsilon_i(r,s)=
\begin{cases} 1, & \text{for } i=1,\dotso , r ,
\\ -1, & \text{for } i=r+1, \dotso , r+s. 
\end{cases}
$$ 
The definition of the index varies in the literature, thus we follow to one which is given above.
Let $\z$ be the center of the 2-step nilpotent Lie algebra $\n$ and $\langle\cdot\,,\cdot\rangle_{\z}$ the restriction of the scalar product $\langle  \cdot\,,\cdot  \rangle_{\n}$ to $\z$. We assume that $\langle\cdot\,,\cdot\rangle_{\z}$ is non-degenerate. Then the orthogonal complement $\vv:=\z^{\bot}$ is also a non-degenerate scalar product space, where we use the symbol $\langle\cdot\,,\cdot\rangle _{\vv}$ to denote the restriction of $\langle  \cdot\,,\cdot  \rangle_{\n}$ on $\vv$. Thus  $\n=\z\oplus_{\bot}\vv$ is an orthogonal decomposition with respect to the scalar product $\langle\cdot\,,\cdot\rangle_{\n}=\langle \cdot\,,\cdot \rangle_{\z} + \langle \cdot\,,\cdot \rangle_{\vv}$. Since $\z$ is the center of $\n$, the commutator is a skew-symmetric bi-linear map $[\cdot\,,\cdot] \colon \vv \times \vv \to \z$.

\begin{defi}\label{def:J_Z}
Let $\n=(\z\oplus_{\bot}\vv,[\cdot\,,\cdot],\langle\cdot\,,\cdot\rangle_{\n} )$ be a Lie algebra described above. We define the map $J \colon \z \to \End(\vv)$ by
\begin{equation}\label{eq:def_J}
 \langle J_Zv , w \rangle_{\vv} = \langle Z , [v, w]\rangle_{\z},\quad \text{for all }v,w \in \vv.
\end{equation}
\end{defi}

\begin{defi}\label{def:pseudo}\cite{Ciatti}
We call a 2-step nilpotent Lie algebra $\n=(\z\oplus_{\bot}\vv,[\cdot\,,\cdot],\langle\cdot\,,\cdot\rangle_{\n} )$ with $J \colon \z \to \End(\vv)$ from Definition~\ref{def:J_Z} a pseudo $H$-type Lie algebra if
\begin{equation}\label{eq:J_Zcomposition}
\langle J_Zv , J_Zv \rangle_{\vv} = \langle Z \,, Z \rangle_{\z} \langle v,v \rangle_{\vv}
\quad\text{for all}\quad Z\in \z\ \text{and}\ v\in \vv.
\end{equation} 
We write $\n_{r,s}$ to emphasize that $\la \cdot \,, \cdot \ra_{\z}$ has signature $(r,s)$.
\end{defi}

We call the attention of the reader that the sign in equation~\eqref{eq:J_Zcomposition} differs from the original work~\cite{Ciatti} due to the different agreement on the signature. It follows directly from Definition~\ref{def:J_Z} that $J_Z$ is skew-adjoint with respect to the scalar product $\langle \cdot\,,\cdot  \rangle_{\vv}$:
\begin{equation}\label{eq:skew_J}
 \langle J_Zv , w \rangle_{\vv} =-\langle v , J_Zw \rangle_{\vv}\quad\text{for all}\quad Z\in\z, \quad v,w \in \vv. 
\end{equation}
Using polarization in~\eqref{eq:J_Zcomposition} we obtain
\begin{equation}\label{eq_depol}
\langle J_Zv, J_{Z'}v  \rangle_{\vv} = \langle Z , Z' \rangle_{\z} \langle v,v \rangle_{\vv}, \quad\text{and}\quad
\langle J_Zv , J_{Z}v'  \rangle_{\vv} = \langle Z , Z \rangle_{\z} \langle v , v'\rangle_{\vv}.
\end{equation}
Applying the skew-adjoint property~\eqref{eq:skew_J} one also obtains
\begin{equation}\label{eq:CliffordRepres}
J_Z \circ J_Z:=J_ZJ_Z=J_Z^2=-\langle Z , Z \rangle_{\z} \Id_{\vv},\ \text{ or }\ 
J_{Z'}J_Z+J_ZJ_{Z'}=-2\langle Z , Z' \rangle_{\z} \Id_{\vv}, 
\end{equation}
for all $Z,Z'\in \z$. 
Equality~\eqref{eq:CliffordRepres} implies that $J\colon \mathfrak{z} \to \End(\vv)$ defines a representation of the Clifford algebra $\Cl(\mathfrak{z} , \langle \cdot \,, \cdot \rangle_{\z} )$. We recall the definition of a Clifford algebra and its representations.

\begin{defi}
A Clifford algebra $\Cl(\mathfrak{z} , \langle \cdot \,, \cdot \rangle_{\z} )$ generated by a scalar product space $(\mathfrak{z},\langle \cdot \,, \cdot \rangle_{\z})$ 
is the unital associative algebra generated by $\mathfrak{z}$  subject to the relations:
$$Z\otimes Z=-\langle Z , Z \rangle_{\z} \mathbb I_{\Cl(\mathfrak{z} , \langle \cdot \,, \cdot \rangle_{\z} )},$$
for all $Z \in \mathfrak{z}$ and the unit $\mathbb I_{\Cl(\mathfrak{z} , \langle \cdot \,, \cdot \rangle_{\z} )}$ of the Clifford algebra.
\end{defi}
Thus, $\mathfrak{z}$ can be considered as a subset of $\Cl(\mathfrak{z} , \langle \cdot \,, \cdot \rangle_{\z} )$ and  
$$ Z\otimes W + W\otimes Z = -2 \langle Z , W \rangle_{\z}\mathbb I_{\Cl(\mathfrak{z} , \langle \cdot \,, \cdot \rangle_{\z} )}\quad\text{for all}\quad Z ,W \in \mathfrak{z}.$$

\begin{prop}\cite{Hus,LawMich}
Let $J \colon \mathfrak{z} \to \mathcal{A}$ be a linear map into an associative algebra $\mathcal{A}$ with an identity element $\Id_{\mathcal{A}}$ and product $"\cdot_A"$, such that
$$ 
J(Z)\cdot_A J(Z)=-\langle Z \,, Z \rangle_{\z} \Id_{\mathcal{A}}\quad\text{for all}\quad Z \in \mathfrak{z}.
$$
Then $J$ extends uniquely to an algebra homomorphism $\tilde{J} \colon \Cl(\mathfrak{z} , \langle \cdot \,, \cdot \rangle_{\z} ) \to \mathcal{A}$. Moreover, 
$\Cl(\mathfrak{z} , \langle \cdot \,, \cdot \rangle_{\z } )$ is the unique associative algebra with this property.
\end{prop}

\begin{defi}
A representation of a Clifford algebra $\Cl(\mathfrak{z} , \langle \cdot \,, \cdot \rangle_{\z} )$ is an algebra homomorphism 
$
 J \colon \Cl(\mathfrak{z} , \langle \cdot \,, \cdot \rangle_{\z} ) \to\End(\mathfrak{v})$ into the algebra of linear transformations of a finite dimensional vector space $\mathfrak{v}$. The space $\mathfrak{v}$ is called a 
$\Cl(\mathfrak{z} , \langle \cdot \,, \cdot \rangle_{\z} )$-module.
\end{defi}
Since $J$ is an algebra homomorphism we obtain that 
$$J_Z^2:=J_Z\circ J_Z=J_{Z\otimes Z}=J_{-\langle Z, Z \rangle_{\z} \mathbb I_{\Cl(\mathfrak{z}, \langle \cdot \,, \cdot \rangle_{\z})}} = -\langle Z, Z \rangle_{\z} \Id_{\mathfrak{v}}.$$
The scalar product space $(\z,\langle \cdot \,, \cdot \rangle_{\z})$ with the signature $(r,s)$ is isomorphic to $\mathbb{R}^{r,s}=(\mathbb{R}^{r+s},\langle \cdot \,, \cdot \rangle_{r,s})$, $r+s=n$, where the scalar product $\langle \cdot \,, \cdot \rangle_{r,s}$ is defined by
$
\langle Z, W \rangle_{r,s} = \sum_{i=1}^r{Z_iW_i}-\sum_{j=r+1}^{r+s}{Z_jW_j}$ for all $Z,W \in \mathbb{R}^{r+s}$. It allows us to use the isomorphism of the Clifford algebras $\Cl(\z,\langle \cdot \,, \cdot \rangle_{\z})$ and $\Cl(\mathbb{R}^{r,s}, \langle \cdot \,, \cdot \rangle_{r,s})$.
The Clifford algebra $\Cl(\mathbb{R}^{r,s}, \langle \cdot \,, \cdot \rangle_{r,s})$ is denoted by $\Cl_{r,s}$, and we benote by $(Z_1,\ldots, Z_{r+s})$ the orthonormal basis of $\mathbb{R}^{r,s}$ with $\langle Z_i, Z_j \rangle_{r,s}= \epsilon_i(r,s) \delta_{ij}$.

%%%%%%%%%%%%%%

\subsection{Admissible Clifford modules and pseudo $H$-type Lie algebras}

%%%%%%%%%%%%%%

In this section we explain when a representation space $\vv$ of a Clifford algebra can be endowed with a scalar product $\langle\cdot\,,\cdot\rangle_{\vv}$ such that the representation map $J$ satisfies~\eqref{eq:skew_J}. We call a positive definite scalar product we an {\it inner product}, and in any case we work with {\it only} non-degenerate scalar products. 

\begin{prop}\label{exist}\cite{Hus}
Let $J \colon \Cl_{r,0} \to \End(\mathfrak{v})$ be a representation. Then there exists an 
inner product $\langle \cdot \,, \cdot \rangle_{\mathfrak{v}}$ on $\mathfrak{v}$, such that the following holds:
\begin{equation}\label{adm0}
\langle J_Z w, J_Z v \rangle_{\mathfrak{v}} = \langle w,v \rangle_{\mathfrak{v}}\quad\text{for all}\quad w, v \in \mathfrak{v},\ Z \in \mathbb{R}^{r,0}\quad\text{with}\quad\langle Z , Z \rangle_{r,0}=1.
\end{equation}
\end{prop}

\begin{coro}\label{coro1}
Any representation $J \colon \Cl_{r,0} \to \End(\mathfrak{v})$ satisfies property~\eqref{eq:skew_J} with respect to an 
inner product $\langle \cdot \,, \cdot \rangle_{\mathfrak{v}}$.
\end{coro}
\begin{proof}
Corollary~\ref{coro1} follows by replacing $w$ by $J_Zw$ in~\eqref{adm0} and applying $J^2_Z=-\Id_{\vv}$. 
\end{proof}
Thus the Clifford algebras $\Cl_{r,0}$ possess always an inner product on $\vv$ such that $J_Z$ is skew-adjoint for all $Z \in \mathbb R^{r,0}$.  A.~Kaplan used inner products on $\vv$ for the construction of $H$-type algebras, which are based on $\Cl_{r,0}$-modules, see~\cite{Ka1,Ka3}. For $\Cl_{r,s}$-modules with $s\geq1$ equation~\eqref{adm0} is only true for orthonormal bases and in general not true for an arbitrary element of $\mathbb{R}^{r,s}$, see~\cite{Hus}.

\begin{defi}\cite{Ciatti}
A pair $(\vv, \langle \cdot \,, \cdot \rangle_{\vv})$, where $\vv$ is a $\Cl_{r,s}$-module is said to be an admissible $\Cl_{r,s}$-module if the representation operators $J_Z \colon \vv \to \vv$ are skew-adjoint with respect to $\langle \cdot \,, \cdot \rangle_{\vv}$, i.e. satisfies~\eqref{eq:skew_J} for all $Z\in\mathbb R^{r,s}$.
\end{defi} 

The following proposition guarantees the existence of admissible $\Cl_{r,s}$-modules.

\begin{prop}\cite{Ciatti}
For any given $\Cl_{r,s}$-module $\mathfrak{v}$ the vector space $\mathfrak{v}$ itself $($or $\mathfrak{v}\oplus\mathfrak{v}$ $)$ can be equipped with a scalar product $\langle \cdot \,, \cdot \rangle_{\mathfrak{v}}$ $($or $\langle \cdot \,, \cdot \rangle_{\mathfrak{v}\oplus\mathfrak{v}}$ $)$, such that
$$\langle J_Zw \,, v \rangle_{\vv}= - \langle w \,, J_Zv\rangle_{\mathfrak{v}},\quad\text{or}\quad \langle J'_Zw \,, v \rangle_{\mathfrak{v}\oplus\mathfrak{v}}= - \langle w \,, J'_Zv\rangle_{\mathfrak{v}\oplus\mathfrak{v}}$$
for all $Z \in \mathbb{R}^{r,s}$ and all $w,v \in \mathfrak{v}$ $($or $w,v \in \mathfrak{v}\oplus\mathfrak{v}$, where the operator $J'\colon \Cl_{r,s}\to \End(\mathfrak{v}\oplus\mathfrak{v})$ should be redefined correspondingly $)$.
\end{prop}

The relation between Lie algebras of Definition~\ref{def:pseudo} and admissible Clifford modules is summarized in the following proposition.

\begin{prop}\cite{Ciatti}
Let $\vv$ be a $\Cl_{r,s}$-module. Then $\n=\vv \oplus \mathbb R^{r,s}$ can be supplied with the structure of the pseudo $H$-type algebra if and only if there exists a scalar product $\langle \cdot \,, \cdot \rangle_{\vv}$ making the $\Cl_{r,s}$-module $\vv$ into an admissible Clifford module $(\vv,\langle \cdot \,, \cdot \rangle_{\vv})$.
The bracket $[ \cdot \,, \cdot ] \colon \vv \times \vv \to \mathbb R^{r,s}$ on $\n$ is given by~\eqref{def:J_Z} and the scalar product is $\langle \cdot \,, \cdot \rangle_{\n}:=\langle \cdot \,, \cdot \rangle_{\vv}+\langle \cdot \,, \cdot \rangle_{r,s}$. The decomposition $\n=\vv \oplus \mathbb R^{r,s}$ is orthogonal and $\mathbb R^{r,s}$ is the center of $\n$. 
\end{prop}

\begin{prop}\cite{Ciatti}
Let $\mathfrak{n}$ be a pseudo $H$-type algebra. Then the corresponding admissible $\Cl_{r,s}$-module $(\vv,\langle \cdot \,, \cdot \rangle_{\mathfrak{v}})$ is a neutral scalar product space for $s\geq1$, i.e. the signature of 
$\langle \cdot \,, \cdot \rangle_{\mathfrak{v}}$ is $(l,l)$, with $l \in \mathbb{N}$.
\end{prop}

%%%%%%%%%%%%%%%%%%%%%%%%%%%%%%%%%%%%%%%%%%%%%%%%

\subsection{Existence of the integral structure on pseudo $H$-type Lie algebras}

%%%%%%%%%%%%%%%%%%%%%%%%%%%%%%%%%%%%%%%%%%%%%%%%
In this subsection we state some necessary facts about the latest research on pseudo $H$-type Lie algebras based on the work~\cite{FurutaniIrina}. The principal result of~\cite{FurutaniIrina} states that pseudo $H$-type Lie algebras admit a special choice of basis giving integer structure constants.

\begin{theorem}\label{lattice}\cite{FurutaniIrina}
Let $\n=\big(\vv\oplus_{\bot}\z, [ \cdot \,, \cdot ] , \langle \cdot \,, \cdot \rangle_{\n} = \langle \cdot \,, \cdot \rangle_{\vv} + \langle \cdot \,, \cdot \rangle_{\z}\big)$ be a pseudo $H$-type Lie algebra. Then for any orthonormal basis $\{Z_1,\dotso,Z_n\}$ for $\z$ there is an orthonormal basis $\{v_1, \dotso ,v_m\}$ for $\vv$ such that $[v_{\alpha} \,, v_{\beta}]=\sum_{k=1}^n{A_{\alpha \beta}^kZ_k}$, where $A_{\alpha \beta}^k=0,\pm1$.
\end{theorem}  

\begin{coro}\label{lattice1}\cite{FurutaniIrina}
There exists an orthonormal basis $\mathcal B=\{v_1, \dotso ,v_m,Z_1,\dotso,Z_n\}$ for any pseudo $H$-type Lie algebra such that $[v_{\alpha} \,, v_{\beta}]=\pm Z_{k_{\alpha,\beta}}$ or $[v_{\alpha} \,, v_{\beta}]=0$. In particular, for every $Z_k$ and $v_{\alpha}$ in $\mathcal B$ there exists exactly one $\beta \in \{1, \dotso,m\}$ such that $[v_{\alpha} \,, v_{\beta}]=\pm Z_{k}$.
\end{coro}

\begin{defi}
We call an orthonormal basis $\{v_1, \dotso ,v_m,Z_1,\dotso,Z_n\}$ of a pseudo $H$-type Lie algebra with the form of Corollary~\ref{lattice1} an {\it integral basis}. The corresponding Clifford module $\vv=\spn\{v_1, \dotso ,v_m\}$ from Corollary~\ref{lattice1} is called an {\it integral module} and if it is of minimal possible dimension we call it {\it minimal integral module}.
\end{defi}

Let  $\{v_1, \dotso ,v_m,Z_1,\dotso,Z_n\}$ be an orthonormal basis of $\n$. Denote by $\epsilon^{\vv}_{\alpha}$
and $\epsilon^{\z}_{k}$ the indices corresponding to scalar product spaces $(\vv,\langle \cdot \,, \cdot \rangle_{\vv})$ and $(\z,\langle \cdot \,, \cdot \rangle_{\z})$.
The structure constants and the coefficients of the representation operator $J\colon \z\to\End(\vv)$ are
\begin{equation}\label{eq:AB}
[v_{\alpha} \,, v_{\beta} ] = \sum_{k=1}^n{A^k_{\alpha \beta}Z_k} \qquad \text{and} \qquad J_{Z_k}v_{\alpha}=\sum_{\beta=1}^m{B^k_{\alpha \beta} v_{\beta}}.
\end{equation}
Then we obtain the relation
\begin{eqnarray}\label{structure}
\epsilon^{\vv}_{\beta} B^{k}_{\alpha \beta} = \epsilon^{\z}_{k}A^k_{\alpha \beta}
\end{eqnarray}
from $\langle J_{Z_k}v_{\alpha} \,, v_{\beta} \rangle_{\vv} = \langle Z_k \,, [v_{\alpha} \,, v_{\beta}]\rangle_{\z}$ by \cite{FurutaniIrina}. Equality~\eqref{structure} allows to relate the structure constants $A^k_{\alpha \beta}$ of pseudo $H$-type Lie algebras and  coefficients $B^k_{\alpha \beta}$ of the representation operator. 
From now on, we denote the pseudo $H$-type Lie algebra induced by $\Cl_{r,s}$ by $\n_{r,s}=\vv_{r,s} \oplus \z_{r,s}$, where $\vv_{r,s}$ is the minimal admissible integral module $\vv_{r,s}$ of $\Cl_{r,s}$ and $\z_{r,s}= \mathbb{R}^{r,s}$ the generator space of the Clifford algebra $\Cl_{r,s}$ and the center of the Lie algebra $\n_{r,s}$. 
\end{section}

%%%%%%%%%%%%%%%%%%%%%%%%%%%%%%%%%%%%%%%%%%%%%%%%

\begin{section}{Necessary condition for isomorphisms of pseudo $H$-type Lie algebras}\label{sec:80_08}

%%%%%%%%%%%%%%%%%%%%%%%%%%%%%%%%%%%%%%%%%%%%%%%%

In the present section we identify the admissible module $\vv_{r,s}$, $s\neq 0$ with $\mathbb{R}^{l,l}$ equipped with the neutral scalar product $\langle x , y \rangle_{l,l}= \sum_{i=1}^l{x_iy_i}-\sum_{j=l+1}^{2l}{x_jy_j}$ for $x, y\in \mathbb{R}^{l,l}$. In the case $\vv_{r,0}$ we use the identification of $\vv_{r,0}$ with $\mathbb R^{2l}$ endowed with the inner product $\langle x, y \rangle_{2l}= \sum_{i=1}^{2l}{x_iy_i}$ for $x, y\in \mathbb{R}^{2l}$. Thus a pseudo $H$-type Lie algebra $\n_{r,s}$ is isometric to $\mathbb R^{l,l}\oplus\mathbb R^{r,s}$. Let $A\in GL(\mathbb R^{2l})$. We denote by $A^{\tau}$ the adjoint map with respect to the neutral scalar product $\langle \cdot \,, \cdot \rangle_{l,l}$
$$
\langle Aw\,, v \rangle_{l,l} = \langle w \,, A^{\tau}v\rangle_{l,l}.
$$
The same symbol we use to write the adjoint map 
$
\langle Aw\,, v \rangle_{2l,0} = \langle w \,, A^{\tau}v\rangle_{l,l}
$
with respect to scalar products $\langle \cdot \,, \cdot \rangle_{2l,0}$ and $\langle \cdot \,, \cdot \rangle_{l,l}$.

The adjoint map $C^{\tau}$ for the map $C\colon \mathbb R^{t,u}\to\mathbb R^{r,s}$ with $t+u=r+s$ with respect to corresponding scalar products is given by
$$
\langle C(Z)\,, \zeta \rangle_{r,s} = \langle Z \,, C^{\tau}(\zeta)\rangle_{t,u}. 
$$
Assume that two pseudo $H$-type Lie algebras $\n_{t,u}$ and $\n_{r,s}$ are isomorphic and $f\colon\n_{t,u}\to\n_{r,s}$ is an isomorphism. Then $t+u=r+s$ and since the center of $\n_{t,u}$ is mapped to the center of $\n_{r,s}$, the matrix of the map $f$ takes the form
\begin{equation}\label{eq:mf}
M_f=\begin{pmatrix} A & 0 \\ B & C \end{pmatrix},\quad A\in GL(\mathbb R^{2l}),\quad C\in GL(\mathbb R^{r+s})
\end{equation}
and $B$ is a $\big((r+s)\times 2l\big)$-matrix. Checking the commutation relations, we get $[Aw, Av]=C([w,v])$ for all $w, v \in \mathbb{R}^{2l}$.
With this notation we prove our first classification result.

\begin{lemma}\label{iso} Let $f\colon\n_{t,u}\to\n_{r,s}$ be a Lie algebra isomorphism represented by~\eqref{eq:mf}. Then the matrices $A$ and $C$ in~\eqref{eq:mf} satisfy 
\begin{equation}\label{con}
A^{\tau} \circ J_Z \circ A=J_{C^{\tau}(Z)},\quad \text{ for all } \quad Z \in \mathbb{R}^{r,s}.
\end{equation}
\end{lemma}

\begin{proof}
Formula~\eqref{con} follows from the calculations
\begin{eqnarray*}
\langle A^\tau \circ J_Z \circ Aw, v \rangle_{l,l}&=&\langle J_ZAw, Av \rangle_{l,l}=\langle Z, [Aw \,, Av] \rangle_{r,s} \\
&=& \langle Z, C([w,v]) \rangle_{r,s}=\langle C^{\tau}(Z), [w,v] \rangle_{t,u}=\langle J_{C^{\tau}(Z)}w , v \rangle_{l,l} 
\end{eqnarray*}
for all $w,v \in \mathbb{R}^{2l}$, $Z \in \mathbb{R}^{r,s}$ with $s \not =0$. If $s=0$ we use the same arguments applied for scalar products $\langle \cdot \,, \cdot \rangle_{2l,0}$ and $\langle \cdot \,, \cdot \rangle_{l,l}$.
\end{proof}

\begin{theorem}\label{r0}
A pseudo $H$-type algebra $\n_{r,0}$ can only be isomorphic to $\n_{0,r}$.
\end{theorem}
\begin{proof}
Formula~\eqref{con} implies that the action $J_{C^{\tau}(Z)}$ is singular if and only if $J_Z$ is singular for $Z \in \mathbb{R}^{r,s}$ and this happens only if $Z$ is a null vector in $\mathbb{R}^{r,s}$. Since the space $\mathbb{R}^{r,0}$ has no null vectors, the isomorphic Lie algebra can only have the center isomorphic to $\mathbb{R}^{0,r}$.
\end{proof}
\begin{theorem}\label{rssr}
A pseudo $H$-type algebra $\n_{r,s}$ can only be isomorphic to $\n_{s,r}$ for $r,s\neq 0$.
\end{theorem}
\begin{proof}
Let $Z_+, Z_- \in \mathbb{R}^{t,u}$ with $\langle Z_+ \,, Z_+ \rangle_{t,u}>0$ and $\langle Z_- \,, Z_- \rangle_{t,u}<0$. We consider the line segment $\gamma(t)=(1-t)Z_++tZ_-$ for $t \in [0,1]$. Then there exists $t_0 \in (0,1)$ such that $\langle \gamma(t_0), \gamma(t_0) \rangle_{t,u}=0$  by the continuity of the scalar product and as $\langle \gamma(0), \gamma(0) \rangle_{t,u}>0$, $\langle \gamma(1), \gamma(1) \rangle_{t,u}<0$. 

Assume that $\mathfrak B_{\mathbb{R}^{t,u}}=\{Z_1,\ldots,Z_{t+u}\}$ is an orthonormal basis in $\mathbb{R}^{t,u}$  such that $\langle Z_i, Z_j \rangle_{t,u}=\epsilon_{i}(t,u)\delta_{ij}$ and we denote 
$$\varphi_{ij}(t)=\langle (1-t)Z_i+tZ_j \,, (1-t)Z_i+tZ_j \rangle_{t,u} = (1-t)^2\langle Z_i,Z_i\rangle_{t,u}+t^2\langle Z_j,Z_j\rangle_{t,u},
$$ for $i\not=j$. Then
\begin{eqnarray}\label{>0}
\varphi_{ij}(t)>0\ \ \text{if}\ \ i,j=1,\ldots, t,\quad\text{and}\quad \varphi_{ij}(t)<0 \ \ \text{if}\  \ i,j=t+1,\ldots, t+u
\end{eqnarray}  
for all $t\in[0,1]$.
 
Let $f\colon\n_{r,s}\to\n_{t,u}$ be the isomorphism represented by~\eqref{eq:mf}. Consider the image $\{C^{\tau}(Z_1), \dotso,C^{\tau}(Z_{t+u})\} \subset \mathbb{R}^{r,s}$ under the map $C^{\tau}$ of the basis $\mathfrak B_{\mathbb{R}^{t,u}}$. First we note that $\langle C^{\tau}(Z_i)\,, C^{\tau}(Z_i) \rangle_{r,s}\not=0$ by Lemma~\ref{iso}.

We claim that for basis vectors $Z_i$, $i=1,\ldots, t$, one gets $\langle C^{\tau}(Z_i)\,, C^{\tau}(Z_i) \rangle_{r,s}>0$ or $\langle C^{\tau}(Z_i)\,, C^{\tau}(Z_i) \rangle_{r,s}<0$ for all indices $i=1,\ldots, t$, simultaneously. 
Indeed, assume that there are $C^{\tau}(Z_i)$ and $C^{\tau}(Z_j)$ for $i,j=1,\ldots,t$ such that products
$
\langle C^{\tau}(Z_{i})\,, C^{\tau}(Z_{i})  \rangle_{r,s}$ and $\langle C^{\tau}(Z_{j})\,, C^{\tau}(Z_{j})  \rangle_{r,s}
$
have opposite sign. Then there exists $t_0 \in (0,1)$ such that 
\begin{eqnarray*}
\langle (1-t_0)C^{\tau}(Z_i)+t_0C^{\tau}(Z_j)\,, (1-t_0)C^{\tau}(Z_i)+t_0C^{\tau}(Z_j) \rangle_{r,s} = 0,
\end{eqnarray*}
which implies that $J_{(1-t_0)Z_i+t_0Z_j}$ is singular by Lemma~\ref{iso}, which contradicts~\eqref{>0}. The same arguments are valid for the basis vectors $Z_i$ with $i=t+1,\ldots,t+u$. Thus we conclude that the scalar product $\langle \cdot\,,\cdot \rangle_{r,s}$ restricted to subspaces $\spn\{C^{\tau}(Z_1),\dotso,C^{\tau}(Z_{t})\}$ and $\spn\{C^{\tau}(Z_{t+1}),\dotso,C^{\tau}(Z_{t+u})\}$ is sign definite.
As $\{C^{\tau}(Z_1), \dotso,C^{\tau}(Z_{t+u})\}$ is a basis of $\mathbb{R}^{r,s}$ it follows that 
$$
r=t\ \ \text{ and }\ \ s=u, \qquad \text{ or }\qquad
r=u\ \ \text{ and }\ \ s=t. 
$$
This implies, that the only possible isomorphic pseudo $H$-type algebra for $\n_{r,s}$ is $\n_{s,r}$.
\end{proof}

\begin{theorem}\label{Id}
If $\n_{r,s}$ and $\n_{s,r}$, $r \not=s$, are isomorphic, then there exists a Lie algebra isomorphism $\varphi\colon \n_{r,s} \to \n_{s,r}$ given by the matrix $\begin{pmatrix} A & 0 \\ B & C \end{pmatrix}$ with $CC^{\tau}=-\Id_{\mathbb R^{s,r}}$. Moreover $C^{\tau}C=-\Id_{\mathbb R^{r,s}}$ and $C$, $C^{\tau}$ are anti-isometries.
\end{theorem}

\begin{proof}
Theorem~\ref{rssr} implies that for any $Z \in \mathbb R^{s,r}$: $\langle Z,Z \rangle_{s,r}=-\lambda \langle C^{\tau}(Z)\,, C^{\tau}(Z) \rangle_{r,s}$ for some $\lambda >0$. To determine $\lambda$ we pick up an arbitrary $Z \in \mathbb R^{s,r}$ and calculate
\begin{equation*}
\big(\det(A^{\tau}J_ZA)\big)^2
=
\big(\det(A^{\tau}A)\big)^2 \big(\langle Z\,, Z \rangle_{s,r}\big)^{2l}.
\end{equation*} 
On the other hand 
\begin{equation*}
\big(\det(A^{\tau}J_ZA)\big)^2
=
\big(\det(J_{C^{\tau}(Z)})\big)^2
=
\big(\langle C^{\tau}(Z)\,, C^{\tau}(Z) \rangle_{r,s}\big)^{2l}  ,
\end{equation*}
which is equivalent to 
$
|\det(A^{\tau}A)|^{1/l} \langle Z\,, Z \rangle_{s,r} = - \langle C^{\tau}(Z)\,, C^{\tau}(Z) \rangle_{r,s}=-\langle Z\,, CC^{\tau}(Z) \rangle_{s,r}$.
It follows that $CC^{\tau}=-|\det(A^{\tau}A)|^{\frac{1}{l}} \Id_{\mathbb R^{r,s}}$. 

If $\varphi\colon \n_{r,s} \to \n_{s,r}$ is a Lie algebra isomorphism, then $\tilde \varphi=\begin{pmatrix} \mu A & 0 \\ B & \mu ^2 C \end{pmatrix}$ for $\mu \not=0$ is also a Lie algebra isomorphism as 
\begin{eqnarray*}
\tilde \varphi([w,v]_{r,s})&=&(\mu^2 C)([w, v ]_{r,s}) = \mu^2 (C([w, v]_{r,s})=\mu^2([Aw, Av]_{s,r}) \\ &=& [\mu Aw, \mu Av]_{s,r} = [\tilde \varphi(w), \tilde \varphi(v)]_{s,r}
\end{eqnarray*}
for all $w,v \in \mathbb R^{l,l}$. Hence, without loss of generality, we can assume that $| \det(A^{\tau}A)|=1$, which implies that $CC^{\tau}= -\Id_{\mathbb R^{s,r}}$.

To show that $C$ and $C^{\tau}$ are anti-isometries we choose an arbitrary $Z \in \mathbb{R}^{s,r}$ and obtain
\begin{equation*}
\langle C^{\tau}(Z)\,, C^{\tau}(Z) \rangle_{r,s} = \langle CC^{\tau}(Z)\,, Z \rangle_{s,r} = - \langle Z\,, Z \rangle _{s,r}.
\end{equation*}
As $C^{\tau}$ is an isomorphism for any $Y \in \mathbb{R}^{r,s}$ there exists a unique $Z_Y \in \mathbb{R}^{s,r}$ such that $C^{\tau}(Z_Y)=Y$. It follows that for any $Y \in \mathbb{R}^{r,s}$ we have the equality
$
\langle Y\,, Y \rangle_{r,s}= \langle C^{\tau}(Z_Y)\,, C^{\tau}(Z_Y) \rangle_{r,s} = - \langle Z_Y\,, Z_Y \rangle_{s,r}$. 
Thus
\begin{eqnarray*}
\langle C^{\tau}C(Y)\,, Y \rangle_{r,s}&=&\langle C(Y)\,, C(Y) \rangle_{s,r}=\langle CC^{\tau}(Z_Y)\,, CC^{\tau}(Z_Y) \rangle_{s,r}= \langle Z_Y\,, Z_Y \rangle_{s,r} \\ &=&  -\langle Y\,, Y \rangle_{r,s}.
\end{eqnarray*}
Hence $C^{\tau}C=-\Id_{\mathbb{R}^{r,s}}$ and $C$ is an anti-isometry.
\end{proof}

\begin{theorem}\label{Idauto}
For any $\n_{r,s}$, $r\neq s$ there exists a Lie algebra automorphism $f \colon \n_{r,s} \to \n_{r,s}$ given by the matrix $\begin{pmatrix} A & 0 \\ B & C \end{pmatrix}$ with $CC^{\tau}=\Id_{\mathbb R^{r,s}}$. Moreover $C^{\tau}C=\Id_{\mathbb R^{r,s}}$ and $C$, $C^{\tau}$ are isometries.
\end{theorem}

\begin{proof}
The proof follows analogously to the proof of Theorem~\ref{Id} for $r\neq s$. For $r=s$ one of the possible automorphisms is the identity map.
\end{proof}

\begin{rem}
 If $r=s$, then there can be cases when there exist two automorphisms: one with $CC^{\tau}=\Id_{\mathbb R^{r,r}}$ and one with $CC^{\tau}=-\Id_{\mathbb R^{r,r}}$, see for instance Theorem~\ref{auto112244}. But there exist also cases where there are only automorphisms with $CC^{\tau}=\Id_{\mathbb R^{r,r}}$, see for instance Theorem~\ref{nonauto33}.
\end{rem}

%%%%%%%%%%%%%%%%%%%%%%%%%%%%%%%%%%%%%%%%%%%%%%%%

\section{Classification of $\n_{r,0}$ and $\n_{0,r}$}\label{sec:low_dim}

%%%%%%%%%%%%%%%%%%%%%%%%%%%%%%%%%%%%%%%%%%%%%%%%

Note that since the isomorphisms have to preserve the dimensions of Lie algebras and their centers, we only need to check algebras $\n_{r,s}$ and $\n_{t,u}$ with $r+s=t+u$.

%%%%%%%%%%%%%%%%%%%%%

\begin{subsection}{Classification of $\n_{r,0}$ and $\n_{0,r}$ with $r=1,2,4,8$}\label{sec:iso10}

%%%%%%%%%%%%%%%%%%%%%

\begin{defi}
Let $\{v_1, \dotso ,v_m,Z_1,\dotso,Z_n\}$ be an integral basis of a pseudo $H$-type Lie algebra $\n$ and $\{\tilde v_1, \dotso ,\tilde v_m, \tilde Z_1,\dotso,\tilde Z_n\}$ an integral basis of a pseudo $H$-type Lie algebra~$\tilde \n$. If a Lie algebra isomorphism $f\colon \n \to \tilde \n$ satisfies $f(v_i)=\tilde v_i$ and $f(Z_l)=\tilde Z _l$, then $f$ is called an integral isomorphism and we say that $\n$ is integral isomorphic to $\tilde \n$.
\end{defi}

\begin{notation}
For a given orthonormal basis $\{Z_1, \dotso,Z_{r+s}\}$ of the center $\z_{r,s}$ of the pseudo $H$-type algebra $\n_{r,s}$ we simplify the notation of the operator $J_{Z_i} \colon \vv_{r,s} \to \vv_{r,s}$ to $J_{Z_i}:=J_i$ for all $Z_i \in \{Z_1, \dotso,Z_{r+s}\}$.
\end{notation}

\begin{theorem}\label{iso10}
The Lie algebras $\n_{r,0}$ and $\n_{0,r}$ with $r=1,2,4,8$ are integral isomorphic.
\end{theorem}
 
\begin{proof} 
For all the cases we denote by $\{Z_1, \dotso,Z_{r+s}\}$ an orthonormal basis of the center $\z_{r,s}$ of the pseudo $H$-type algebra $\n_{r,s}$ with $\la Z_k\,, Z_k \ra_{\z_{r,s}}=\epsilon_k(r,s)$. Furthermore, all admissible modules are assumed to be minimal and all structural constants are obtained by the use of relation~\eqref{structure}. For more details how to obtain integral bases for general pseudo $H$-type algebras see~\cite{FurutaniIrina}.

{\it Isomorphism on $\n_{1,0}$ and $\n_{0,1}$.}
Let $(\vv_{1,0} , \langle \cdot \,, \cdot \rangle_{\vv_{1,0}})$ be an admissible module and  $w \in \vv_{1,0}$ be such that $\langle w\,, w  \rangle_{\vv_{1,0}} =1$. Then the basis 
$w_1=w$, $w_2=J_1w$
is integral and $ \langle w_i, w_i  \rangle_{\mathfrak{v}_{1,0}} =1$, $i=1,2$. 

Let $(\vv_{0,1} , \langle \cdot \,, \cdot \rangle_{\vv_{0,1}})$ be an admissible module and $\tilde w \in \vv_{0,1}$ such that $\langle \tilde w\,, \tilde w  \rangle_{\vv_{0,1}} =1$. Then the basis 
$\tilde w_1=\tilde w$, $\tilde w_2=\tilde J_1\tilde w$,
is integral and $ \langle \tilde w_i\,, \tilde w_i  \rangle_{\mathfrak{v}_{0,1}} =\epsilon_i(1,1)$, $i=1,2$.

Calculating the commutators with respect to both integral bases, presented in Table~\ref{10}, we conclude that they coincide. It follows that $\n_{1,0}$ is integral isomorphic to $\n_{0,1}$ under the isomorphism $\varphi_{1,0} \colon \n_{1,0} \to \n_{0,1}$ defined by
$
w_1 \mapsto \tilde w_1$, $w_2 \mapsto \tilde w_2$, $Z_1 \mapsto \tilde Z_1$.
{ \tiny
\begin{table}[h]
\caption{Commutation relations for $\n_{1,0}$ and $\n_{0,1}$}
\centering
\begin{tabular}{| c | c | c |} 
\hline
 $[ row \,, col. ]$  & $w_1$ & $w_2$ \\
\hline
$ w_1$ & $0$ & $Z_1$  \\
\hline
$w_2$ & $-Z_1$ & $0$ \\
\hline
\end{tabular}\label{10}
\end{table}   }

{\it Isomorphism of $\n_{2,0}$ and $\n_{0,2}$.}
In the admissible module $(\vv_{2,0} , \langle \cdot \,, \cdot \rangle_{\vv_{2,0}})$ we pick up $w \in \vv_{2,0}$ such that $\langle w\,, w  \rangle_{\vv_{2,0}} =1$. Then the basis 
$w_1=w$, $w_2=J_2J_1w$, $w_3=J_1w$, and $w_4=J_2w$
is integral and $ \langle w_i\,, w_i  \rangle_{\mathfrak{v}_{2,0}} =1$, $i=1,\dotso,4$. 

In the admissible module $(\vv_{0,2} , \langle \cdot \,, \cdot \rangle_{\vv_{0,2}})$ we choose $\tilde w \in \vv_{0,2}$ with $\langle \tilde w \,, \tilde w  \rangle_{\vv_{0,2}} =1$ and construct the integral basis 
$\tilde w_1=\tilde w$, $\tilde w_2=J_1J_2 \tilde w$, $\tilde w_3=J_1\tilde w$, and $\tilde w_4=J_2 \tilde w$
with $\langle \tilde w_i\,, \tilde w_i  \rangle_{\vv_{0,2}}=\epsilon_i(2,2)$. 

The commutation relations with respect to both bases are equal, see Table~\ref{02}, and this leads to the integral isomorphism. $\varphi_{2,0} \colon \n_{2,0} \to \n_{0,2}$ defined by 
$$
w_1 \mapsto \tilde w_1, \quad
w_2 \mapsto \tilde w_2 \quad w_3 \mapsto \tilde w_3 , \quad w_4 \mapsto \tilde w_4,\quad Z_1 \mapsto \tilde Z_1 , \quad Z_2 \mapsto \tilde Z_2.
$$
{ \tiny
\begin{table}[h]
\caption{Commutation relations on $\n_{2,0}$ and $\n_{0,2}$}
\centering
\begin{tabular}{| c | c | c | c | c |} 
\hline
 $[ row \,, col. ]$  & $w_1$ & $w_2$ & $w_3$ & $w_4$ \\
\hline
$w_1$ & $0$ & $0$ & $Z_1$ & $Z_2$ \\
\hline
$w_2$ & $0$ & $0$ & $-Z_2$ & $Z_1$ \\
\hline
$w_3$ & $-Z_1$ & $Z_2$ & $0$ & $0$ \\
\hline
$w_4$ & $-Z_2$ & $-Z_1$ & $0$ & $0$ \\
\hline
\end{tabular}\label{02}
\end{table} }

{\it Isomorphism of $\n_{4,0}$ and $\n_{0,4}$.}
Let $(\vv_{4,0},\langle\cdot\,,\cdot\rangle_{\vv_{4,0}})$ be an admissible module with $w\in \vv_{4,0}$ such that 
$J_1J_2J_3J_4w=w$ and $\langle w\,,w\rangle_{\vv_{4,0}}=1$. Then the basis 
$$
\begin{array}{lllllllllllll}
&w_1=w, \qquad &w_2=J_1J_2w, \qquad &w_3=J_1J_3w, \qquad &w_4=J_1J_4w, 
\\
&w_5=J_1w, \qquad &w_6=J_2w, \qquad &w_7=J_3w, \qquad &w_8=J_4w,
\end{array}
$$
with  $\langle w_i\,, w_i  \rangle_{\vv_{4,0}} =\epsilon_i(8,0)=1$
is integral, with commutation relations in Table~\ref{40}.  
{ \tiny
\begin{table}[h]
\caption{Commutation relations on  $\n_{4,0}$}
\centering
\begin{tabular}{| c | c | c | c | c | c | c | c | c |} 
\hline
 $[ row \,, col. ]$  & $w_1$ & $w_2$ & $w_3$ & $w_4$ & $w_5$ & $w_6$ & $w_7$ & $w_8$\\
\hline
$w_1$ & $0$ & $0$ & $0$ & $0$ & $Z_1$ & $Z_2$ & $Z_3$ & $Z_4$ \\
\hline
$w_2$ & $0$ & $0$ & $0$ & $0$ & $Z_2$ & $-Z_1$ & $-Z_4$ & $Z_3$ \\
\hline
$w_3$ & $0$ & $0$ & $0$ & $0$  & $Z_3$ & $Z_4$ & $-Z_1$ & $-Z_2$ \\
\hline
$w_4$ & $0$ & $0$ & $0$ & $0$  & $Z_4$ & $-Z_3$ & $Z_2$ & $-Z_1$\\
\hline
$w_5$ & $-Z_1$ & $-Z_2$ & $-Z_3$ & $-Z_4$ & $0$ & $0$ & $0$ & $0$\\
\hline
$w_6$ & $-Z_2$ & $Z_1$ & $-Z_4$ & $Z_3$ & $0$ & $0$ & $0$ & $0$\\
\hline
$w_7$ & $-Z_3$ & $Z_4$ & $Z_1$ & $-Z_2$ & $0$ & $0$ & $0$ & $0$\\
\hline
$w_8$ & $-Z_4$ & $-Z_3$ & $Z_2$ & $Z_1$ & $0$ & $0$ & $0$ & $0$\\
\hline  
\end{tabular} \label{40}
\end{table}  }

Let $(\vv_{0,4},\langle\cdot\,,\cdot\rangle_{\vv_{0,4}})$ be an admissible module and $w\in \vv_{0,4}$ be such that 
$J_1J_2J_3J_4w=w$ and $\langle \tilde w\,,\tilde w\rangle_{\vv_{0,4}}=1$. Then the basis 
$$
\begin{array}{lllllllllllll}
&\tilde w_1=\tilde w, \qquad &\tilde w_2=J_1J_2\tilde w, \qquad &\tilde w_3=J_1J_3\tilde w, \qquad &\tilde w_4=J_1J_4\tilde w, 
\\
&\tilde w_5=J_1\tilde w, \qquad &\tilde w_6=J_2\tilde w, \qquad &\tilde w_7=J_3\tilde w, \qquad &\tilde w_8=J_4\tilde w,
\end{array}
$$
with  $\langle \tilde w_i\,, \tilde w_i  \rangle_{\vv_{0,4}} =\epsilon_i(4,4)$
is integral with commutation relations listed in Table~\ref{04}.  
{ \tiny
\begin{table}[h]
\caption{Commutation relations on  $\n_{0,4}$}
\centering
\begin{tabular}{| c | c | c | c | c | c | c | c | c |} 
\hline
 $[ row \,, col. ]$  & $\tilde w_1$ & $\tilde w_2$ & $\tilde w_3$ & $\tilde w_4$ & $\tilde w_5$ & $\tilde w_6$ & $\tilde w_7$ & $\tilde w_8$\\
\hline
$\tilde w_1$ & $0$ & $0$ & $0$ & $0$ & $\tilde Z_1$ & $\tilde Z_2$ & $\tilde Z_3$ & $\tilde Z_4$ \\
\hline
$\tilde w_2$ & $0$ & $0$ & $0$ & $0$ & $-\tilde Z_2$ & $\tilde Z_1$ & $\tilde Z_4$ & $-\tilde Z_3$ \\
\hline
$\tilde w_3$ & $0$ & $0$ & $0$ & $0$  & $-\tilde Z_3$ & $-\tilde Z_4$ & $\tilde Z_1$ & $\tilde Z_2$ \\
\hline
$\tilde w_4$ & $0$ & $0$ & $0$ & $0$  & $-\tilde Z_4$ & $\tilde Z_3$ & $-\tilde Z_2$ & $\tilde Z_1$\\
\hline
$\tilde w_5$ & $-\tilde Z_1$ & $\tilde Z_2$ & $\tilde Z_3$ & $\tilde Z_4$ & $0$ & $0$ & $0$ & $0$\\
\hline
$\tilde w_6$ & $-\tilde Z_2$ & $-\tilde Z_1$ & $\tilde Z_4$ & $-\tilde Z_3$ & $0$ & $0$ & $0$ & $0$\\
\hline
$\tilde w_7$ & $-\tilde Z_3$ & $-\tilde Z_4$ & $-\tilde Z_1$ & $\tilde Z_2$ & $0$ & $0$ & $0$ & $0$\\
\hline
$\tilde w_8$ & $-\tilde Z_4$ & $\tilde Z_3$ & $-\tilde Z_2$ & $-\tilde Z_1$ & $0$ & $0$ & $0$ & $0$\\
\hline  
\end{tabular} \label{04}
\end{table}  }

We see from Tables~\ref{40} and~\ref{04} that the linear map $\varphi_{4,0} \colon \n_{4,0} \to \n_{0,4}$ defined by
$$
\begin{array}{ccclll}
 w_i & \mapsto&  \tilde w_i &\text{ if } \quad i=1,5,6,7,8, 
\\
 w_i & \mapsto&  -\tilde w_i   &\text{ if } \quad i=2,3,4, 
\\
 Z_k &\mapsto& \tilde{Z}_k &\text{ if } \quad k=1,2,3,4, 
\end{array}
$$
is an integral isomorphism.

{\it Isomorphism of $\n_{8,0}$ and $\n_{0,8}$.}
Let $\{Z_1,\ldots, Z_8\}$ be an orthonormal basis for $\mathbb R^{8,0}$.
Take a minimal admissible module $(\mathfrak{v}_{8,0} , \langle \cdot \,, \cdot \rangle_{\mathfrak{v}_{8,0}})$ and choose $w \in \mathfrak{v}_{8,0}$ with $\langle w \,, w  \rangle_{\mathfrak{v}_{8,0}} =1$ such that \begin{eqnarray*}
J_1J_2J_3J_4w=J_1J_2J_5J_6w=J_2J_3J_5J_7w=J_1J_2J_7J_8w=w.
\end{eqnarray*} 
A method of work~\cite{FurutaniIrina} shows that the basis 
\begin{equation}\label{onb801}
\begin{array}{llllll}
& u_1:=w, \quad &u_2:=J_1J_2w, \quad &u_3:=J_1J_3w, \quad &u_4:=J_1J_4w, 
\\ 
& u_5 :=J_1J_5w, \quad &u_6:=J_1J_6w, \quad &u_7:=J_1J_7w, \quad &u_8:=J_1J_8w,
\\
& u_9 :=J_1w, \quad &u_{10}:=J_2w, \quad &u_{11}:=J_3w, \quad &u_{12}:=J_4w, 
\\ 
& u_{13} :=J_5w, \quad &u_{14}:=J_6w, \quad &u_{15}:=J_7w, \quad &u_{16}:=J_8w,
\end{array}
\end{equation}
is orthonormal and satisfies
$ \langle u_i \,, u_i  \rangle_{\mathfrak{v}_{8,0}} =\epsilon_i(16,0)=1$. Thus the minimal admissible module $(\mathfrak{v}_{8,0} , \langle \cdot \,, \cdot \rangle_{\mathfrak{v}_{8,0}})$ receives the integral basis~\eqref{onb801}.

Let $\{\tilde Z_1,\ldots,\tilde Z_8\}$ be an orthonormal basis for $\mathbb R^{0,8}$. 
Given a minimal  admissible module $(\mathfrak{v}_{0,8} , \langle \cdot \,, \cdot \rangle_{\mathfrak{v}_{0,8}})$, we choose a vector $w^1 \in \mathfrak{v}_{0,8}$ with $\langle w^1 \,, w^1  \rangle_{\mathfrak{v}_{0,8}} =1$ such that
\begin{eqnarray*}
\tilde J_1\tilde J_2\tilde J_3\tilde J_4w^1=\tilde J_1\tilde J_2\tilde J_5\tilde J_6w^1=\tilde J_2\tilde J_3\tilde J_5\tilde J_7w^1=\tilde J_1\tilde J_2\tilde J_7\tilde J_8w^1=w^1.
\end{eqnarray*}
Then the orthonormal basis 
\begin{equation}\label{onb081}
\begin{array}{lllllll}
& v_1:=w^1, \quad &v_2:=\tilde{J}_1\tilde{J}_2w^1, \quad &v_3:=\tilde{J}_1\tilde{J}_3w^1, \quad &v_4:=\tilde{J}_1\tilde{J}_4w^1,
\\ 
& v_5:=\tilde{J}_1\tilde{J}_5w^1, \quad &v_6:=\tilde{J}_1\tilde{J}_6w^1, \quad &v_7:=\tilde{J}_1\tilde{J}_7w^1, \quad &v_8:=\tilde{J}_1\tilde{J}_8w^1,
\\
& v_9:=\tilde{J}_1w^1, \quad &v_{10}:=\tilde{J}_2w^1, \quad &v_{11}:=\tilde{J}_3w^1, \quad &v_{12}:=\tilde{J}_4w^1, 
\\ 
& v_{13}:=\tilde{J}_5w^1, \quad &v_{14}:=\tilde{J}_6w^1, \quad &v_{15}:=\tilde{J}_7w^1, \quad &v_{16}:=\tilde{J}_8w^1,
\end{array}
\end{equation}
with
$ \langle v_i \,, v_i  \rangle_{\mathfrak{v}_{0,8}} =\epsilon_i(8,8)$ is integral, see~\cite{FurutaniIrina}.
Tables~\ref{80} and~\ref{Cl08} in the Appendix show the non-vanishing commutation relations on the pseudo $H$-type Lie algebras $\n_{8,0}$ and $\n_{0,8}$. It allows us to construct the Lie algebra integral isomorphism
\begin{eqnarray}\label{isom8008varphi}
\varphi_{8,0} \colon \begin{cases} u_i \mapsto v_i & \mbox{ if }\quad i = 1,9,10,\dotso,16, 
\\ u_i\mapsto -v_i & \mbox{ if }\quad i = 2,\dotso,8 ,
\\
Z_k \mapsto \tilde{Z}_k & \mbox{ if }\quad k=1,\ldots, 8.
\end{cases}
\end{eqnarray}
\end{proof}
\end{subsection}

%%%%%%%%%%%%%%%%%%%%%%%%

\subsection{Structure constants for $\n_{r+8,s}$ and $\n_{r,s+8}$.}\label{sec:technical}

%%%%%%%%%%%%%%%%%%%%%%%%

This subsection is purely technical and auxiliary for the upcoming classification. 
A result of~\cite{FurutaniIrina} gives an integral basis which satisfies Theorem~\ref{lattice} for all admissible Clifford modules $\vv_{r,s}$. Furthermore, the authors proved that it is possible to obtain any minimal admissible integral module $\vv_{t,u}$ by taking the tensor product of minimal admissible integral $\vv_{r,s}$-modules $0\leq r,s \leq 8$ by the minimal admissible integral modules $\vv_{8,0}$, $\vv_{0,8}$ or $\vv_{4,4}$. 

\begin{prop}\label{ext08}\cite{FurutaniIrina}
Consider two minimal admissible integral modules $(\vv_{r,s} , \langle \cdot \,, \cdot \rangle_{\vv_{r,s}})$ and $(\vv_{0,8} , \langle \cdot \,, \cdot \rangle_{\vv_{0,8}})$, where the representations $J_{\bar{Z}_j}\colon \Cl_{0,8}\to \End(\vv_{0,8})$ permute the integral basis of $\vv_{0,8}$ up to sign for all orthonormal generators $\bar{Z}_j \in \mathbb{R}^{0,8}$. Then the scalar product space given by the tensor product $(\vv_{r,s} \otimes \vv_{0,8} , \langle \cdot \,, \cdot \rangle_{\vv_{r,s}}\cdot \langle \cdot \,, \cdot \rangle_{\vv_{0,8}})$ is a minimal admissible integral module $(\vv_{r,s+8}, \langle \cdot \,, \cdot \rangle_{\vv_{r,s+8}})$.
\end{prop}

\begin{rem}\label{rem:FurMar}
In Proposition~\ref{ext08} one can change the minimal admissible integral module $(\mathfrak{v}_{0,8} , \langle \cdot \,, \cdot \rangle_{\mathfrak{v}_{0,8}})$ to the minimal admissible integral module $(\vv_{8,0}, \langle \cdot \,, \cdot \rangle_{\mathfrak{v}_{8,0}})$ or $(\vv_{4,4}, \langle \cdot \,, \cdot \rangle_{\mathfrak{v}_{4,4}})$ and, taking the tensor product, to obtain the modules 
$$(\vv_{r+8,s}, \langle \cdot \,, \cdot \rangle_{\vv_{r+8,s}})= (
\vv_{r,s} \otimes \vv_{8,0} , \langle \cdot\,, \cdot \rangle_{\vv_{r,s}}\cdot \langle \cdot \,, \cdot \rangle_{\vv_{8,0}})$$  $$ (\vv_{r+4,s+4}, \langle \cdot \,, \cdot \rangle_{\vv_{r+4,s+4}})= (
\vv_{r,s} \otimes \vv_{4,4} , \langle \cdot \,, \cdot \rangle_{\vv_{r,s}}\cdot \langle \cdot \,, \cdot \rangle_{\vv_{4,4}})$$ respectively, which are minimal admissible integral. 
Details can be found in~\cite{FurutaniIrina}.
\end{rem}
A pseudo $H$-type Lie algebra $\n_{r,s}$ is called {\it extended} if its minimal admissible integral module $\vv_{r,s}$ was constructed as in Proposition~\ref{ext08} or in Remark~\ref{rem:FurMar}. The tensor products can be taken several times and with different spaces. Before we show how the structure constants for the Lie algebras $\n_{r+8,s}$, $\n_{r+4,s+4}$ and $\n_{r,s+8}$ depend on the structure constants of $\n_{r,s}$, $\n_{8,0}$, $\n_{0,8}$ and $\n_{4,4}$, we state the notation that will be used in the forthcoming sections. We write $\mathfrak B_V$ for an integral basis of the space $V$. Thus
$$
\mathfrak B_{\z_{r,s}}:=\{Z_1^{r,s},\ldots,Z_{r+s}^{r,s}\}, \quad \langle Z_k^{r,s}\,,Z^{r,s}_m\rangle_{\z_{r,s}}=\epsilon_k(r,s)\delta_{km},
$$
\begin{eqnarray*}
\mathfrak B_{\vv_{r,s}}:=\{w_1,\ldots,w_{2l}\}, \quad \langle w_i\,,w_j\rangle_{\vv_{r,s}}= \begin{cases} \epsilon_i(l,l)\delta_{ij} & \text{ for } s\not=0, \\
\delta_{ij} & \text{ for } s=0, \end{cases}
\end{eqnarray*}
and define $\mathfrak B_{\n_{r,s}}:=\mathfrak B_{\vv_{r,s}} \cup \mathfrak B_{\z_{r,s}}$.

We fix the letters $u$, $v$, and $y$ for the following bases given by~\eqref{onb801}, and a modification of~\eqref{isom8008varphi} and~\eqref{onb441}
\begin{eqnarray*}
\begin{array}{lllll}
\mathfrak B_{\vv_{8,0}}=\{u_1,\ldots,u_{16}\}, \quad &\mathfrak B_{\vv_{0,8}}=\{v_1,\ldots,v_{16}\},\quad &\mathfrak B_{\vv_{4,4}}=\{y_1,\ldots,y_{16}\}, \\ \mathfrak B_{\z_{8,0}}=\{Z_1^{8,0},\ldots,Z_{8}^{8,0}\}, \quad &\mathfrak B_{\z_{0,8}}=\{Z_1^{0,8},\ldots,Z_{8}^{0,8}\},\quad& \mathfrak B_{\z_{4,4}}=\{Z_1^{4,4},\ldots,Z_{8}^{4,4}\}
\end{array}
\end{eqnarray*}
with
\begin{eqnarray*}
\begin{array}{llll}
\langle u_i,u_j\rangle_{\vv_{8,0}}=\delta_{ij},\quad &
\langle v_i,v_j\rangle_{\vv_{0,8}}=\epsilon_i(8,8)\delta_{ij},\quad &
\langle y_i,y_j\rangle_{\vv_{4,4}}=\epsilon_i(8,8)\delta_{ij}, \\
\langle Z_k^{8,0},Z_m^{8,0}\rangle_{\z_{8,0}}=\delta_{km},\quad &
\langle Z_k^{0,8},Z_m^{0,8}\rangle_{\z_{0,8}}=-\delta_{km},\quad &
\langle Z_k^{4,4},Z_m^{4,4}\rangle_{\z_{4,4}}=\epsilon_k(4,4)\delta_{km}, \\
\end{array}
\end{eqnarray*}
such that $\mathfrak B_{\n_{8,0}}$ and $\mathfrak B_{\n_{0,8}}$ have the same structural constants by Theorem~\ref{iso10}, i.e.
\begin{eqnarray*}
\varphi_{8,0}(u_i)&=&v_i, \quad \text{ for all } i=1,\dotso,16, \\
\varphi_{8,0}(Z_k^{8,0})&=&Z_k^{0,8}, \quad \text{ for all } k=1,\dotso,8.
\end{eqnarray*}

If $\vv_{t,u}$ is obtained by taking the tensor product of $\vv_{r,s}$ by one of the $\vv_{8,0}$, $\vv_{0,8}$ or $\vv_{4,4}$, we write for the basis
$$
\mathfrak B_{\n_{t,u}}=\{w_i\otimes \alpha_j, Z_m^{t,u}\ \vert\ i=1,\ldots,2l,\ j=1,\ldots,16, \ m=1,\ldots, r+s+8\}, 
$$
where $w_i\in\mathfrak B_{\vv_{r,s}}$, the vectors $\alpha_j$ are from the corresponding integral bases $\mathfrak B_{\vv_{8,0}}$, $\mathfrak B_{\vv_{0,8}}$ or $\mathfrak B_{\vv_{4,4}}$, and $Z_m^{t,u}\in \mathfrak B_{\z_{r+p,s+q}}=\mathfrak B_{\z_{t,u}}$, with $(p,q)$ equal to one of the pairs $(8,0),(0,8),(4,4)$. For definiteness we preserve the order of the elements in $\mathfrak B_{\z_{t,u}}$ and we write first those which have positive squares of the scalar product and then those with negative squares of the scalar product. 

\begin{lemma}\label{depA80}
Let $\n_{r,s}=\vv_{r,s} \oplus \z_{r,s}$ be a pseudo $H$-type algebra and $A^m_{ij}$ the structure constants with respect to the integral basis $\mathfrak B_{\n_{r,s}}$. Let $\n_{8,0}=\vv_{8,0} \oplus \z_{8,0}$ has the integral basis $\mathfrak B_{\n_{8,0}}$ with the corresponding structure constants $\bar{A}^m_{ij}$. 
Then the Lie algebra $\n_{r+8,s}=(\vv_{r,s} \otimes \vv_{8,0}) \oplus \z_{r+8,s}$ has the following structure constants $\tilde{A}_{ij,pq}^m$ with respect to the integral basis $\mathfrak B_{\n_{r+8,s}}$.
\\
\noindent If $s=0$, then
\begin{eqnarray}\label{A80s0} 
\tilde{A}_{ij,pq}^m= \begin{cases} 
-A_{ip}^m & \text{ if }\ \  m =1,\dotso,r\qquad \qquad  \text{ and }\ \  j=q =1,\dotso,8, 
\\
A_{ip}^m & \text{ if }\ \  m =1,\dotso,r\qquad \qquad \text{ and }\ \ j=q =9,\dotso,16, 
\\
\bar{A}_{jq}^{m-r} & \text{ if }\ \ m =r+1,\dotso,r+8\ \ \text{ and }\ \ i=p=1,\dotso,2l, 
\\
0 & \text{ otherwise}. 
\end{cases}
\end{eqnarray}
If $s>0$, then
\begin{eqnarray}\label{A80s}
\tilde{A}_{ij,pq}^m= 
\begin{cases} 
-A_{ip}^m & \text{ if }\  m=1,\dotso,r\qquad\qquad\qquad\ \ \ \text{ and } j=q=1, \dotso, 8,
\\
-A_{ip}^{m-8} & \text{ if }\  m=r+8+1,\dotso,r+8+s \ \text{ and } j=q =1,\dotso,8,
\\
A_{ip}^m & \text{ if }\  m=1,\dotso,r\qquad\qquad\qquad\ \ \ \text{ and } j=q=9,\dotso,16, 
\\
A_{ip}^{m-8} & \text{ if }\  m=r+8+1,\dotso,r+8+s \ \text{ and } j=q=9,\dotso,16,
\\
\bar{A}_{jq}^{m-r} & \text{ if }\  m=r+1, \dotso, r+8\qquad\quad\ \ \text{ and } i=p=1, \dotso, l, 
\\
-\bar{A}_{jq}^{m-r} & \text{ if }\  m=r+1, \dotso, r+8\qquad\quad\ \ \text{ and } i=p=l+1, \dotso, 2l,
\\
0 &  \text{ otherwise}.
\end{cases}
\end{eqnarray}
\end{lemma}

\begin{proof}
We recall that the scalar product $\langle \cdot \,, \cdot \rangle_{\vv_{r+8,s}}$ of $\vv_{r+8,s}$ is given by the product $\langle \cdot \,, \cdot \rangle_{\vv_{r,s}} \cdot\langle \cdot \,, \cdot \rangle_{\vv_{8,0}}$. To shorten the notation we write 
$$
J_{Z_m^{r,s}}=J_{Z_m}\colon \vv_{r,s}\to\vv_{r,s},\quad 
J_{Z_m^{8,0}}=\bar{J}_{\bar{Z}_m}\colon \vv_{8,0}\to\vv_{8,0},\quad 
J_{Z_m^{r+8,s}}=\tilde{J}_{\tilde{Z}_m}\colon \vv_{r+8,s}\to\vv_{r+8,s}
$$
and 
the operator $E:=\bar{J}_{\bar{Z}_1}\dotsb\bar{J}_{\bar{Z}_8} \colon \vv_{8,0} \to \vv_{8,0}$ with the properties
\begin{equation*}
E^2=\Id_{\vv_{8,0}},\quad E\bar{J}_{\bar{Z}_j}=-\bar{J}_{\bar{Z}_j}E ,\ j=1, \dotso,8, \quad
\langle Eu, u^{*} \rangle_{\vv_{8,0}} = \langle u, Eu^{*} \rangle_{\vv_{8,0}}, \  u,u^{*} \in \vv_{8,0}.
\end{equation*}
It leads to the following equalities:
\begin{eqnarray}\label{eext80}
\langle \tilde{J}_{\tilde{Z}_m}\tilde{w}_{i,j}, \tilde{w}_{p,q} \rangle_{\vv_{r+8,s}} &=& \langle J_{Z_m}w_{i}, w_{p} \rangle_{\vv_{r,s}} \langle Eu_j, u_q \rangle_{\vv_{8,0}},\qquad m=1,\dotso ,r,
\\
\langle \tilde{J}_{\tilde{Z}_m}\tilde{w}_{i,j}, \tilde{w}_{p,q} \rangle_{\vv_{r+8,s}} &=& \langle w_{i}, w_{p} \rangle_{\vv_{r,s}} \langle \bar{J}_{\bar{Z}_{m-r}}u_j, u_q \rangle_{\vv_{8,0}},\qquad m=r+1,\dotso,r+8, \nonumber
\\
\langle \tilde{J}_{\tilde{Z}_m}\tilde{w}_{i,j}, \tilde{w}_{p,q} \rangle_{\vv_{r+8,s}} &=& \langle J_{Z_{m-8}}w_{i}, w_{p} \rangle_{\vv_{r,s}} \langle Eu_j, u_q \rangle_{\vv_{8,0}}, \  \quad m =r+9,\dotso,r+s+8,\nonumber
\end{eqnarray}
with the integral basis $\{\tilde{w}_{i,j}=w_i\otimes u_j,\tilde{Z}_m\vert i=1,\ldots,2l,j=1,\ldots,16, m=1,\ldots,r+s+8\}$ by~\cite{FurutaniIrina}. Similar equations for the cases $\n_{0,8}$ and $\n_{4,4}$ can be found in~\cite{FurutaniIrina}. 

Let $m=1,\ldots,r$ and note that the mapping $E$ acts on the integral basis $\mathfrak B_{\vv_{8,0}}$ by 
$Eu_j = -u_j\epsilon_j(8,8)$, which follows from the permutation Table~\ref{PermutE} in the Appendix. That leads to $\langle Eu_j, u_q \rangle_{\vv_{8,0}}=-\epsilon_j(8,8)\delta_{jq}$. Then the first equation in~\eqref{eext80} gives
$$
\tilde{B}_{ij,pq}^m\epsilon^{\vv_{r+8,s}}_{pq}=-B^{m}_{i,p}\epsilon^{\vv_{r,s}}_{p}\epsilon_j(8,8)\delta_{jq}
$$
by~\eqref{eq:AB}. Making use of formula~\eqref{structure} we obtain the following equations for structure constants
$$
\tilde{A}_{ij,pq}^m\big(\epsilon^{\vv_{r+8,s}}_{pq}\big)^2\epsilon^{\z_{r+8,s}}_m=-A^{m}_{i,p}\big(\epsilon^{\vv_{r,s}}_{p}\big)^2\epsilon^{\z_{r,s}}_m\epsilon_j(8,8)\delta_{jq}
$$
that yields to the first two lines in formula~\eqref{A80s0} and corresponding lines in~\eqref{A80s}. 
Arguing in a similar way for the rest of the formulas in~\eqref{eext80} we obtain
$$
\begin{array}{ccll}
 \epsilon^{\z_{r+8,s}}_{m}\tilde{A}_{ij,pq}^m
 &=&
 -A^{m}_{i,p}\epsilon^{\z_{r,s}}_m\epsilon_j(8,8)\delta_{jq}, &\qquad m =1,\dotso,r, 
 \\
 \epsilon^{\z_{r+8,s}}_{m}\tilde{A}_{ij,pq}^m
 &=&
 \bar{A}^{m-r}_{jq}\epsilon^{\z_{8,0}}_{m-r}\epsilon^{\vv_{r,s}}_{i}\delta_{ip} ,&\qquad m=r+1,\dotso,r+8, 
 \\
\epsilon^{\z_{r+8,s}}_{m}\tilde{A}_{ij,pq}^m
&=&
-A^{m-8}_{ip}\epsilon^{\z_{r,s}}_{m-8}\epsilon_j(8,8)\delta_{jq}, &\qquad m=r+9,\dotso,r+s+8.
\end{array}
$$
This implies~\eqref{A80s0} and~\eqref{A80s}.
\end{proof}

\begin{lemma}\label{depA08}
Let $\n_{r,s}=\vv_{r,s} \oplus \z_{r,s}$ be a pseudo $H$-type algebra with the structure constants $A^m_{ij}$ written with respect to $\mathfrak B_{\n_{r,s}}$ and $\n_{0,8}=\vv_{0,8} \oplus \z_{0,8}$ with the integral basis $\mathfrak B_{\n_{0,8}}$, and the corresponding structure constants $\bar{A}^m_{ij}$.
Then the Lie algebra $\n_{r,s+8}=(\vv_{r,s} \otimes \vv_{0,8}) \oplus \z_{r,s+8}$ has the following structure constants $\tilde{A}_{ij,pq}^m$ with respect to the integral basis $\mathfrak B_{\n_{r,s+8}}$.
\\

\noindent If $s=0$, then
\begin{eqnarray}\label{A08u0} 
\tilde{A}_{ij,pq}^m= 
\begin{cases} 
-A_{ip}^m & \text{ if }\ \ m =1,\dotso,r\qquad\qquad\ \text{ and }\ \  j=q=1, \dotso,16, 
\\
\bar{A}_{jq}^{m-r} & \text{ if }\ \ m=r+1,\dotso,r+8\ \ \ \text{ and }\ \ i=p=1,\dotso,2l, 
\\
0 & \text{ otherwise}. 
\end{cases}
\end{eqnarray}
If $s>0$, then
\begin{eqnarray}\label{A08u} 
\tilde{A}_{ij,pq}^m= 
\begin{cases} 
-A_{ip}^m & \text{ if }\  m=1,\dotso, r+s\qquad\qquad \qquad \text{and } j=q=1, \dotso, 16, 
\\
\bar{A}_{jq}^{m-r-s} & \text{ if }\  m=r+s+1,\dotso,r+s+8\  \text{ and } i=p=1,\dotso, l, 
\\
-\bar{A}_{jq}^{m-r-s} & \text{ if }\  m=r+s+1,\dotso,r+s+8\  \text{ and } i=p=l+1, \dotso,2l, 
\\
0 &  \text{ otherwise}. 
\end{cases}
\end{eqnarray}
\end{lemma}

The proof of Lemma~\ref{depA08} is analogous to the proof of Lemma~\ref{depA80}.

Before we present the structure constants for the Lie algebra $\n_{r+4,s+4}$ we write the integral basis for the pseudo $H$-type Lie algebra $\n_{4,4}$. 
\begin{equation}\label{onb441}
\begin{array}{llllll}
&y_1=w, \quad &y_2=J_1w, \quad &y_3=J_2w, \quad &y_4=J_3w, 
\\ 
&y_5=J_4w, \quad &y_6=J_1J_2w, \quad &y_7=J_1J_3w, \quad &y_8=J_1J_4w,
\\
&y_9=J_5w, \quad &y_{10}=J_6w, \quad &y_{11}=J_7w, \quad &y_{12}=J_8w, 
\\ 
&y_{13}=J_1J_5w, \quad &y_{14}=J_1J_6w, \quad &y_{15}=J_1J_7w, \quad &y_{16}=J_1J_8w,
\end{array}
\end{equation}
for $
J_1J_2J_3J_4w=J_1J_2J_5J_6w=J_2J_3J_5J_7w=J_1J_2J_7J_8w=w 
$ with
\begin{eqnarray*}
\langle w_i \,, w_i  \rangle_{\mathfrak{v}_{4,4}}=\epsilon_i(8,8), \qquad \la Z_k \,, Z_k \ra_{\z_{4,4}}=\epsilon_k(4,4).
\end{eqnarray*}

\begin{lemma}\label{depA44}
Let $\n_{r,s}=\vv_{r,s} \oplus \z_{r,s}$ has the structure constants $A^m_{ij}$ with respect to $\mathfrak B_{\n_{r,s}}$ and $\n_{4,4}=\vv_{4,4} \oplus \z_{4,4}$ has the structure constants $\bar{A}^m_{ij}$ with respect to the integral basis $\mathfrak B_{\n_{4,4}}$. 
Then the Lie algebra $\n_{r+4,s+4}=(\vv_{r,s} \otimes \vv_{4,4}) \oplus \z_{r+4,s+4}$ has the following structure constants $\tilde{A}_{ij,pq}^m$ with respect to the integral basis $\mathfrak B_{\n_{r+4,s+4}}$.
\\

\noindent If $s=0$, then
\begin{eqnarray}\label{A44u0} 
\tilde{A}_{ij,pq}^m= 
\begin{cases} 
A_{ip}^m & \text{ if }\  m=1,\dotso,r\qquad\qquad\ \text{ and }\ j=q=2,\dotso,5,13,\dotso,16, 
\\
-A_{ip}^m & \text{ if }\ m=1,\dotso,r\qquad\qquad\ \text{ and }\ j=q=1,6,\dotso,12, 
\\
\bar{A}_{jq}^{m-r} & \text{ if }\ m=r+1,\dotso,r+8\ \ \text{ and }\ i=p=1,\dotso,2l, 
\\
0 & \text{ otherwise}. 
\end{cases}
\end{eqnarray}
If $s>0$, then
\begin{eqnarray}\label{A44u} 
\tilde{A}_{ij,pq}^m= 
\begin{cases} 
A_{ip}^m & \text{ if } m=1,\dotso,r\qquad\qquad\ \ \ \text{ and }\ j=q=2,\dotso,5,13,\dotso,16, 
\\
-A_{ip}^m & \text{ if } m=1,\dotso,r\qquad\qquad\ \ \ \text{ and }\ j=q=1,6,\dotso,12, 
\\
A_{ip}^{m-8} & \text{ if } m=r+9,\dotso,r+s+8 \text{ and } j=q=2,\dotso,5,13,\dotso,16, 
\\
-A_{ip}^{m-8} & \text{ if } m=r+5,\dotso,r+s+4 \text{ and } j=q =1,6,\dotso,12,
\\
\bar{A}_{jq}^{m-r} & \text{ if } m=r+1,\dotso,r+4\quad\ \ \text{ and } i=p=1,\dotso,l, 
\\
\bar{A}_{jq}^{m-r} & \text{ if } m=r+5,\dotso,r+8 \qquad \text{ and } i=p=1,\dotso,l, 
\\
-\bar{A}_{jq}^{m-r} & \text{ if } m=r+1,\dotso,r+4\qquad \text{ and } i=p=l+1,\dotso, 2l, 
\\
-\bar{A}_{jq}^{m-r} & \text{ if } m=r+5,\dotso,r+8 \qquad \text{ and } i=p=l+1,\dotso,2l, 
\\
0 &  \text{ otherwise}. 
\end{cases}
\end{eqnarray}
\end{lemma}

The proof of Lemma~\ref{depA44} is analogous to the proof of Lemma~\ref{depA80}.

%%%%%%%%%%%%%%%%%%%%%%%%%%%%%%%%%

\subsection{Classification of $\n_{r,0}$ and $\n_{0,r}$ for $r>8$}

%%%%%%%%%%%%%%%%%%%%%%%%%%%%%%%%%%

We observe an interesting property of some of the $H$-type Lie algebras $\n_{r,0}$ and $\n_{0,r}$ that will be used in the proof of Theorem~\ref{th:r>8}.
\begin{defi} 
We say that an orthonormal basis $\{w_1, \dotso, w_{2l}, Z_1, \dotso,Z_n \}$ of a pseudo $H$-type Lie algebra is of block-type if $[w_i \,, w_j]=0$ for both indices $i,j=1, \dotso,l$ and $i,j=l+1, \dotso,2l$. 
We call a pseudo $H$-type Lie algebra of block-type if it has a block-type basis.
\end{defi}

\begin{lemma}\label{lem:block}
The pseudo $H$-type algebras $\n_{r,0}$ with $r\in \{0,1,2,4\}\mod(8) $  and $\n_{0,s}$ with $s \in \mathbb{N}$ are of block-type.
\end{lemma}

\begin{proof}
We prove by induction that $\n_{r,0}$ with $r \in \{0,1,2,4\}\mod(8)$ is of block-type. The base of induction follows from Tables~\ref{10},~\ref{02},~\ref{40},~\ref{80} of non-vanishing commutators on $\n_{1,0}$, $\n_{2,0}$, $\n_{4,0}$ and $\n_{8,0}$. For the induction step we assume that $\n_{r,0}$ with $r \in \{0,1,2,4\}\mod(8)$ has a block-type basis $\{w_1, \dotso,w_{2l},Z_1^{r,0}, \dotso, Z_r^{r,0}\}$ with $[w_i \,, w_j]=0$ when both indices $i,j=1, \dotso,l$ and $i,j=l+1, \dotso,2l$. By extension we construct the Lie algebra $\n_{r+8,0}$ of dimension $32l+r+8$ with the basis $\{w_1 \otimes u_1, \dotso, w_{2l} \otimes u_{16}, Z_1^{r+8,0} , \dotso , Z_{r+8}^{r+8,0}\}$. Equations~\eqref{eext80} imply that $[w_i \otimes u_j \,, w_p \otimes u_q]=0$ for the following cases:
\begin{itemize}
\item $i=p$ and both $j,q=1,\dotso,8$ and $j,q=9,\dotso,16$,
\item $j=q$ and both $i,p=1,\dotso,l$ and $i,p=l+1, \dotso,2l$,
\item $i \not=p$ and $j\not = q$ or $i=p$ and $j=q$.
\end{itemize}
We define a decomposition of the basis vectors of $\vv_{r+8,0}$ by
\begin{equation}\label{Ar0Br0}
\begin{array}{lcl}
\A^1_{r,0}&:=& \{ w_i \otimes u_j\ \vert\ i=1, \dotso,l ,\quad\qquad j=1, \dotso,8 \}, \\
\A^2_{r,0}&:=& \{ w_i \otimes u_j\ \vert\ i=l+1, \dotso,2l ,\quad j=9, \dotso,16 \}, \\
\B^1_{r,0}&:=& \{ w_i \otimes u_j\ \vert\ i=1, \dotso,l ,\quad\qquad j=9, \dotso,16 \}, \\
\B^2_{r,0}&:=& \{ w_i \otimes u_j\ \vert\ i=l+1, \dotso,2l ,\quad j=1, \dotso,8 \},\\
\A_{r,0}&:=&\A^1_{r,0} \cup \A^2_{r,0}, \qquad
\B_{r,0}:=\B^1_{r,0} \cup \B^2_{r,0}.
\end{array}
\end{equation}
It follows that for any $\tilde w, \tilde v \in A_{r,0}$ and for any $\tilde x , \tilde y \in B_{r,0}$ we obtain that
$[\tilde w \,, \tilde v]=0=[\tilde x \,, \tilde y]$.
As the cardinality of the basis of $\vv_{r+8,s}$ is $32l$ and the cardinality of each of the sets $A_{r,0}$ and $B_{r,0}$ is $16l$ we proved that $\n_{r,0}$ with $r\in \{0,1,2,4\}\mod8$ is of block-type.

Now we consider $H$-type Lie algebras $\n_{0,s}$. The space $\vv_{0,s}$ is neutral with an integral basis $\{w_1, \dotso, w_{2l}\}$ satisfying $\langle w_i \,, w_j \rangle_{l,l}=\epsilon_i (l,l) \delta_{ij}$. Let $\mathfrak B_{\z_{0,s}}=\{Z_1^{0,s},\ldots,Z_{s}^{0,s}\}$. The map $J_{Z_k^{0,s}}$ is an anti-isometry and permutes the basis $\{w_1, \dotso, w_{2l}\}$ for all $k = 1, \dotso,s$. It follows that $B_{ij}^k=0$ for $i,j=1, \dotso, l$ or $i,j=l+1, \dotso,2l$. Then by $\epsilon^{\vv}_{\beta} B^{k}_{\alpha \beta} = \epsilon^{\z}_{k}A^k_{\alpha \beta}$ we obtain that $A_{ij}^k=0$ when both indices $i,j=1, \dotso, l$ or $i,j=l+1, \dotso,2l$. Hence $\n_{0,s}$ is a block-type Lie algebra for all $s\in \mathbb{N}$.
\end{proof}

\begin{theorem}\label{th:r>8}
The Lie algebras $\n_{r,0}$ and $\n_{0,r}$ are integral isomorphic if and only if $r \in \{0,1,2,4\} \mod8$.
\end{theorem}
\begin{proof}
The $H$-type Lie algebras $\n_{r,0}$ are integral isomorphic to $\n_{0,r}$ for $r = 1,2,4,8$ by Theorem~\ref{iso10}. The Lie algebras $\n_{r,0}$ and $\n_{0,r}$ are non-isomorphic for $r =3,5,6,7$ by Theorem~\ref{noni} due to the different dimensions which are preserved under the extension process. Thus it remains to show that the Lie algebras $\n_{r,0}$ and $\n_{0,r}$ are integral isomorphic if $r  \in \{0,1,2,4\}\mod8$.

By induction we assume that $\n_{r,0}$ and $\n_{0,r}$ are integral isomorphic for $r  \in \{0,1,2,4\}\mod8$ with the integral block-type bases 
$$
\mathfrak B_{\n_{r,0}}=\{w_1, \dotso,w_{2l}, Z^{r,0}_1, \dotso ,Z^{r,0}_r \},\quad
\mathfrak B_{\n_{0,r}}=\{x_1, \dotso,x_{2l}, Z^{0,r}_1, \dotso ,Z^{0,r}_r \},
$$ 
with equal structure constants $A_{ij}^m$, i.e. $\varphi_{r,0}(w_i)=x_i$ for all $i=1, \dotso,2l$ and $\varphi_{r,0}(Z_k^{r,0})=Z_k^{0,r}$ for all $k=1, \dotso,r$. Furthermore, we recall that in the integral block-type bases $\mathfrak B_{\n_{8,0}}$ and $\mathfrak B_{\n_{0,8}}$ the structure constants denoted by $\bar{A}^m_{ij}$ are equal for both Lie algebras, i.e. $\varphi_{8,0}(u_i)=v_i$ for all $i=1, \dotso,16$ and $\varphi_{8,0}(Z_k^{8,0})=Z_k^{0,8}$ for all $k=1, \dotso,8$. 

We exploit Proposition~\ref{ext08} and Remark~\ref{rem:FurMar} and obtain the integral bases 
$$
\mathfrak B_{\n_{r+8,0}}=\{w_i \otimes u_j, Z_m^{r+8,0}\},\quad
\mathfrak B_{\n_{0,r+8}}=\{x_i \otimes v_j,  Z_m^{0,r+8}\}.
$$

We define the bijective linear map $\varphi_{r+8,0} \colon \n_{r+8,0} \to \n_{0,r+8}$ by
\begin{equation*}
\begin{array}{clccllc}
w_i \otimes u_j &\mapsto & x_i \otimes v_j   &\text{ if } &\quad i \in \{1, \dotso,l\}, & j \in \{1,\dotso,16\}, 
\\
\nonumber w_i \otimes u_j &\mapsto & x_i \otimes v_j   &\text{ if }& \quad i \in \{l+1, \dotso,2l\},  & j \in \{1,\dotso,8\}, 
\\
\nonumber w_i \otimes u_j &\mapsto& -x_i \otimes v_j   &\text{ if }& \quad i \in \{l+1, \dotso,2l\}, & j \in \{9,\dotso,16\}, \\
\nonumber Z^{r+8,0}_m &\mapsto& Z^{0,r+8}_m &\text{ if }& \quad m \in \{1, \dotso, r+8\}.
\end{array}
\end{equation*}

It remains to prove that $\varphi_{r+8,0}$ is a Lie algebra isomorphism, i.e. $\varphi_{r+8,0}([w_i \otimes u_j \,, w_p \otimes u_q])=[\varphi_{r+8,0}(w_i \otimes u_j) \,, \varphi_{r+8,0}(w_p \otimes u_q)]$. We know that the structural constants $\tilde A_{ij,pq}^m$ of $[w_i \otimes u_j \,, w_p \otimes u_q]$ are given by formula~\eqref{A80s0} and that the structural constants $C_{ij,pq}^m$ for $[x_i \otimes v_j \,, x_p \otimes v_q]$ are given by formula~\eqref{A08u}, where we have to put the index $r$ instead of $r+s$. It follows that if $i\not=p$ and $j\not=q$ or $i=p$ and $j=q$ the commutators vanish:
$$\varphi_{r+8,0}([w_i \otimes u_j \,, w_p \otimes u_q])=\varphi_{r+8,0}(0)=0=\pm [x_i \otimes v_j \,, x_p \otimes v_q].$$
Let us consider the case $i=p$ and $j\not=q$.

$\bullet$ if $i=p=1, \dotso,l$, then: 
$$
\varphi_{r+8,0}([w_i \otimes u_j \,, w_i \otimes u_q])=\bar A _{jq}^{m-r}\varphi_{r+8,0}(Z^{r+8,0}_m)=\bar A _{jq}^{m-r}Z^{0,r+8}_m,
$$
$$
[\varphi_{r+8,0}(w_i \otimes u_j) \,, \varphi_{r+8,0}(w_i \otimes u_q)]=[x_i \otimes v_j \,, x_i \otimes v_q] =\bar A _{jq}^{m-r}Z^{0,r+8}_m
$$
with $m=r+1, \dotso,r+8$ by formulas~\eqref{A80s0} and~\eqref{A08u}.

$\bullet$ if $i=p=l+1, \dotso,2l$, then we use Lemma~\ref{lem:block}.
$$
\varphi_{r+8,0}([w_i \otimes u_j \,, w_i \otimes u_q])=\bar A _{jq}^{m-r}\varphi_{r+8,0}(Z^{r+8,0}_m)=\bar A _{jq}^{m-r}Z^{0,r+8}_m,
$$
\begin{eqnarray*}
&\hskip-0.0cm \ &[\varphi_{r+8,0}(w_i \otimes u_j) \,, \varphi_{r+8,0}(w_i \otimes u_q)]
\\
&= &
\begin{cases}
[x_i \otimes v_j \,, x_i \otimes v_q]\quad & \text{if}\quad  j,q=1, \dotso,8\ \ \text{or}\ \  j,q=9,\dotso,16,
\\
-[x_i \otimes v_j \,, x_i \otimes v_q] \quad & \text{otherwise},
\end{cases}
\\
&= &
\begin{cases}
-\bar A _{jq}^{m-r}Z_m^{0,r+8}\qquad\qquad\quad & \text{if}\quad  j,q=1, \dotso,8\ \ \text{or}\ \  j,q=9,\dotso,16,
\\
\bar A _{jq}^{m-r} Z_m^{0,r+8}\quad & \text{otherwise},
\end{cases}
\end{eqnarray*}
with $m=r+1, \dotso,r+8$ by formulas~\eqref{A80s0},~\eqref{A08u}, and the definition of $\varphi_{r+8,0}$. We observe that $\bar A _{jq}^{m-r}=0$ when for both indices $j,q$ simultaneously either $j,q=1, \dotso,8$ or $j,q=9,\dotso,16$, since the Lie algebras $\n_{8,0}$ and $\n_{0,8}$ are of block type, see Table~\ref{80}. We see that the map $\varphi_{r+8,0}$ satisfies the Lie algebra isomorphism properties in this case.

We turn to consider the case $i\not=p$ and $j=q$

$\bullet$ if $j=q=1, \dotso,8$, then
\begin{eqnarray*}
\varphi_{r+8,0}([w_i \otimes u_j \,, w_p \otimes u_j])&=&-A _{ip}^{m}\varphi_{r+8,0}(Z^{r+8,0}_m)=- A _{ip}^{m}Z^{0,r+8}_m ,\\
{}[\varphi_{r+8,0}(w_i \otimes u_j) \,, \varphi_{r+8,0}(w_p \otimes u_j)]&=&[x_i \otimes v_j \,, x_p \otimes v_j]=- A _{ip}^{m}Z^{0,r+8}_m,
\end{eqnarray*}
with $m=1, \dotso,r$ by formulas~\eqref{A80s0} and~\eqref{A08u}.

$\bullet$ if $j=q=9, \dotso,16$, then we use the block form of Lie algebras $\n_{r,0}$ and $\n_{0,r}$. We calculate as above
$$
\varphi_{r+8,0}([w_i \otimes u_j, w_p \otimes u_j])=A _{ip}^{m}\varphi_{r+8,0}(Z^{r+8,0}_m) =A _{ip}^{m}Z^{0,r+8}_m.
$$
On the other side
\begin{eqnarray*}
&\hskip-0.3cm \ &[\varphi_{r+8,0}(w_i \otimes u_j) \,, \varphi_{r+8,0}(w_p \otimes u_j)]
\\
&= &
\begin{cases}
[x_i \otimes v_j \,, x_p \otimes v_q]\quad & \text{if}\quad  i,p=1,\dotso,l\ \ \text{or}\ \  i,p=l+1,\dotso,2l,
\\
-[x_i \otimes v_j \,, x_p \otimes v_q] \quad & \text{otherwise},
\end{cases}
\\
&= &
\begin{cases}
-A _{ip}^{m}Z^{0,r+8}_m=0\qquad\qquad & \text{if}\quad  i,p=1,\dotso,l\ \ \text{or}\ \  i,p=l+1,\dotso,2l,
\\
 A _{ip}^{m}Z^{0,r+8}_m\quad & \text{otherwise},
\end{cases}
\end{eqnarray*}
with $m=1, \dotso,r$ by formulas~\eqref{A80s0} and~\eqref{A08u}.
This finishes the proof of the theorem.
\end{proof} 

\begin{theorem}\label{noni} 
The pseudo $H$-type Lie algebra $\n_{r+8t,0}$ is not isomorphic to $\n_{0,r+8t}$ for $r=3,5,6,7$ and a non-negative integer $t$.
\end{theorem}
\begin{proof} We prove the theorem by counting the dimensions of the Lie algebras. The dimensions of the minimal admissible modules are given by
\begin{eqnarray*}
\dim(\vv_{3+8t,0})&=&4\cdot16^t\not = 8\cdot16^t=\dim(\vv_{0,3+8t}), 
\\
\dim(\vv_{r+8t,0})&=&8\cdot16^t\not = 16\cdot16^t=\dim(\vv_{0,r+8t}),  \quad \text{ for } r=5,6,7.
\end{eqnarray*}
\end{proof}
\end{section}

%%%%%%%%%%%%%%%%%%%%%%%%%%%%%%%%%%%%%%

\section{Isomorphism of Lie algebras $\n_{r,s}$ with $r,s\not=0$}\label{New}

%%%%%%%%%%%%%%%%%%%%%%%%%%%%%%%%%%%%%%

In this section we show, making use of the ideas developed in the previous section, that the Bott-periodicity is inherited in isomorphism properties of Lie algebras of block type. 

%%%%%%%%%%%%%%%%%%%%%%%%%%%%%%%%%%%%%%

\subsection{Decompositions of bases $r,s\not=0$}\label{sub:New}

%%%%%%%%%%%%%%%%%%%%%%%%%%%%%%%%%%%%%%

We recall the notations of the bases $\mathfrak B_{\n_{r,s}}$ and state the result that extends the notion of the block type algebras.

\begin{lemma}
Let us assume that the integral basis $\mathfrak B_{\vv_{r,s}}$, $r,s\not=0$, for the pseudo $H$-type Lie algebra $\n_{r,s}$ with $\dim(\vv_{r,s})=2l$ satisfies the following decomposition 
\begin{equation}\label{BD}
\begin{split}
& \mathfrak B_{\vv_{r,s}}= \A_{r,s}\cup  \B_{r,s},\quad \card(\A_{r,s})=\card(\B_{r,s})=l,
\\
& [\A_{r,s},\A_{r,s}]=[\B_{r,s},\B_{r,s}]=0,
\\
& \A_{r,s}=\A_{r,s}^+\cup \A_{r,s}^-,\quad \B_{r,s}=\B_{r,s}^+\cup \B_{r,s}^-,\quad  \card(\A_{r,s}^{\pm})=\card(\B_{r,s}^{\pm})=\frac{l}{2},
\\ 
&\text{where}\quad
\la w_i,w_i \ra_{\vv_{r,s}}=
\begin{cases}
1\quad&\text{if}\quad w_i\in\A_{r,s}^+\cup \B_{r,s}^+,
\\
-1\quad&\text{if}\quad w_i\in\A_{r,s}^-\cup \B_{r,s}^-.
\end{cases}
\end{split}
\end{equation}
Then the extended pseudo $H$-type Lie algebras $\n_{r+8,s}$, $\n_{r,s+8}$, and $\n_{r+4,s+4}$ admit a decomposition of type~\eqref{BD}.
\end{lemma}

\begin{proof}
Define the following sets
\begin{equation}\label{b0880}
\begin{split}
\A_{8,0}=\{u_1,\ldots,u_8\},\quad \B_{8,0}=\{u_9,\ldots,u_{16}\},\quad\text{for}\quad u_i\in\mathfrak B_{\vv_{8,0}}
\\
\A_{0,8}=\{v_1,\ldots,v_8\},\quad \B_{0,8}=\{v_9,\ldots,v_{16}\},\quad\text{for}\quad v_i\in\mathfrak B_{\vv_{0,8}}
\end{split}
\end{equation}
\begin{equation}\label{dec44}
\begin{split}
\A_{4,4}^+=\{y_1,y_6,y_7,y_8\},\quad \A^-_{4,4}= \{y_{13},y_{14},y_{15},y_{16}\}, 
\\
\B_{4,4}^+=\{y_2,y_3,y_4,y_5\},\quad \B^-_{4,4}=\{y_9,y_{10},y_{11},y_{12}\}. 
\end{split}
\end{equation}
for $y_i\in \mathfrak B_{\vv_{4,4}}$ given by~\eqref{onb441}.

Now, making use of $\A_{r,s}$, $\B_{r,s}$ and~\eqref{b0880},~\eqref{dec44}, we define the decompositions for the bases of the extended algebras. 
\begin{eqnarray*}
\A_{r+8,s}^+&=& \A_{r,s}^+ \otimes \A_{8,0}  \cup \B_{r,s}^+ \otimes \B_{8,0}, \qquad
\A_{r+8,s}^-= \A_{r,s}^- \otimes \A_{8,0} \cup \B_{r,s}^- \otimes \B_{8,0},\\
\B_{r+8,s}^+&=&\A_{r,s}^+ \otimes \B_{8,0} \cup \B_{r,s}^+ \otimes \A_{8,0},\qquad
\B_{r+8,s}^-=\A_{r,s}^- \otimes \B_{8,0} \cup \B_{r,s}^- \otimes \A_{8,0}, 
\end{eqnarray*}
\begin{eqnarray*}
\A_{r,s+8}^+&=& \A_{r,s}^+ \otimes \A_{0,8}  \cup \B_{r,s}^- \otimes \B_{0,8}, \qquad
\A_{r,s+8}^-= \A_{r,s}^- \otimes \A_{0,8} \cup \B_{r,s}^+ \otimes \B_{0,8},\\
\B_{r,s+8}^+&=&\A_{r,s}^- \otimes \B_{0,8} \cup \B_{r,s}^+ \otimes \A_{0,8},\qquad
\B_{r,s+8}^-=\A_{r,s}^+ \otimes \B_{0,8} \cup \B_{r,s}^- \otimes \A_{0,8},
\end{eqnarray*}
\begin{eqnarray*}
\A_{r+4,s+4}^+&=&\B_{r,s}^+ \otimes \B_{4,4}^+ \cup \B_{r,s}^- \otimes \B_{4,4}^-\cup \A_{r,s}^+ \otimes \A_{4,4}^+ \cup \A_{r,s}^- \otimes \A^-_{4,4},\\
\A_{r+4,s+4}^-&=&\B_{r,s}^- \otimes \B_{4,4}^+ \cup \B_{r,s}^+ \otimes \B_{4,4}^- \cup \A_{r,s}^- \otimes \A_{4,4}^+ \cup \A_{r,s}^+ \otimes \A_{4,4}^-,\\
\B_{r+4,s+4}^+&=&\A_{r,s}^+ \otimes \B_{4,4}^+ \cup \A_{r,s}^- \otimes \B_{4,4}^- \cup \B^+_{r,s} \otimes \A_{4,4}^+ \cup \B_{r,s}^- \otimes \A_{4,4}^-,\\
\B_{r+4,s+4}^-&=&\A_{r,s}^+ \otimes \B_{4,4}^- \cup \A_{r,s}^- \otimes \B_{4,4}^+\cup \B^-_{r,s} \otimes \A_{4,4}^+ \cup \B_{r,s}^+ \otimes \A_{4,4}^- .
\end{eqnarray*}
All the necessary properties follows directly from the definition of the basis for the extended Lie algebras and Lemmas~\ref{depA80}, ~\ref{depA08}, and~\ref{depA44}. We illustrate only the proof of the following property
$
[\A_{r+8,s},\A_{r+8,s}]=0
$, considering several cases. To show that $[\A_{r,s}^+\otimes \A_{8,0},\A_{r,s}^+\otimes\A_{8,0}]=0$ we choose $w_i,w_j\in\A_{r,s}^+$ and $u_{p},u_q\in\A_{8,0}$ Then
$$
[w_i\otimes u_p,w_j\otimes u_q]=
\begin{cases}
0\quad&\text{if}\quad i\neq j,\ p\neq q,\ \text{or}\ i=j,\ p=q,
\\
[u_p,u_q]=0&\text{if}\quad i= j,\ p\neq q,\ \text{since}\ [u_p,u_q]\in[\A_{8,0},\A_{8,0}]=0,
\\
-[w_i,w_j]=0&\text{if}\quad i\neq j,\ p= q,\ \text{since}\ [w_i,w_j]\in[\A_{r,s}^+,\A_{r,s}^+]=0.
\end{cases}
$$
Analogously we show
$$
[\A_{r,s}^{\pm}\otimes \A_{8,0},\A_{r,s}^{\pm}\otimes\A_{8,0}]=[\B_{r,s}^{\pm}\otimes \B_{8,0},\B_{r,s}^{\pm}\otimes\B_{8,0}]=0
$$ for any combinations of $+$ and $-$. Any term of the type $[\A_{r,s}^{\pm}\otimes \A_{8,0},\B_{r,s}^{\pm}\otimes\B_{8,0}]$ vanishes since  $\A_{r,s}^{\pm}\cap\B_{r,s}^{\pm}=\emptyset$ and $\A_{8,0}\cap\B_{8,0}=\emptyset$ and as if both $i\neq j$, $p\neq q$ we obtain that $[w_i\otimes u_p,w_j\otimes u_q]=0$ for any $w_i\in\A_{r,s}^{\pm}$, $w_j\in \B_{r,s}^{\pm}$, $u_p\in \A_{8,0}$, $u_q\in \B_{8,0}$.
\end{proof}

In the following lemma we present a list of pseudo $H$-type Lie algebras $\n_{r,s}$ satisfying~\eqref{BD}, which can be used as a base for the successive extensions. 

\begin{lemma}
The pseudo $H$-type Lie algebras $\n_{r,8}$, $\n_{8,r}$, $\n_{r+4,4}$, $\n_{4,r+4}$ for $r\in\{1,2,4,8\}\mod 8$ and $\n_{11}$, $\n_{2,2}$, $\n_{4,4}$ admit decomposition~\eqref{BD} of their bases.
\end{lemma}
\begin{proof}
{\sc Decompositions for $\n_{r,8}$, $\n_{8,r}$, $r\in \{1,2,4,8\}\mod 8$.}
\begin{eqnarray*}
\A^+_{r,8}&=& \{ w_i \otimes v_j\ \vert\ i=1, \dotso,l ,\quad\qquad j=1, \dotso,8 \}, \\
\A^-_{r,8}&=& \{ w_i \otimes v_j\ \vert\ i=l+1, \dotso,2l ,\quad j=9, \dotso,16 \}, \\
\B^-_{r,8}&=& \{ w_i \otimes v_j\ \vert\ i=1, \dotso,l ,\quad\qquad j=9, \dotso,16 \}, \\
\B^+_{r,8}&=& \{ w_i \otimes v_j\ \vert\ i=l+1, \dotso,2l ,\quad j=1, \dotso,8 \},
\end{eqnarray*}
for $w_i\in\mathfrak B_{\vv_{r,0}}$, $v_j\in\mathfrak B_{\vv_{0,8}}$.
\begin{eqnarray*}
\A^+_{8,r}&=& \{ w_i \otimes u_j\ \vert\ i=1, \dotso,l ,\quad\qquad j=1, \dotso,8 \}, \\
\A^-_{8,r}&=& \{ w_i \otimes u_j\ \vert\ i=l+1, \dotso,2l ,\quad j=9, \dotso,16 \}, \\
\B^+_{8,r}&=& \{ w_i \otimes u_j\ \vert\ i=1, \dotso,l ,\quad\qquad j=9, \dotso,16 \}, \\
\B^-_{8,r}&=& \{ w_i \otimes u_j\ \vert\ i=l+1, \dotso,2l ,\quad j=1, \dotso,8 \},
\end{eqnarray*}
for $w_i\in\mathfrak B_{\vv_{0,r}}$, $u_j\in\mathfrak B_{\vv_{8,0}}$.
For the proof we use Tables~\eqref{10}-\eqref{04}, Tables~\eqref{80}, \eqref{Cl08} and the block structure of the corresponding algebras. 
\\

{\sc Decompositions for $\n_{r+4,4}$, $\n_{4,r+4}$ for $r\in\{1,2,4,8\}\mod 8$.}
Recall decompositions~\eqref{Ar0Br0} and~\eqref{dec44} and define
\begin{equation*}
\begin{array}{lcllcl}
\A_{r+4,4}^+&=&\B_{r,0} \otimes \B_{4,4}^+  \cup \A_{r,0} \otimes \A_{4,4}^+,\qquad
\A_{r+4,4}^-&=&\B_{r,0} \otimes \B_{4,4}^- \cup \A_{r,0} \otimes \A_{4,4}^-,\\
\B_{r+4,4}^+&=&\B_{r,0} \otimes \A_{4,4}^+ \cup \A_{r,0} \otimes \B_{4,4}^+ ,\qquad
\B_{r+4,4}^-&=&\B_{r,0} \otimes \A_{4,4}^- \cup \A_{r,0} \otimes \B_{4,4}^-,
\\
\A_{4,r+4}^+&=&\B_{0,r} \otimes \B_{4,4}^- \cup \A_{0,r} \otimes \A_{4,4}^+,\qquad
\A_{4,r+4}^-&=&\B_{0,r} \otimes \B_{4,4}^+ \cup \A_{0,r}\otimes \A_{4,4}^-,\\
\B_{4,r+4}^+&=&\B_{0,r} \otimes \A_{4,4}^- \cup \A_{0,r} \otimes \B_{4,4}^+ ,\qquad
\B_{4,r+4}^-&=&\B_{0,r} \otimes \A_{4,4}^+ \cup \A_{0,r} \otimes \B_{4,4}^-.
\end{array}
\end{equation*}
For the proof we use Tables~\eqref{10}-\eqref{04}, Tables~\eqref{80}, \eqref{Cl08}, \eqref{n44}, the block structure of the corresponding algebras, and Lemma~\ref{depA44}. 
\\

{\sc Decompositions for $\n_{1,1}$, $\n_{2,2}$, and $\n_{4,4}$.}

We define an orthonormal basis of $\n_{1,1}$ by
\begin{equation}\label{b11}
\mathfrak B_{\vv_{1,1}}=\{w_1=w, \ w_2=J_1w, \ w_3=J_2w, \ w_4=J_2J_1w\}, \quad \mathfrak B_{\z_{1,1}}=\{Z_1, Z_2\},
\end{equation}
with
$
\langle w_i, w_i  \rangle_{\mathfrak{v}_{1,1}}=\epsilon_i(2,2)$, $\la Z_k, Z_k \ra_{\z_{1,1}}=\epsilon_k(1,1)$. The commutators are in Table~\ref{n11}.

\begin{table}[h]
\caption{Commutation relations on $\n_{1,1}$}
\centering
\begin{tabular}{| c | c | c | c | c |} 
\hline
 $[ row \,, col. ]$  & $w_1$ & $w_4$ & $w_2$ & $w_3$ \\
\hline
$w_1$ & $0$ & $0$ & $Z_1$ & $Z_2$ \\
\hline
$w_4$ & $0$ & $0$ & $-Z_2$ & $-Z_1$ \\
\hline
$w_2$ & $-Z_1$ & $Z_2$ & $0$ & $0$ \\
\hline
$w_3$ & $-Z_2$ & $Z_1$ & $0$ & $0$ \\
\hline
\end{tabular}\label{n11}
\end{table} 
The sets $\A_{1,1}$ and $\B_{1,1}$ are given by
\begin{equation*}
\A_{1,1}=\A_{1,1}^+ \cup \A^-_{1,1}=\{w_1\} \cup \{w_4\}, \qquad
\B_{1,1}=\B_{1,1}^+ \cup \B^-_{1,1}=\{w_2\} \cup \{w_3\}. 
\end{equation*}

We define an orthonormal basis of $\mathfrak B_{\z_{2,2}}=\{Z_1, Z_2,  Z_3, Z_4\}$ and 
\begin{equation}\label{b22}
\mathfrak B_{\vv_{2,2}}=\Big\{
\begin{array}{llllll}
&w_1=w, \quad & w_2=J_1w, \quad & w_3=J_2w, \quad & w_4=J_1J_2w, \\
&w_5=J_3w, \quad & w_6=J_4w, \quad & w_7=J_1J_3w, \quad & w_8=J_1J_4w
\end{array}
\Big\},
\end{equation}
for $J_1J_2J_3J_4w=w$ with
$
\langle w_i, w_i  \rangle_{\mathfrak{v}_{2,2}}=\epsilon_i(4,4)$, $\la Z_k, Z_k \ra_{\z_{2,2}}=\epsilon_k(2,2)$.
The sets $\A_{2,2}$ and $\B_{2,2}$ are given by
$$
\A_{2,2}=\A_{2,2}^+ \cup \A^-_{2,2}=\{w_1,w_4\} \cup \{w_7,w_8\}, \qquad
\B_{2,2}=\B_{2,2}^+ \cup \B^-_{2,2}=\{w_2,w_3\} \cup \{w_5,w_6\}
$$
according to Table~\ref{n22}.
\begin{table}[h]
\caption{Commutation relations on $\n_{2,2}$}
\centering
\begin{tabular}{| c | c | c | c | c | c | c | c | c |} 
\hline
 $[ row \,, col. ]$  & $w_1$ & $w_4$ & $w_7$ & $w_8$ & $w_2$ & $w_3$ & $w_5$ & $w_6$ \\
\hline
$w_1$ & $0$ & $0$ & $0$ & $0$ & $Z_1$ & $Z_2$ & $Z_3$ & $Z_4$\\
\hline
$w_4$ & $0$ & $0$ & $0$ & $0$ & $Z_2$ & $-Z_1$ & $Z_4$ & $-Z_3$ \\
\hline
$w_7$ & $0$ & $0$ & $0$ & $0$ & $Z_3$ & $-Z_4$ & $Z_1$ & $-Z_2$\\
\hline
$w_8$ & $0$ & $0$ & $0$ & $0$ & $Z_4$ & $Z_3$ & $Z_2$ & $Z_1$ \\
\hline
$w_2$ & $-Z_1$ & $-Z_2$  & $-Z_3$ & $-Z_4$ & $0$ & $0$ & $0$ & $0$ \\
\hline
$w_3$ & $-Z_2$ & $Z_1$ & $Z_4$ & $-Z_3$ & $0$ & $0$ & $0$ & $0$ \\
\hline
$w_5$ & $-Z_3$ & $-Z_4$ & $-Z_1$ & $-Z_2$ & $0$ & $0$ & $0$ & $0$ \\
\hline
$w_6$ & $-Z_4$ & $Z_3$ & $Z_2$ & $-Z_1$ & $0$ & $0$ & $0$ & $0$ \\
\hline
\end{tabular}\label{n22}
\end{table} 

The integral basis $\mathfrak{B}_{\n_{4,4}}$ of $\n_{4,4}$ is given in~\eqref{onb441} and the decomposition is given in~\eqref{dec44} according to Table~\ref{n44}.

\end{proof}

%%%%%%%%

\subsection{Inductive construction of isomorphisms of Lie algebras $\n_{r,s}$ and $\n_{s,r}$}\label{induc}

%%%%%%%%

In this subsection we prove that if two pseudo $H$-type algebras possess decomposition~\eqref{BD} and they are isomorphic under a map satisfying some special conditions, then the extensions of them are also isomorphic and the corresponding isomorphism map satisfies the same properties. It allows us to perform an induction proof. Before we state the base of induction we formulate the properties we require from the isomorphism.

\begin{rem}\label{isom} {\sc Properties of the isomorphism $\varphi_{r,s} \colon \n_{r,s} \to \n_{s,r}$}. Let the integral bases $\mathfrak{B}_{\n_{r,s}}$, $\mathfrak{B}_{\n_{s,r}}$ of the pseudo $H$-type algebras $\n_{r,s}$, $\n_{s,r}$ admit decomposition~\eqref{BD}. Assume there exists a Lie algebra isomorphism $\varphi_{r,s} \colon \n_{r,s} \to \n_{s,r}$ such that 
$$
\varphi_{r,s}(\A_{r,s}^{\pm})=\A_{s,r}^{\pm}\quad\text{and}\quad \varphi_{r,s}(\B_{r,s}^{\pm})=\B_{s,r}^{\mp}.
$$ 
Furthermore, the restriction $\varphi_{r,s} \vert _{\z_{r,s}}$ is an anti-isometry and is a permutation of the set $\{Z_1, \dotso,Z_{r+s}\}$, i.e. $\varphi_{r,s}(Z_k)=Z_{\pi_{r,s}(k)}$ with the permutation $\pi_{r,s} \colon \{1, \dotso,r+s\} \to \{1, \dotso,r+s\}$ such that $\pi_{r,s}(\{1, \dotso,r\})=\{s+1, \dotso,s+r\}$ and $\pi_{r,s}(\{r+1, \dotso,r+s\})=\{1, \dotso,s\}$. 
\end{rem}

\begin{theorem}\label{8rr8}
The Lie algebras $\n_{r,8}$ and $\n_{8,r}$ are integral isomorphic if and only if $r \in \{0,1,2,4\}\mod8$ and the Lie algebra isomorphism $\varphi_{r,8}\colon  \n_{r,8} \to \n_{8,r}$ satisfies Remark~\ref{isom} with $s=8$.
\end{theorem}
\begin{proof}
The $H$-type Lie algebras $\n_{r,0}$ are integral isomorphic to $\n_{0,r}$ for $r \in\{ 1,2,4,8\}\mod8$ by Theorem~\ref{th:r>8}. Recall that we used the following integral block-type bases 
$$
\mathfrak B_{\n_{r,0}}=\{w_1, \dotso,w_{2l}, Z^{r,0}_1, \dotso ,Z^{r,0}_r \},
\qquad
\mathfrak B_{\n_{0,r}}=\{x_1, \dotso,x_{2l}, Z^{0,r}_1, \dotso ,Z^{0,r}_r \},
$$ 
with $\langle w_i, w_i \rangle_{\vv_{r,0}}=1$ for $i=1, \dotso,2l$, $\langle x_i, x_i \rangle_{\vv_{0,r}}=\epsilon_i(l,l)$, $\langle Z^{r,0}_k, Z^{r,0}_k \rangle_{\z_{r,0}}=1$, and  $\langle Z^{0,r}_k , Z^{0,r}_k \rangle_{\z_{r,0}}=-1$ for all $k=1, \dotso,r$, where $\varphi_{r,0}(w_i)=x_i$ for all $i=1, \dotso,2l$ and $\varphi_{r,0}(Z_k^{r,0})=Z_k^{0,r}$ for all $k=1, \dotso,r$. The equal structure constants are denoted by $A_{ij}^m$. We write $\varphi_{8,0}(u_i)=v_i$, for $u_i\in\mathfrak B_{\vv_{8,0}}$,  $v_i\in\mathfrak B_{\vv_{0,8}}$, $i=1, \dotso,16$ and $\varphi_{8,0}(Z_k^{8,0})=Z_k^{0,8}$, $k=1, \dotso,8$. The equal structure constants are denoted by $\bar{A}^m_{ij}$ for both Lie algebras. 

We exploit Proposition~\ref{ext08} and Remark~\ref{rem:FurMar} and obtain the integral bases 
$$
\{w_1 \otimes v_1, \dotso, w_{2l} \otimes v_{16}, Z_1^{r,8}, \dotso, Z^{r,8}_{r+8}\}\quad\text{ for }\quad \n_{r,8},
$$
$$
\{x_1 \otimes u_1, \dotso, x_{2l} \otimes u_{16}, Z_1^{8,r}, \dotso, Z^{8,r}_{r+8}\}\quad\text{ for }\quad \n_{8,r}.$$
Define the bijective linear map $\varphi_{r,8} \colon \n_{r,8} \to \n_{8,r}$ by
\begin{equation*}
\begin{array}{clccllc}
w_i \otimes v_j &\mapsto & x_i \otimes u_j   &\text{ if } &\quad i \in \{1, \dotso,l\}, & j \in \{1,\dotso,16\}, 
\\
\nonumber w_i \otimes v_j &\mapsto & x_i \otimes u_j   &\text{ if }& \quad i \in \{l+1, \dotso,2l\},  & j \in \{1,\dotso,8\}, 
\\
\nonumber w_i \otimes v_j &\mapsto& -x_i \otimes u_j   &\text{ if }& \quad i \in \{l+1, \dotso,2l\}, & j \in \{9,\dotso,16\}, \\
\nonumber Z^{r,8}_m &\mapsto& Z^{8,r}_{m+8} &\text{ if }& \quad m \in \{1, \dotso, r\}, \\
\nonumber Z^{r,8}_m &\mapsto& Z^{8,r}_{m-r} &\text{ if }& \quad m \in \{r+1, \dotso, r+8\}.
\end{array}
\end{equation*}
It remains to prove that $\varphi_{r,8}$ is a Lie algebra isomorphism, i.e. $\varphi_{r,8}([w_i \otimes v_j \,, w_p \otimes v_q])=[\varphi_{r,8}(w_i \otimes v_j) \,, \varphi_{r,8}(w_p \otimes v_q)]$. The structure constants $\tilde A_{ij,pq}^m$ of the commutators $[w_i \otimes v_j, w_p \otimes v_q]$ of $\n_{r,8}$ are given by formula~\eqref{A08u0}, and the structure constants $C_{ij,pq}^m$ for $[x_i \otimes u_j , x_p \otimes u_q]$ of the Lie algebra $\n_{8,r}$ are given by formula~\eqref{A80s}. It follows that if $i\not=p$ and $j\not=q$ or $i=p$ and $j=q$ the commutators vanish:
$$\varphi_{r,8}([w_i \otimes v_j \,, w_p \otimes v_q])=\varphi_{r,8}(0)=0=\pm [x_i \otimes u_j \,, x_p \otimes u_q].$$
It is left to consider the following two remaining cases.
\\

\noindent{\sc Case $i=p$ and $j\not=q$.} If, additionally, both indices simultaneously satisfy either $j,q=1, \dotso,8$ or $j,q=9, \dotso, 16$, then 
$$\varphi_{r,8}([w_i \otimes v_j \,, w_i \otimes v_q])=\varphi_{r,8}(0)=0=\pm [x_i \otimes u_j \,, x_i \otimes u_q],
$$
because of the block form of the Lie algebras $\n_{8,0}$, $\n_{0,8}$.
Thus, we can assume without limit of generality that $j=1, \dotso,8$ and $q=9, \dotso,16$.

$\bullet$ if $i=p=1, \dotso,l$, $j=1, \dotso,8$ and $q=9, \dotso,16$, then: 
$$
\varphi_{r,8}([w_i \otimes v_j \,, w_i \otimes v_q])=\bar A _{jq}^{m-r}\varphi_{r,8}(Z^{r,8}_m)=\bar A _{jq}^{m-r}Z^{8,r}_{m-r},
$$
$$
[\varphi_{r,8}(w_i \otimes v_j) \,, \varphi_{r,8}(w_i \otimes v_q)]=[x_i \otimes u_j \,, x_i \otimes u_q] =\bar A _{jq}^{m-r}Z^{8,r}_{m-r}
$$
with $m=r+1, \dotso,r+8$ by formulas~\eqref{A08u0} and~\eqref{A80s}.

$\bullet$ if $i=p=l+1, \dotso,2l$, $j=1, \dotso,8$ and $q=9, \dotso,16$, then
$$
\varphi_{r,8}([w_i \otimes v_j \,, w_i \otimes v_q])=\bar A _{jq}^{m-r}\varphi_{r,8}(Z^{r,8}_m)=\bar A _{jq}^{m-r}Z^{8,r}_{m-r},
$$
$$
[\varphi_{r,8}(w_i \otimes v_j) \,, \varphi_{r,8}(w_i \otimes v_q)]=[x_i \otimes u_j \,, - x_i \otimes u_q] =-(-\bar A _{jq}^{m-r}Z^{8,r}_{m-r})
$$
with $m=r+1, \dotso,r+8$ by formulas~\eqref{A08u0} and~\eqref{A80s}. We see that the map $\varphi_{r,8}$ satisfies the Lie algebra isomorphism properties in this case.
\\

\noindent{\sc Case $i\not=p$ and $j=q$.} If in addition $i,p=1, \dotso,l$ or $i,p=l+1, \dotso, 2l$, then the block form of the Lie algebras $\n_{r,0}$, $\n_{0,r}$ implies
$$\varphi_{r,8}([w_i \otimes v_j \,, w_p \otimes v_j])=\varphi_{r,8}(0)=0=\pm [x_i \otimes u_j \,, x_p \otimes u_j],$$
such that we can assume that $i=1, \dotso,l$ and $p=l+1, \dotso,2l$.

$\bullet$ if $j=q=1, \dotso,8$, $i=1, \dotso,l$ and $p=l+1, \dotso,2l$, then
\begin{eqnarray*}
\varphi_{r,8}([w_i \otimes v_j, w_p \otimes v_j])=-A _{ip}^{m}\varphi_{r,8}(Z^{r,8}_m)=-A _{ip}^{m}Z^{8,r}_{m+8},\\
{}[\varphi_{r,8}(w_i \otimes v_j), \varphi_{r,8}(w_p \otimes v_j)]=[x_i \otimes u_j, x_p \otimes u_j]=-A _{ip}^{m}Z^{8,r}_{m+8}.
\end{eqnarray*}

$\bullet$ if $j=q=9, \dotso,16$, $i=1, \dotso,l$ and $p=l+1, \dotso,2l$ then
\begin{eqnarray*}
\varphi_{r,8}([w_i \otimes v_j, w_p \otimes v_j])=-A _{ip}^{m}\varphi_{r,8}(Z^{r,8}_m)=-A _{ip}^{m}Z^{8,r}_{m+8},\\
{}[\varphi_{r,8}(w_i \otimes v_j), \varphi_{r,8}(w_p \otimes v_j)]=[x_i \otimes u_j, - x_p \otimes u_j]=-A _{ip}^{m}Z^{8,r}_{m+8},
\end{eqnarray*}
with $m=1, \dotso,r$ by formulas~\eqref{A08u0} and~\eqref{A80s}.
This shows that $\varphi_{r,8}$ is a Lie algebra isomorphism.
The map $\varphi_{r,8}$ satisfies Remark~\ref{isom} by its definition.
\end{proof} 

\begin{theorem}\label{condrs8008}
Assume that Lie algebras $\n_{r,s}$ and $\n_{s,r}$, $r,s\not=0$, satisfy Remark~\ref{isom}.
Then there exists a Lie algebra isomorphism $\varphi_{r+8,s} \colon \n_{r+8,s} \to \n_{s,r+8}$ and two integral bases $\mathfrak B_{r+8,s}$ and $\mathfrak B_{r,s+8}$ satisfying Remark~\ref{isom}.
\end{theorem}

\begin{proof} Let $\varphi_{r,s} \colon \n_{r,s} \to \n_{s,r}$ be the assumed Lie algebra isomorphism.
By extension we construct the Lie algebra $\n_{r+8,s}$ of dimension $32l+r+s+8$ with the basis $\mathfrak B_{r+8,s}=\{x_1 \otimes u_1, \dotso, x_{2l} \otimes u_{16}, Z_1^{r+8,s} , \dotso , Z_{r+s+8}^{r+8,s} \}$. The assumptions imply that $[x_i \otimes u_j \,, x_p \otimes u_q]=0$ for the following cases:
\begin{itemize}
\item $x_i=x_p$ and both $u_j,u_q \in \A_{8,0}$ or $u_j, u_q \in \B_{8,0}$,
\item $u_j=u_q$ and both $x_i, x_p \in \A_{r,s}$ or $x_i, x_p \in \B_{r,s}$,
\item $x_i \not=x_p$ and $u_j\not = u_q$ or $x_i=x_p$ and $u_j=u_q$,
\end{itemize}
by formula~\eqref{A80s}, where $\A_{8,0},\B_{8,0}$ are defined in~\eqref{b0880}. 
Then we define the bijective linear map $\varphi_{r+8,s} \colon \n_{r+8,s} \to \n_{s,r+8}$ by 
\begin{equation*}
\begin{array}{cllcllc}
x_i \otimes u_{\alpha}  &\mapsto & -\varphi_{r,s}(x_i) \otimes \varphi_{8,0}(u_{\alpha})  &\text{ if } &\quad x_i \in  \B_{r,s}, \quad \text{ and} \quad u_{\alpha} \in \B_{8,0}, 
\\
\nonumber x_i  \otimes u_{\alpha} &\mapsto & \varphi_{r,s}(x_i) \otimes  \varphi_{8,0}(u_{\alpha})   &\text{ if }& \quad  \text{otherwise },  &  
\\
\nonumber Z_m^{r+8,s} &\mapsto& Z_{\pi_{r,s}(m)}^{s,r+8} &\text{ if }& \quad m \in \{1, \dotso, r\}, \\
\nonumber Z_m^{r+8,s} &\mapsto& Z_{\pi_{8,0}(m-r)+r+s}^{s,r+8} &\text{ if }& \quad m \in \{r+1, \dotso, r+8\},\\
\nonumber Z_m^{r+8,s} &\mapsto& Z_{\pi_{r,s}(m-8)}^{s,r+8} &\text{ if }& \quad m \in \{r+9, \dotso, r+s+8\},
\end{array}
\end{equation*}
where $\varphi_{8,0} \colon \n_{8,0} \to \n_{0,8}$ is the Lie algebra isomorphism given by~\eqref{isom8008varphi}.
We see that the restriction of $\varphi_{r+8,s}$ to $\z_{r+8,s}$ is an anti-isometry, such that it remains to prove that $\varphi_{r+8,s}$ is a Lie algebra homomorphism.

Before we continue, we draw the attention of the reader to the following. By Lemma~\ref{depA80} we know that $[x_i \otimes u_j, x_i \otimes u_q]=\pm[u_j, u_q]_{r+8,s} \in \spn\{Z_k^{r+8,s} \vert k=r+1, \dotso, r+8\}$. Since the index $k$ belongs to the set $\{r+1, \dotso, r+8\}$, the structure constants $[u_j, u_q]$ in $\n_{r+8,s}$ coincide with the structure constants $[u_j, u_q]$ in $\n_{8,0}$. Analogously we write $[x_i \otimes u_j, x_p \otimes u_j]=\pm[x_i, x_p]_{r+8,s} \in \spn\{Z_k \vert k=1, \dotso,r, r+9, \dotso ,r+s+8\}$ and observe that $[x_i, x_p]_{r+8,s}=[x_i, x_p]_{r,s}$.
Thus 
$$\varphi_{8,0}([u_j \,, u_q]_{r+8,s})=[\varphi_{8,0}(u_j) \,,\varphi_{8,0}(u_q)]_{r+8,s},
\quad \varphi_{r,s}([x_i \,, x_p]_{r+8,s})=[\varphi_{r,s}(x_i) \,,\varphi_{r,s}(x_p)]_{r+8,s}
$$ as $\varphi_{8,0}$ and $\varphi_{r,s}$ are Lie algebra isomorphisms. 
Now, we consider the following cases by using formulas~\eqref{A80s} and~\eqref{A08u}.

$\bullet$ If $x_i=x_p \in  \B_{r,s}^+ $, $u_j \in \A_{8,0}$ and $u_q \in \B_{8,0}$, then: 
$$
\varphi_{r+8,s}([x_i \otimes u_j \,, x_i \otimes u_q])=\varphi_{r+8,s}([u_j \,, u_q]_{r+8,s})=\varphi_{8,0}([u_j \,, u_q]_{r+8,s}),
$$
\begin{eqnarray*}
[\varphi_{r+8,s}(x_i \otimes u_j) \,, \varphi_{r+8,s}(x_i \otimes u_q)]&=&[\varphi_{r,s}(x_i) \otimes \varphi_{8,0}(u_j) \,, -\varphi_{r,s}(x_i) \otimes \varphi_{8,0}(u_q)] \\&=&[\varphi_{8,0}(u_j) \,,\varphi_{8,0}(u_q)]_{r+8,s}
\end{eqnarray*}
as $\varphi_{r,s}(x_i) \in  \pm\B^-_{s,r}$. 

$\bullet$ If $x_i=x_p \in \B_{r,s}^- $, $u_j \in \A_{8,0}$ and $u_q \in \B_{8,0}$, then: 
$$
\varphi_{r+8,s}([x_i \otimes u_j \,, x_i \otimes u_q])=\varphi_{r+8,s}(-[u_j \,, u_q]_{r+8,s})=-\varphi_{8,0}([u_j \,, u_q]_{r+8,s}),
$$
\begin{eqnarray*}
[\varphi_{r+8,s}(x_i \otimes u_j) \,, \varphi_{r+8,s}(x_i \otimes u_q)]&=&[\varphi_{r,s}(x_i) \otimes \varphi_{8,0}(u_j) \,, -\varphi_{r,s}(x_i) \otimes \varphi_{8,0}(u_q)] \\&=&-[\varphi_{8,0}(u_j) \,,\varphi_{8,0}(u_q)]_{r+8,s}
\end{eqnarray*}
as $\varphi_{r,s}(x_i) \in \pm  \B^+_{s,r}$. 

%%%%

$\bullet$ If $x_i=x_p \in \A_{r,s}^+ $, $u_j \in \A_{8,0}$ and $u_q \in \B_{8,0}$, then: 
$$
\varphi_{r+8,s}([x_i \otimes u_j \,, x_i \otimes u_q])=\varphi_{r+8,s}([u_j \,, u_q]_{r+8,s})=\varphi_{8,0}([u_j \,, u_q]_{r+8,s}),
$$
\begin{eqnarray*}
[\varphi_{r+8,s}(x_i \otimes u_j) \,, \varphi_{r+8,s}(x_i \otimes u_q)]&=&[\varphi_{r,s}(x_i) \otimes \varphi_{8,0}(u_j) \,, \varphi_{r,s}(x_i) \otimes \varphi_{8,0}(u_q)] \\ &=&[\varphi_{8,0}(u_j) \,,\varphi_{8,0}(u_q)]_{r+8,s}
\end{eqnarray*}
as $\varphi_{r,s}(x_i) \in \pm \A^+_{s,r}$. 

%%%%

$\bullet$ If $x_i=x_p \in \A_{r,s}^- $, $u_j \in \A_{8,0}$ and $u_q \in \B_{8,0}$, then: 
$$
\varphi_{r+8,s}([x_i \otimes u_j \,, x_i \otimes u_q])=\varphi_{r+8,s}(-[u_j \,, u_q]_{r+8,s})=-\varphi_{8,0}([u_j \,, u_q]_{r+8,s}),
$$
\begin{eqnarray*}
[\varphi_{r+8,s}(x_i \otimes u_j) \,, \varphi_{r+8,s}(x_i \otimes u_q)]&=&[\varphi_{r,s}(x_i) \otimes \varphi_{8,0}(u_j) \,, \varphi_{r,s}(x_i) \otimes \varphi_{8,0}(u_q)] \\&=&-[\varphi_{8,0}(u_j) \,,\varphi_{8,0}(u_q)]_{r+8,s}
\end{eqnarray*}
as $\varphi_{r,s}(x_i) \in \pm \A^-_{s,r}$.

%%%%%%%%%%%%%%

$\bullet$ If $u_j=u_q \in \B_{8,0} $, $x_i \in \A_{r,s}$ and $x_p \in \B_{r,s}$, then: 
$$
\varphi_{r+8,s}([x_i \otimes u_j \,, x_p \otimes u_j])=\varphi_{r+8,s}([x_i \,, x_p]_{r+8,s})=\varphi_{r,s}([x_i \,, x_p]_{r+8,s}),
$$
\begin{eqnarray*}
[\varphi_{r+8,s}(x_i \otimes u_j) \,, \varphi_{r+8,s}(x_p \otimes u_j)]&=&[\varphi_{r,s}(x_i) \otimes \varphi_{8,0}(u_j) \,, -\varphi_{r,s}(x_p) \otimes \varphi_{8,0}(u_j)]\\ &=&[\varphi_{r,s}(x_i) \,,\varphi_{r,s}(x_p)]_{r+8,s}
\end{eqnarray*}
as $\varphi_{8,0}(u_j) \in \pm \B_{0,8}$.

%%%%

$\bullet$ If $u_j=u_q \in \A_{8,0}$, $x_i \in \A_{r,s}$ and $x_p \in \B_{r,s}$, then: 
$$
\varphi_{r+8,s}([x_i \otimes u_j \,, x_p \otimes u_j])=\varphi_{r+8,s}(-[x_i \,, x_p]_{r+8,s})=-\varphi_{r,s}([x_i \,, x_p]_{r+8,s}),
$$
\begin{eqnarray*}
[\varphi_{r+8,s}(x_i \otimes u_j) \,, \varphi_{r+8,s}(x_p \otimes u_j)]&=&[\varphi_{r,s}(x_i) \otimes \varphi_{8,0}(u_j) \,, \varphi_{r,s}(x_p) \otimes \varphi_{8,0}(u_j)]\\ &=&-[\varphi_{r,s}(x_i) \,,\varphi_{r,s}(x_p)]_{r+8,s}
\end{eqnarray*}
as $\varphi_{8,0}(u_j) \in \pm \A_{0,8}$.

Hence $\varphi_{r+8,s}$ is a Lie algebra isomorphism satisfying Remark~\ref{isom}.
\end{proof}

Now we turn to consider the extension obtained by making use of the tensor product with $\vv_{4,4}$.

\begin{theorem}\label{auto112244}
For any $\n_{r,r}$ with $r=1,2,4$ there exists an automorphism $\varphi_{r,r} \colon \n_{r,r} \to \n_{r,r}$ and an integral basis $\mathfrak B_{r,r}$ 
satisfying Remark~\ref{isom}.
\end{theorem}

\begin{proof}
In this proof we explicitly state the automorphisms.

The basis of $\n_{1,1}$ is given in~\eqref{b11} and the commutations in Table~\ref{n11}.
The automorphism $\varphi_{1,1} \colon \n_{1,1} \to \n_{1,1}$ with anti-isometry on the center is given by
\begin{equation}\label{11}
\begin{array}{llllll}
&Z_1^{1,1}\mapsto Z_2^{1,1}, \quad & Z_2^{1,1}\mapsto Z_1^{1,1}, \\
&w_1\mapsto w_1, \quad & w_2 \mapsto w_3, \quad & w_3\mapsto w_2, \quad & w_4\mapsto w_4.
\end{array}
\end{equation}

We defined an orthonormal basis of $\n_{2,2}$ in~\eqref{b22} with commutators in Table~\ref{n22}.
The automorphism $\varphi_{2,2} \colon \n_{2,2} \to \n_{2,2}$ with anti-isometry on the center is given by
\begin{equation}\label{22}
\begin{array}{llllll}
&Z_1^{2,2}\mapsto Z_3^{2,2}, \quad & Z_2^{2,2}\mapsto Z_4^{2,2}, \quad &Z_3^{2,2}\mapsto Z_1^{2,2}, \quad & Z_4^{2,2}\mapsto Z_2^{2,2}, \\
&w_1\mapsto w_1, \quad & w_2 \mapsto w_5, \quad & w_3\mapsto w_6, \quad & w_4\mapsto w_4, \\
&w_5\mapsto w_2, \quad & w_6 \mapsto w_3, \quad & w_7\mapsto w_7, \quad & w_8\mapsto w_8.
\end{array}
\end{equation}

Recalling the basis~\eqref{onb441} and Table~\ref{n44} we define the automorphism $\varphi_{4,4} \colon \n_{4,4} \to \n_{4,4}$ with anti-isometry on the center by
\begin{equation}\label{44}
\begin{array}{llllll}
&Z_1^{4,4}\mapsto Z_5^{4,4}, \quad & Z_2^{4,4}\mapsto Z_6^{4,4}, \quad & Z_3^{4,4}\mapsto Z_8^{4,4}, \quad & Z_4^{4,4} \mapsto Z_7^{4,4}, \\
&Z_5^{4,4}\mapsto Z_1^{4,4}, \quad & Z_6^{4,4} \mapsto Z_2^{4,4}, \quad & Z_7^{4,4}\mapsto Z_4^{4,4}, \quad & Z_8^{4,4}\mapsto Z_3^{4,4}, \\
&y_1\mapsto y_1, \quad & y_2 \mapsto y_9, \quad & y_3\mapsto y_{10}, \quad & y_4\mapsto y_{12}, \\
&y_5\mapsto y_{11}, \quad & y_6 \mapsto y_6, \quad & y_7\mapsto y_{7}, \quad & y_8\mapsto -y_{8}, \\
&y_9\mapsto y_{2}, \quad & y_{10} \mapsto y_3, \quad & y_{11}\mapsto y_{5}, \quad & y_{12}\mapsto y_{4}, \\
&y_{13}\mapsto y_{13}, \quad & y_{14} \mapsto y_{14}, \quad & y_{15}\mapsto -y_{15}, \quad & y_{16}\mapsto y_{16}. \\
\end{array}
\end{equation}
\end{proof}

\begin{theorem}\label{4rr4}
The Lie algebras $\n_{r+4,4}$ and $\n_{4,r+4}$ are integral isomorphic if and only if $r \in \{0,1,2,4\}\mod8$. In the case $r \in \{0,1,2,4\}\mod8$ there exists a Lie algebra isomorphism $\varphi_{r+4,4}\colon  \n_{r+4,4} \to \n_{4,r+4}$ and two integral bases $\mathfrak B_{r+4,4}$ and $\mathfrak B_{4,r+4}$ satisfying Remark~\ref{isom}.
\end{theorem}

\begin{proof}
By extension we construct the Lie algebra $\n_{r+4,4}$ of dimension $32l+r+8$ with the integral basis $\{x_1 \otimes y_1, \dotso, x_{2l} \otimes y_{16}, Z_1^{r+4,4} , \dotso , Z_{r+8}^{r+4,4} \}$, where $\{x_1, \dotso,x_{2l}\}=\mathfrak B_{\vv_{r,0}}$. Lemma~\ref{depA44} implies that $[x_i \otimes y_j \,, x_p \otimes y_q]=0$ for the following cases:
\begin{itemize}
\item $x_i=x_p$ and both $y_j,y_q \in \A_{4,4}$ or $y_j, y_q \in \B_{4,4}$,
\item $y_j=y_q$ and both $x_i, x_p \in \A_{r,0}$ or $x_i, x_p \in \B_{r,0}$,
\item $x_i \not=x_p$ and $y_j\not = y_q$ or $x_i=x_p$ and $y_j=y_q$, 
\end{itemize}
where $\A_{r,0}, \B_{r,0}$ are defined in~\eqref{Ar0Br0} and $\A_{4,4}, \B_{4,4}$ are defined in~\eqref{dec44}.
We define the bijective linear map $\varphi_{r+4,4} \colon \n_{r+4,4} \to \n_{4,r+4}$ by 
\begin{equation*}
\begin{array}{clccllc}
x_i \otimes y_{\alpha}  &\mapsto & -\varphi_{r,0}(x_i) \otimes \varphi_{4,4}(y_{\alpha})  &\text{ if } &\quad x_i \in  \B_{r,0}, \quad \text{ and} \quad y_{\alpha} \in \B_{4,4}, 
\\
\nonumber x_i  \otimes y_{\alpha} &\mapsto & \varphi_{r,0}(x_i) \otimes  \varphi_{4,4}(y_{\alpha})   &\text{ if }& \quad  \text{otherwise },  &  
\\
\nonumber Z_m^{r+4,4} &\mapsto& Z_{\pi_{r,0}(m)+8}^{4,r+4} &\text{ if }& \quad m \in \{1, \dotso, r\}, \\
\nonumber Z_m^{r+4,4} &\mapsto& Z_{\pi_{4,4}(m-r)}^{4,r+4} &\text{ if }& \quad m \in \{r+1, \dotso, r+8\}.
\end{array}
\end{equation*}
We see that the restriction of $\varphi_{r+4,4}$ to $\z_{r+4,4}$ is an anti-isometry, such that it remains to prove that $\varphi_{r+4,4}$ is a Lie algebra homomorphism.

As in Theorem~\ref{condrs8008} we make the following observation. By Lemma~\ref{depA44} we know that $[x_i \otimes y_j , x_i \otimes y_q]=\pm[y_j , y_q]_{r+4,4} \in \spn\{Z_k \vert k=r+1, \dotso, r+8\}$ and therefore $[y_j , y_q]_{r+4,4}=[y_j , y_q]_{4,4}$. Analogously, because of
$[x_i \otimes y_j , x_p \otimes y_j]=\pm[x_i , x_p]_{r+4,4} \in \spn\{Z_k \vert k=1, \dotso ,r\}$ we obtain $[x_i , x_p]_{r+4,4}=[x_i , x_p]_{r,0}$. Thus 
 $\varphi_{4,4}([y_j \,, y_q]_{r+4,4})=[\varphi_{4,4}(y_j),\varphi_{4,4}(y_q)]_{r+4,4}$ and $\varphi_{r,0}([x_i \,, x_p]_{r+4,4})=[\varphi_{r,0}(x_i) ,\varphi_{r,0}(x_p)]_{r+4,4}$, respectively, as $\varphi_{4,4}$ and $\varphi_{r,0}$ are Lie algebra isomorphisms. We turn to consider several cases, where we use formulas~\eqref{A44u0} and~\eqref{A44u}.

$\bullet$ If $x_i=x_p \in \B_{r,0} $, $y_j \in \A_{4,4}$ and $y_q \in \B_{4,4}$, then: 
$$
\varphi_{r+4,4}([x_i \otimes y_j \,, x_i \otimes y_q])=\varphi_{r+4,4}([y_j \,, y_q]_{r+4,4})=\varphi_{4,4}([y_j \,, y_q]_{r+4,4}),
$$
\begin{eqnarray*}
[\varphi_{r+4,4}(x_i \otimes y_j) \,, \varphi_{r+4,4}(x_i \otimes y_q)]&=&[\varphi_{r,0}(x_i) \otimes \varphi_{4,4}(y_j) \,, -\varphi_{r,0}(x_i) \otimes \varphi_{4,4}(y_q)] \\&=&[\varphi_{4,4}(y_j) \,,\varphi_{4,4}(y_q)]_{r+4,4}
\end{eqnarray*}
as $\varphi_{r,0}(x_i) \in \pm \B_{0,r}$. 

%%%%

$\bullet$ If $x_i=x_p \in \A_{r,0} $, $y_j \in \A_{4,4}$ and $y_q \in \B_{4,4}$, then: 
$$
\varphi_{r+4,4}([x_i \otimes y_j \,, x_i \otimes y_q])=\varphi_{r+4,4}([y_j \,, y_q]_{r+4,s+4})=\varphi_{4,4}([y_j \,, y_q]_{r+4,4}),
$$
\begin{eqnarray*}
[\varphi_{r+4,4}(x_i \otimes y_j) \,, \varphi_{r+4,4}(x_i \otimes y_q)]&=&[\varphi_{r,0}(x_i) \otimes \varphi_{4,4}(y_j) \,, \varphi_{r,0}(x_i) \otimes \varphi_{4,4}(y_q)] \\&=&[\varphi_{4,4}(y_j) \,,\varphi_{4,4}(y_q)]_{r+4,4}
\end{eqnarray*}
as $\varphi_{r,0}(x_i) \in \pm  \A_{0,r}$.

%%%%%%%%%%%%%%

$\bullet$ If $y_j=y_q \in \B_{4,4}^+ $, $x_i \in \A_{r,0}$ and $x_p \in \B_{r,0}$, then: 
$$
\varphi_{r+4,4}([x_i \otimes y_j \,, x_p \otimes y_j])=\varphi_{r+4,4}([x_i \,, x_p]_{r+4,4})=\varphi_{r,0}([x_i \,, x_p]_{r+4,4}),
$$
\begin{eqnarray*}
[\varphi_{r+4,4}(x_i \otimes y_j) \,, \varphi_{r+4,4}(x_p \otimes y_j)]&=&[\varphi_{r,0}(x_i) \otimes \varphi_{4,4}(y_j) \,, -\varphi_{r,0}(x_p) \otimes \varphi_{4,4}(y_j)]\\ &=&[\varphi_{r,0}(x_i) \,,\varphi_{r,0}(x_p)]_{r+4,4}
\end{eqnarray*}
as $\varphi_{4,4}(y_j) \in \pm \B^-_{4,4}$.

%%%%

$\bullet$ If $y_j=y_q \in \B_{4,4}^- $, $x_i \in \A_{r,0}$ and $x_p \in \B_{r,0}$, then: 
$$
\varphi_{r+4,4}([x_i \otimes y_j \,, x_p \otimes y_j])=\varphi_{r+4,4}(-[x_i \,, x_p]_{r+4,4})=-\varphi_{r,0}([x_i \,, x_p]_{r+4,4}),
$$
\begin{eqnarray*}
[\varphi_{r+4,4}(x_i \otimes y_j) \,, \varphi_{r+4,4}(x_p \otimes y_j)]&=&[\varphi_{r,0}(x_i) \otimes \varphi_{4,4}(y_j) \,, -\varphi_{r,0}(x_p) \otimes \varphi_{4,4}(y_j)]\\ &=&-[\varphi_{r,0}(x_i) \,,\varphi_{r,0}(x_p)]_{r+4,4}
\end{eqnarray*}
as $\varphi_{4,4}(y_j) \in \pm \B^+_{4,4}$.

%%%%

$\bullet$ If $y_j=y_q \in \A_{4,4}^+ $, $x_i \in \A_{r,0}$ and $x_p \in \B_{r,0}$, then: 
$$
\varphi_{r+4,4}([x_i \otimes y_j \,, x_p \otimes y_j])=\varphi_{r+4,4}(-[x_i \,, x_p]_{r+4,4})=-\varphi_{r,0}([x_i \,, x_p]_{r+4,4}),
$$
\begin{eqnarray*}
[\varphi_{r+4,4}(x_i \otimes y_j) \,, \varphi_{r+4,4}(x_p \otimes y_j)]&=&[\varphi_{r,0}(x_i) \otimes \varphi_{4,4}(y_j) \,, \varphi_{r,0}(x_p) \otimes \varphi_{4,4}(y_j)] \\ &=&-[\varphi_{r,0}(x_i) \,,\varphi_{r,0}(x_p)]_{r+4,4}
\end{eqnarray*}
as $\varphi_{4,4}(y_j) \in \pm \A^+_{4,4}$.

%%%%

$\bullet$ If $y_j=y_q \in \A_{4,4}^- $, $x_i \in \A_{r,0}$ and $x_p \in \B_{r,0}$, then: 
$$
\varphi_{r+4,4}([x_i \otimes y_j \,, x_p \otimes y_j])=\varphi_{r+4,4}([x_i \,, x_p]_{r+4,4})=\varphi_{r,0}([x_i \,, x_p]_{r+4,4}),
$$
\begin{eqnarray*}
[\varphi_{r+4,4}(x_i \otimes y_j) \,, \varphi_{r+4,4}(x_p \otimes y_j)]&=&[\varphi_{r,0}(x_i) \otimes \varphi_{4,4}(y_j) \,, \varphi_{r,0}(x_p) \otimes \varphi_{4,4}(y_j)]\\ &=&[\varphi_{r,0}(x_i) \,,\varphi_{r,0}(x_p)]_{r+4,s+4}
\end{eqnarray*}
as $\varphi_{4,4}(y_j) \in \pm \A^-_{4,4}$.
Hence $\varphi_{r+4,4}$ is a Lie algebra isomorphism which is an anti-isometry on its center $\z_{r+4,4}$ satisfying Remark~\ref{isom}.
\end{proof}

The generalization of these results is based on the technical Lemma~\ref{depA44}.

\begin{theorem}\label{condrs44}
Assume that the pseudo $H$-type Lie algebras  $\n_{r,s}$ and $\n_{s,r}$, $r,s\not=0$ satisfy Remark~\ref{isom}. Then there exists a Lie algebra isomorphism $\varphi_{r+4,s+4} \colon \n_{r+4,s+4} \to \n_{s+4,r+4}$ also satisfying Remark~\ref{isom}.
\end{theorem}

\begin{proof}
By extension we construct the Lie algebra $\n_{r+4,s+4}$ of dimension $32l+r+s+8$ with the basis $\{x_1 \otimes y_1, \dotso, x_{2l} \otimes y_{16}, Z_1^{r+4,s+4} , \dotso , Z_{r+s+8}^{r+4,s+4} \}$, where $\{x_1, \dotso,x_{2l}\}=\mathfrak{B}_{\vv_{r,s}}$. The assumptions imply that $[x_i \otimes y_j \,, x_p \otimes y_q]=0$ for the following cases:
\begin{itemize}
\item $x_i=x_p$ and both $y_j,y_q \in \A_{4,4}$ or $y_j, y_q \in \B_{4,4}$,
\item $y_j=y_q$ and both $x_i, x_p \in \A_{r,s}$ or $x_i, x_p \in \B_{r,s}$,
\item $x_i \not=x_p$ and $y_j\not = y_q$ or $x_i=x_p$ and $y_j=y_q$,
\end{itemize}
by formula~\eqref{A44u}.
We define the bijective linear map $\varphi_{r+4,s+4} \colon \n_{r+4,s+4} \to \n_{s+4,r+4}$ by 
\begin{equation*}
\begin{array}{clccllc}
x_i \otimes y_{\alpha}  &\mapsto & -\varphi_{r,s}(x_i) \otimes \varphi_{4,4}(y_{\alpha})  &\text{ if } &\quad x_i \in  \B_{r,s}, \quad \text{ and} \quad y_{\alpha} \in \B_{4,4}, 
\\
\nonumber x_i  \otimes y_{\alpha} &\mapsto & \varphi_{r,s}(x_i) \otimes  \varphi_{4,4}(y_{\alpha})   &\text{ if }& \quad  \text{otherwise },  &  
\\
\nonumber Z_m^{r+4,s+4} &\mapsto& Z_{\pi_{r,s}(m)+8}^{r+4,s+4} &\text{ if }& \quad m \in \{1, \dotso, r\}, \\
\nonumber Z_m^{r+4,s+4} &\mapsto& Z_{\pi_{4,4}(m-r)+s}^{r+4,s+4} &\text{ if }& \quad m \in \{r+1, \dotso, r+8\},\\
\nonumber Z_m^{r+4,s+4} &\mapsto& Z_{\pi_{r,s}(m-8)}^{r+4,s+4} &\text{ if }& \quad m \in \{r+9, \dotso, r+s+8\}.
\end{array}
\end{equation*}
We see that the restriction of $\varphi_{r+4,s+4}$ to $\z_{r+4,s+4}$ is an anti-isometry. Let us show that $\varphi_{r+4,s+4}$ is a lie algebra homomorphism.

Observe that  Lemma~\ref{depA44} implies that $[x_i \otimes y_j , x_i \otimes y_q]=\pm[y_j , y_q]_{r+4,s+4} \in \spn\{Z_k \vert k=r+1, \dotso, r+8\}$ and $[x_i \otimes y_j, x_p \otimes y_j]=\pm[x_i, x_p]_{r+4,s+4} \in \spn\{Z_k \vert k=1, \dotso r,r+9, \dotso ,r+s+8\}$. Thus  $[y_j, y_q]_{r+4,s+4}=[y_j, y_q]_{4,4}$ and $[x_i, x_p]_{r+4,s+4}=[x_i, x_p]_{r,s}$. Therefore, 
$$\varphi_{4,4}([y_j , y_q]_{r+4,s+4})=[\varphi_{4,4}(y_j),\varphi_{4,4}(y_q)]_{r+4,s+4},
$$
$$
\varphi_{r,s}([x_i , x_p]_{r+4,s+4})=[\varphi_{r,s}(x_i),\varphi_{r,s}(x_p)]_{r+4,s+4},
$$ respectively, as $\varphi_{4,4}$ and $\varphi_{r,s}$ are Lie algebra isomorphisms. The remaining cases follow from formula~\eqref{A44u}.

$\bullet$ If $x_i=x_p \in \B_{r,s}^+ $, $y_j \in \A_{4,4}$ and $y_q \in \B_{4,4}$, then: 
$$
\varphi_{r+4,s+4}([x_i \otimes y_j \,, x_i \otimes y_q])=\varphi_{r+4,s+4}([y_j \,, y_q]_{r+4,s+4})=\varphi_{4,4}([y_j \,, y_q]_{r+4,s+4}),
$$
\begin{eqnarray*}
[\varphi_{r+4,s+4}(x_i \otimes y_j) \,, \varphi_{r+4,s+4}(x_i \otimes y_q)]&=&[\varphi_{r,s}(x_i) \otimes \varphi_{4,4}(y_j) \,, -\varphi_{r,s}(x_i) \otimes \varphi_{4,4}(y_q)] \\&=&[\varphi_{4,4}(y_j) \,,\varphi_{4,4}(y_q)]_{r+4,s+4}
\end{eqnarray*}
as $\varphi_{r,s}(x_i) \in \pm \B^-_{s,r}$. 

%%%%

$\bullet$ If $x_i=x_p \in \B_{r,s}^- $, $y_j \in \A_{4,4}$ and $y_q \in \B_{4,4}$, then: 
$$
\varphi_{r+4,s+4}([x_i \otimes y_j \,, x_i \otimes y_q])=\varphi_{r+4,s+4}(-[y_j \,, y_q]_{r+4,s+4})=-\varphi_{4,4}([y_j \,, y_q]_{r+4,s+4}),
$$
\begin{eqnarray*}
[\varphi_{r+4,s+4}(x_i \otimes y_j) \,, \varphi_{r+4,s+4}(x_i \otimes y_q)]&=&[\varphi_{r,s}(x_i) \otimes \varphi_{4,4}(y_j) \,, -\varphi_{r,s}(x_i) \otimes \varphi_{4,4}(y_q)] \\&=&-[\varphi_{4,4}(y_j) \,,\varphi_{4,4}(y_q)]_{r+4,s+4}
\end{eqnarray*}
as $\varphi_{r,s}(x_i) \in \pm \B^+_{s,r}$. 

%%%%

$\bullet$ If $x_i=x_p \in \A_{r,s}^+ $, $y_j \in \A_{4,4}$ and $y_q \in \B_{4,4}$, then: 
$$
\varphi_{r+4,s+4}([x_i \otimes y_j \,, x_i \otimes y_q])=\varphi_{r+4,s+4}([y_j \,, y_q]_{r+4,s+4})=\varphi_{4,4}([y_j \,, y_q]_{r+4,s+4}),
$$
\begin{eqnarray*}
[\varphi_{r+4,s+4}(x_i \otimes y_j) \,, \varphi_{r+4,s+4}(x_i \otimes y_q)]&=&[\varphi_{r,s}(x_i) \otimes \varphi_{4,4}(y_j) \,, \varphi_{r,s}(x_i) \otimes \varphi_{4,4}(y_q)] \\ &=&[\varphi_{4,4}(y_j) \,,\varphi_{4,4}(y_q)]_{r+4,s+4}
\end{eqnarray*}
as $\varphi_{r,s}(x_i) \in \pm \A^+_{s,r}$. 

%%%%

$\bullet$ If $x_i=x_p \in \A_{r,s}^- $, $y_j \in \A_{4,4}$ and $y_q \in \B_{4,4}$, then: 
$$
\varphi_{r+4,s+4}([x_i \otimes y_j \,, x_i \otimes y_q])=\varphi_{r+4,s+4}(-[y_j \,, y_q]_{r+4,s+4})=-\varphi_{4,4}([y_j \,, y_q]_{r+4,s+4}),
$$
\begin{eqnarray*}
[\varphi_{r+4,s+4}(x_i \otimes y_j) \,, \varphi_{r+4,s+4}(x_i \otimes y_q)]&=&[\varphi_{r,s}(x_i) \otimes \varphi_{4,4}(y_j) \,, \varphi_{r,s}(x_i) \otimes \varphi_{4,4}(y_q)] \\&=&-[\varphi_{4,4}(y_j) \,,\varphi_{4,4}(y_q)]_{r+4,s+4}
\end{eqnarray*}
as $\varphi_{r,s}(x_i) \in \pm \A^-_{s,r}$.

%%%%%%%%%%%%%%

$\bullet$ If $y_j=y_q \in \B_{4,4}^+ $, $x_i \in \A_{r,s}$ and $x_p \in \B_{r,s}$, then: 
$$
\varphi_{r+4,s+4}([x_i \otimes y_j \,, x_p \otimes y_j])=\varphi_{r+4,s+4}([x_i \,, x_p]_{r+4,s+4})=\varphi_{r,s}([x_i \,, x_p]_{r+4,s+4}),
$$
\begin{eqnarray*}
[\varphi_{r+4,s+4}(x_i \otimes y_j) \,, \varphi_{r+4,s+4}(x_p \otimes y_j)]&=&[\varphi_{r,s}(x_i) \otimes \varphi_{4,4}(y_j) \,, -\varphi_{r,s}(x_p) \otimes \varphi_{4,4}(y_j)]\\ &=&[\varphi_{r,s}(x_i) \,,\varphi_{r,s}(x_p)]_{r+4,s+4}
\end{eqnarray*}
as $\varphi_{4,4}(y_j) \in \pm \B^-_{4,4}$.

%%%%

$\bullet$ If $y_j=y_q \in \B_{4,4}^- $, $x_i \in \A_{r,s}$ and $x_p \in \B_{r,s}$, then: 
$$
\varphi_{r+4,s+4}([x_i \otimes y_j \,, x_p \otimes y_j])=\varphi_{r+4,s+4}(-[x_i \,, x_p]_{r+4,s+4})=-\varphi_{r,s}([x_i \,, x_p]_{r+4,s+4}),
$$
\begin{eqnarray*}
[\varphi_{r+4,s+4}(x_i \otimes y_j) \,, \varphi_{r+4,s+4}(x_p \otimes y_j)]&=&[\varphi_{r,s}(x_i) \otimes \varphi_{4,4}(y_j) \,, -\varphi_{r,s}(x_p) \otimes \varphi_{4,4}(y_j)]\\ &=&-[\varphi_{r,s}(x_i) \,,\varphi_{r,s}(x_p)]_{r+4,s+4}
\end{eqnarray*}
as $\varphi_{4,4}(y_j) \in \pm \B^+_{4,4}$.

%%%%

$\bullet$ If $y_j=y_q \in \A_{4,4}^+ $, $x_i \in \A_{r,s}$ and $x_p \in \B_{r,s}$, then: 
$$
\varphi_{r+4,s+4}([x_i \otimes y_j \,, x_p \otimes y_j])=\varphi_{r+4,s+4}([x_i \,, x_p]_{r+4,s+4})=\varphi_{r,s}([x_i \,, x_p]_{r+4,s+4}),
$$
\begin{eqnarray*}
[\varphi_{r+4,s+4}(x_i \otimes y_j) \,, \varphi_{r+4,s+4}(x_p \otimes y_j)]&=&[\varphi_{r,s}(x_i) \otimes \varphi_{4,4}(y_j) \,, \varphi_{r,s}(x_p) \otimes \varphi_{4,4}(y_j)] \\ &=&[\varphi_{r,s}(x_i) \,,\varphi_{r,s}(x_p)]_{r+4,s+4}
\end{eqnarray*}
as $\varphi_{4,4}(y_j) \in \pm \A^+_{4,4}$.

%%%%

$\bullet$ If $y_j=y_q \in \A_{4,4}^- $, $x_i \in \A_{r,s}$ and $x_p \in \B_{r,s}$, then: 
$$
\varphi_{r+4,s+4}([x_i \otimes y_j \,, x_p \otimes y_j])=\varphi_{r+4,s+4}(-[x_i \,, x_p]_{r+4,s+4})=-\varphi_{r,s}([x_i \,, x_p]_{r+4,s+4}),
$$
\begin{eqnarray*}
[\varphi_{r+4,s+4}(x_i \otimes y_j) \,, \varphi_{r+4,s+4}(x_p \otimes y_j)]&=&[\varphi_{r,s}(x_i) \otimes \varphi_{4,4}(y_j) \,, \varphi_{r,s}(x_p) \otimes \varphi_{4,4}(y_j)]\\ &=&-[\varphi_{r,s}(x_i) \,,\varphi_{r,s}(x_p)]_{r+4,s+4}
\end{eqnarray*}
as $\varphi_{4,4}(y_j) \in \pm \A^-_{4,4}$.

Hence $\varphi_{r+4,s+4}$ is a Lie algebra isomorphism which satisfies Remark~\ref{isom}.
\end{proof}

%%%%%%%%

\subsection{Main results of Section~\ref{New}}

%%%%%%%%

\begin{theorem}\label{th:14}
The $H$-type Lie algebras $\n_{r+4t_1+8t_2, 8t_3+4t_1}$ and $\n_{8t_3+4t_1,r+4t_1+8t_2}$ are integral isomorphic for $r \in \{0,1,2,4\}\mod8$ and $t_1,t_2,t_3 \in \mathbb{N}\cup\{0\}$.
\end{theorem}
\begin{proof}
We prove by induction. The beginning of the induction is stated in Theorem~\ref{8rr8} and Theorem~\ref{4rr4}.
Then the induction step is given by Theorem~\ref{condrs8008} and Theorem~\ref{condrs44}.
\end{proof}

\begin{theorem}\label{th:15}
The $H$-type Lie algebras $\n_{r+8t_1+4t_2,r+8t_3+4t_2}$ and $\n_{r+8t_3+4t_2,r+8t_1+4t_2}$ are integral isomorphic for $r \in \{0,1,2\}\mod4$  and $t_1,t_2,t_3 \in \mathbb{N}\cup\{0\}$.
\end{theorem}
\begin{proof}
We prove by induction. The beginning of the induction is stated in Theorem~\ref{auto112244}.
Then the induction step is given by Theorem~\ref{condrs8008} and Theorem~\ref{condrs44}.
\end{proof}

%%%%%%%%%%%%%%%%%%%%%%%%

\section{Some non-isomorphic Lie algebras $\n_{r,s}$ and $\n_{s,r}$}\label{sec:nonisom}

%%%%%%%%%%%%%%%%%%%%%%%%

The behavior of the Lie algebras of block type, or those that admit the decomposition~\eqref{BD} are very special and as we saw in the previous section it is preserved under extension. The situation is much less predictable if the Lie algebra is not of block type and different situations can occur. In the present section we show one example of non isomorphic algebras: $\n_{2,3}$ and $\n_{3,2}$. We show also that, in a contrary to algebras $\n_{r,r}$, $r\in\{1,2,4\}\mod 4$ the pseudo $H$-type Lie algebra $\n_{3,3}$ does not admit an automorphism such that the restriction to the center is an anti-isometry. We also observe that our method does not allow to show that, for instance, $\n_{7,7}$, with $\vv_{7,7}=\vv_{3,3}\otimes\vv_{4,4}$, admits or does not admit an automorphism such that the restriction to the center is an anti-isometry.

%%%%%%%%%%%%%%%%%%%%%%%%

\subsection{Non-isomorphism of $\n_{3,2}$ and $\n_{2,3}$}

%%%%%%%%%%%%%%%%%%%%%%%%

First we introduce the integral basis of $\n_{3,2}$ and $\n_{2,3}$ which is essential for the proof of Theorem~\ref{noniso3223}.

We define an orthonormal basis of $\n_{3,2}$ by $\mathfrak B_{\z_{3,2}}=\{ Z_0,Z_1, Z_2, Z_3, Z_4\}$ and 
\begin{equation*}
\mathfrak B_{\vv_{3,2}}=\Big\{
\begin{array}{llllll}
&w_1=w, \quad & w_2=J_1w, \quad & w_3=J_2w, \quad & w_4=J_1J_2w, \\
&w_5=J_3w, \quad & w_6=J_4w, \quad & w_7=J_1J_3w, \quad & w_8=J_1J_4w,
\end{array}
\Big\}
\end{equation*}
for $J_1J_2J_3J_4w=J_0J_1J_2w=w$ with
$
\langle w_i, w_i  \rangle_{\mathfrak{v}_{3,2}}=\epsilon_i(4,4)$, $\la Z_{k}, Z_{k} \ra_{\z_{3,2}}=\epsilon_{k+1}(3,2)$.

\begin{table}[h]
\caption{Commutation relations on $\n_{3,2}$}
\centering
\begin{tabular}{| c | c | c | c | c | c | c | c | c |} 
\hline
 $[ row \,, col. ]$  & $w_1$ & $w_4$ & $w_7$ & $w_8$ & $w_2$ & $w_3$ & $w_5$ & $w_6$ \\
\hline
$w_1$ & $0$ & $-Z_0$ & $0$ & $0$ & $Z_1$ & $Z_2$ & $Z_3$ & $Z_4$\\
\hline
$w_4$ & $Z_0$ & $0$ & $0$ & $0$ & $Z_2$ & $-Z_1$ & $Z_4$ & $-Z_3$ \\
\hline
$w_7$ & $0$ & $0$ & $0$ & $-Z_0$ & $Z_3$ & $-Z_4$ & $Z_1$ & $-Z_2$\\
\hline
$w_8$ & $0$ & $0$ & $Z_0$ & $0$ & $Z_4$ & $Z_3$ & $Z_2$ & $Z_1$ \\
\hline
$w_2$ & $-Z_1$ & $-Z_2$  & $-Z_3$ & $-Z_4$ & $0$ & $-Z_0$ & $0$ & $0$ \\
\hline
$w_3$ & $-Z_2$ & $Z_1$ & $Z_4$ & $-Z_3$ & $Z_0$ & $0$ & $0$ & $0$ \\
\hline
$w_5$ & $-Z_3$ & $-Z_4$ & $-Z_1$ & $-Z_2$ & $0$ & $0$ & $0$ & $Z_0$ \\
\hline
$w_6$ & $-Z_4$ & $Z_3$ & $Z_2$ & $-Z_1$ & $0$ & $0$ & $-Z_0$ & $0$ \\
\hline
\end{tabular}\label{32}
\end{table} 

We define an orthonormal basis of $\n_{2,3}$ by $\mathfrak B_{\z_{2,3}}=\{ \bar Z_1, \bar Z_2, \bar Z_3, \bar Z_4,\bar Z_5\}$ and
\begin{equation*}
\mathfrak B_{\vv_{2,3}}=\Big\{
\begin{array}{llllll}
&\bar w_1=\bar w, \quad & \bar w_2=J_1\bar w, \quad & \bar w_3=J_2\bar w, \quad & \bar w_4=J_1J_2\bar w, \\
&\bar w_5=J_3\bar w, \quad & \bar w_6=J_4\bar w, \quad & \bar w_7=J_1J_3\bar w, \quad & \bar w_8=J_1J_4\bar w, 
\end{array}
\Big\}
\end{equation*}
for $J_1J_2J_3J_4\bar w=J_1J_4J_5\bar w=\bar w$ with $
\langle \bar w_i, \bar w_i  \rangle_{\mathfrak{v}_{2,3}}=\epsilon_i(4,4)$, $\la \bar Z_{k}, \bar Z_{k} \ra_{\z_{2,3}}=\epsilon_{k}(2,3)$.

\begin{table}[h]
\caption{Commutation relations on $\n_{2,3}$}
\centering
\begin{tabular}{| c | c | c | c | c | c | c | c | c |} 
\hline
 $[ row \,, col. ]$  & $\bar w_1$ & $\bar w_4$ & $\bar w_7$ & $\bar w_8$ & $\bar w_2$ & $\bar w_3$ & $\bar w_5$ & $\bar w_6$ \\
\hline
$\bar w_1$ & $0$ & $0$ & $0$ & $\bar Z_5$ & $\bar Z_1$ & $\bar Z_2$ & $\bar Z_3$ & $\bar Z_4$\\
\hline
$\bar w_4$ & $0$ & $0$ & $-\bar Z_5$ & $0$ & $\bar Z_2$ & $-\bar Z_1$ & $\bar Z_4$ & $-\bar Z_3$ \\
\hline
$\bar w_7$ & $0$ & $\bar Z_5$ & $0$ & $0$ & $\bar Z_3$ & $-\bar Z_4$ & $\bar Z_1$ & $-\bar Z_2$\\
\hline
$\bar w_8$ & $-\bar Z_5$ & $0$ & $0$ & $0$ & $\bar Z_4$ & $\bar Z_3$ & $\bar Z_2$ & $\bar Z_1$ \\
\hline
$\bar w_2$ & $-\bar Z_1$ & $-\bar Z_2$  & $-\bar Z_3$ & $-\bar Z_4$ & $0$ & $0$ & $0$ & $\bar Z_5$ \\
\hline
$\bar w_3$ & $-\bar Z_2$ & $\bar Z_1$ & $\bar Z_4$ & $-\bar Z_3$ & $0$ & $0$ & $\bar Z_5$ & $0$ \\
\hline
$\bar w_5$ & $-\bar Z_3$ & $-\bar Z_4$ & $-\bar Z_1$ & $-\bar Z_2$ & $0$ & $-\bar Z_5$ & $0$ & $0$ \\
\hline
$\bar w_6$ & $-\bar Z_4$ & $\bar Z_3$ & $\bar Z_2$ & $-\bar Z_1$ & $-\bar Z_5$ & $0$ & $0$ & $0$ \\
\hline
\end{tabular}\label{23}
\end{table}  

\begin{prop}\label{32surjectiveXX=0} The following is true.
\begin{itemize}
\item The linear map $\ad_X \colon \vv_{3,2} \to \z_{3,2}$ is surjective if and only if $\la X, X \ra_{\vv_{3,2}}\not=0$.
\item The linear map $\ad_X \colon \vv_{2,3} \to \z_{2,3}$ is surjective if and only if $\la X, X \ra_{\vv_{2,3}}\not=0$.
\end{itemize}
\end{prop}

\begin{proof}
First we note that $\ad_X$ is surjective for all $X \in \vv_{r,s}$ with $\la X, X \ra_{\vv_{r,s}}\not=0$ by Definition~\ref{def:general}, such that it suffices to prove that for $\la X, X \ra_{\vv_{3,2}}=0$, $\la X, X \ra_{\vv_{2,3}}=0$, respectively, the map $\ad_X$ is not surjective. 

We write $X=\sum_{i=1}^{8}{\lambda_i w_i}$ for $X\in \vv_{3,2}$ and define the representation matrix $M_X$ of $\ad_X$ with respect to the orthonormal basis  $\mathfrak B_{\n_{3,2}}$ by 
$$M_X=\begin{bmatrix} V^X_1 & V^X_4 & V^X_7 & V^X_8 & V^X_2 & V^X_3 & V^X_5 & V^X_6 \end{bmatrix},$$ where $V^X_i$ is the vector representation $\begin{pmatrix} \mu_{0i}^X \\ \vdots \\ \mu^X_{4i} \end{pmatrix}$ of $[X, w_i]= \sum_{k=0}^4{\mu^X_{ki} Z_k}$. The matrix $M_X$ for $\n_{3,2}$ is given by

\begin{eqnarray*}
\begin{pmatrix}
\lambda_4 & -\lambda_1 & \lambda_8 & -\lambda_7 & \lambda_3 & -\lambda_2 & -\lambda_6 & \lambda_5 \\
-\lambda_2 & \lambda_3 & -\lambda_5 & -\lambda_6 & \lambda_1 & -\lambda_4 & \lambda_7 & \lambda_8 \\
-\lambda_3 & -\lambda_2 & \lambda_6 & -\lambda_5 & \lambda_4 & \lambda_1 & \lambda_8 & -\lambda_7 \\
-\lambda_5 & \lambda_6 & -\lambda_2 & -\lambda_3 & \lambda_7 & \lambda_8 & \lambda_1 & -\lambda_4 \\
-\lambda_6 & -\lambda_5 & \lambda_3 & -\lambda_2 & \lambda_8 & -\lambda_7 & \lambda_4 & \lambda_1 
\end{pmatrix}.
\end{eqnarray*}
Note that $M_X$ is surjective if and only if $\det(M_XM_X^T)\not=0$ as $\rm {rank}(M_XM_X^T)=rank(M_X)$. The determinant of $M_XM_X^T$ is 
\begin{eqnarray*}
&&(\lambda_1^2+\lambda_2^2+\lambda_3^2+\lambda_4^2-\lambda_5^2-\lambda_6^2-\lambda_7^2-\lambda_8^2)^2(\lambda_1^2+\lambda_4^2+\lambda_7^2+\lambda_8^2+\lambda_2^2+\lambda_3^2+\lambda_5^2+\lambda_6^2)\\
&\times& \Big [ \lambda_1^4+\lambda_4^4+\lambda_7^4+2\lambda_7^2\lambda_8^2+\lambda_8^4+2\lambda_7^2\lambda_2^2+2\lambda_8^2\lambda_2^2+\lambda_2^4+2\lambda_7^2\lambda_3^2+2\lambda_8^2\lambda_3^2+2\lambda_2^2\lambda_3^2\\
&+&\lambda_3^4+2\lambda_7^2\lambda_5^2+2\lambda_8^2\lambda_5^2-2\lambda_2^2\lambda_5^2-2\lambda_3^2\lambda_5^2+\lambda_5^4+2(\lambda_7^2+\lambda_8^2-\lambda_2^2-\lambda_3^2+\lambda_5^2)\lambda_6^2 \\
&+&\lambda_6^4-8\lambda_1(\lambda_7\lambda_2\lambda_5+\lambda_8\lambda_3\lambda_5+\lambda_8\lambda_2\lambda_6-\lambda_7\lambda_3\lambda_6)+8\lambda_4(-\lambda_8\lambda_2\lambda_5+\lambda_7\lambda_3\lambda_5\\ &+&\lambda_7\lambda_2\lambda_6+\lambda_8\lambda_3\lambda_6) 
+2\lambda_4^2(-\lambda_7^2-\lambda_8^2+\lambda_2^2+\lambda_3^2+\lambda_5^2+\lambda_6^2)\\&+&2\lambda_1^2(\lambda_4^2-\lambda_7^2-\lambda_8^2+\lambda_2^2+\lambda_3^2+\lambda_5^2+\lambda_6^2)\Big].
\end{eqnarray*}
It follows that for all $X\in \vv_{3,2}$ with $\la X, X \ra_{\vv_{3,2}}= \lambda_1^2+\lambda_2^2+\lambda_3^2+\lambda_4^2-\lambda_5^2-\lambda_6^2-\lambda_7^2-\lambda_8^2=0$ the determinant
$\det(M_XM_X^T)$ vanishes. This finishes the proof for $\n_{3,2}$.

For $\n_{2,3}$ the matrix $M_X$ is given by
\begin{eqnarray*}
\begin{pmatrix}
-\lambda_2 & \lambda_3 & -\lambda_5 & -\lambda_6 & \lambda_1 & -\lambda_4 & \lambda_7 & \lambda_8 \\
-\lambda_3 & -\lambda_2 & \lambda_6 & -\lambda_5 & \lambda_4 & \lambda_1 & \lambda_8 & -\lambda_7 \\
-\lambda_5 & \lambda_6 & -\lambda_2 & -\lambda_3 & \lambda_7 & \lambda_8 & \lambda_1 & -\lambda_4 \\
-\lambda_6 & -\lambda_5 & \lambda_3 & -\lambda_2 & \lambda_8 & -\lambda_7 & \lambda_4 & \lambda_1 \\
-\lambda_8 & \lambda_7 & -\lambda_4 & \lambda_1 & -\lambda_6 & -\lambda_5 & \lambda_3 & \lambda_2
\end{pmatrix},
\end{eqnarray*}
and the determinant of $M_XM_X^T$ is given by
\begin{eqnarray*}
&&(\lambda_1^2+\lambda_2^2+\lambda_3^2+\lambda_4^2-\lambda_5^2-\lambda_6^2-\lambda_7^2-\lambda_8^2)^2
(\lambda_1^2+\lambda_4^2+\lambda_7^2+\lambda_8^2+\lambda_2^2+\lambda_3^2+\lambda_5^2+\lambda_6^2) 
\\
&\times& \Big [\lambda_1^4+\lambda_2^4+\lambda_3^4+2\lambda_3^2\lambda_4^2+\lambda_4^4-2\lambda_3^2\lambda_5^2+2\lambda_4^2\lambda_5^2
+\lambda_5^4-2\lambda_3^2\lambda_6^2+2\lambda_4^2\lambda_6^2+2\lambda_5^2\lambda_6^2+\lambda_6^4
\\
&-&
8\lambda_3\lambda_4\lambda_5\lambda_7 
+2\lambda_3^2\lambda_7^2-2\lambda_4^2\lambda_7^2+2\lambda_5^2\lambda_7^2+2\lambda_6^2\lambda_7^2+
\lambda_7^4+8\lambda_3\lambda_4\lambda_6\lambda_8+2\lambda_3^2\lambda_8^2-2\lambda_4^2\lambda_8^2
\\
&+&
2\lambda_5^2\lambda_8^2+
2\lambda_6^2\lambda_8^2+2\lambda_7^2\lambda_8^2
+\lambda_8^4-8\lambda_2\lambda_4(\lambda_6\lambda_7+\lambda_5\lambda_8)
\\
&+&
2\lambda_1^1(\lambda_2^2
+\lambda_3^2+\lambda_4^2+\lambda_5^2\lambda_6^2-\lambda_7^2-\lambda_8^2)+
2\lambda_2^2(\lambda_3^2+\lambda_4^2-\lambda_5^2-\lambda_6^2+\lambda_7^2+\lambda_8^2) 
\\
&+& 8\lambda_1(\lambda_3(\lambda_6\lambda_7+\lambda_6\lambda_8)+\lambda_2(-\lambda_5\lambda_7+\lambda_6\lambda_8))    \Big].
\end{eqnarray*}
Thus if $X\in \vv_{2,3}$ and $\la X, X \ra_{\vv_{2,3}}= \lambda_1^2+\lambda_2^2+\lambda_3^2+\lambda_4^2-\lambda_5^2-\lambda_6^2-\lambda_7^2-\lambda_8^2=0$, then
$\det(M_XM_X^T)=0$. This finishes the proof for $\n_{2,3}$.
\end{proof}

We stress that the results of Proposition~\ref{32surjectiveXX=0} and Proposition~\ref{33surjectiveXX=0} represent quite exceptional cases. Definition~\ref{def:general} does not imply that $\ad_X$ is not surjective for any $X \in \vv_{r,s}$ with $\la X, X \ra_{\vv_{r,s}}=0$. In the following we state a couple of lemmas illustrating this possibility.

\begin{lemma}\label{noextensionofx0sur} 
For any of the pseudo $H$-type Lie algebras $\n_{11,2}$, $\n_{7,6}$,$\n_{6,7}$ and $\n_{2,11}$ there exists $X$ in the corresponding space $\vv_{r,s}$ such that $\la X,X\ra_{\vv_{r,s}}=0$ but, nevertheless, the map $\ad_X$ is surjective.
\end{lemma}
\begin{proof}
 Recall that the pseudo $H$-type Lie algebras $\n_{11,2}$, $\n_{7,6}$,$\n_{6,7}$ and $\n_{2,11}$ are obtained from $\n_{3,2}$ and $\n_{2,3}$ by extensions.
We define $X=w_1 \otimes u_1+w_7 \otimes u_2\in \n_{11,2}$ where $w_1,w_7\in\mathfrak{B}_{\vv_{3,2}}$ and $u_1,u_2 \in \mathfrak{B}_{\vv_{8,0}}$ and note that $\la X, X \ra_{\vv_{11,2}}=0$. Then
\begin{eqnarray*}
[X, w_i \otimes u_1]&=& [w_1 \otimes u_1, w_i \otimes u_1] + [w_7 \otimes u_2, w_i \otimes u_1] \\
&=& \begin{cases}-[w_1 \,, w_i]_{\n_{11,2}} & \text{ for } \quad i=1,\dotso,6,8, \\ 
-[w_1, w_i]_{\n_{11,2}}-[u_2, u_1]_{\n_{11,2}} & \text{ for } \quad i=7,  \end{cases} \\
&=&-[w_1, w_i]_{\n_{11,2}}, \qquad\quad \qquad\qquad\text{ for } i=1, \dotso,8. 
\end{eqnarray*}  
Hence $\spn\{Z_1,Z_2,Z_3,Z_{12},Z_{13}\} \subset \rm{Image}(\ad_X)$. Furthermore,
\begin{eqnarray*}
[X, w_1 \otimes u_j]&=& [w_1 \otimes u_1, w_1 \otimes u_j] + [w_7 \otimes u_2, w_1 \otimes u_j] \\
&=& \begin{cases}[u_1, u_j]_{\n_{11,2}} & \text{ for } \quad j=1,3, \dotso,16, \\ 
[u_1 \,, u_j]_{\n_{11,2}}-[w_7, w_1]_{\n_{11,2}} & \text{ for } \quad j=2,  \end{cases} \\
&=&[u_1, u_j]_{\n_{11,2}}, \qquad\qquad\qquad\quad\ \text{ for } j=1, \dotso,16. 
\end{eqnarray*}  
Hence $\spn\{Z_4,\dotso,Z_{11}\} \subset \rm{Image}(\ad_X)$, i.e. the map $\ad_{X}$ is surjective. 

The proof for $\n_{7,6}$ and $\n_{6,7}$ is obtained analogously by replacing $u_1$ and $u_2$ by $y_1,y_6\in\mathfrak B_{4,4}$. For 
the proof for $\n_{2,11}$ we replace $u_1,u_2\in\mathfrak B_{8,0}$ by $v_1,v_2\in\mathfrak B_{0,8}$, respectively.
\end{proof}

\begin{lemma}
For any pseudo $H$-type Lie algebra $\n_{r,s}$ with $r,s\not=0$, satisfying~\eqref{BD}, there exists at least one $X \in \vv_{r,s}$ with $\la X, X \ra_{\vv_{r,s}}=0$ such that the map $\ad_X$ is surjective. 
\end{lemma}
\begin{proof}
We choose the basis vectors $w_i \in \A^+_{r,s}$ and $w_j \in \B^-_{r,s}$ and define $X=w_i+w_j$ such that $\la X, X \ra_{\vv_{r,s}}=0$. We note that the map $\ad_{w_i} \colon \mathfrak{V}_{w_i} \to \z_{r,s}$ and $\ad_{w_j} \colon \mathfrak{V}_{w_j} \to \z_{r,s}$ are surjective, where we denote by $\mathfrak{V}_{w_i}$ the orthogonal complement to the kernel of $\ad_{w_i}$. Therefore $\mathfrak{V}_{w_i} \subset \spn\{\B_{r,s}\}$ and $\mathfrak{V}_{w_j} \subset \spn\{\A_{r,s}\}$ as $[w_i, \A_{r,s}]=0$ and $[w_j, \B_{r,s}]=0$. 
It follows that
\begin{eqnarray*}
\spn\{[X, \A_{r,s}]\}&=&\spn\{[w_i, \A_{r,s}]+[w_j, \A_{r,s}]\}=\spn\{[w_j, \A_{r,s}]\}\supset[w_j, \mathfrak{V}_{w_j}]=\z_{r,s}, 
\\
\spn\{[X, \B_{r,s}]\}&=&\spn\{[w_i, \B_{r,s}]+[w_j, \B_{r,s}]\}=\spn\{[w_i, \B_{r,s}]\}\supset[w_i, \mathfrak{V}_{w_i}]=\z_{r,s}.
\end{eqnarray*}
Hence the map $\ad_X$ is surjective.
\end{proof}

\begin{theorem}\label{noniso3223}
The $H$-type Lie algebras $\n_{3,2}$ and $\n_{2,3}$ are not isomorphic.
\end{theorem}
\begin{proof}
We assume that there exists an isomorphism $\varphi_{3,2} \colon \n_{3,2} \to \n_{2,3}$ where the restriction $\varphi_{3,2} \vert _{\z_{3,2}} \colon \z_{3,2} \to \z_{2,3}$ is an anti-isometry. 
The adjoint operator $\ad_{w_i} \colon \mathfrak{V}_{w_i} \to \z_{3,2}$ is an isometry or anti-isometry by Definition~\ref{def:general} for any $w_i\in\mathfrak B_{\vv_{3,2}}$, i.e. 
$$\la \ad_{w_i}(X), \ad_{w_i}(X) \ra_{\z_{3,2}}=\la w_i, w_i \ra_{\vv_{3,2}} \la X, X \ra_{\vv_{3,2}}
$$ 
for all $X \in \mathfrak{V}_{w_i}$, where $\mathfrak{V}_{w_i}$ is the orthogonal complement to the kernel of $\ad_{w_i}$. As the map $\varphi_{3,2} \vert _{\z_{3,2}}$ is an anti-isometry, it follows that the composition 
$\varphi_{3,2}\circ \ad_{w_i} \colon \mathfrak{V}_{w_i} \to \z_{2,3}$ is an anti-isometry for $\la w_i \,, w_i \ra_{\vv_{3,2}}=1$ and is an isometry for $\la w_i \,, w_i \ra_{\vv_{3,2}}=-1$, hence
\begin{eqnarray*}
-\la w_i \,, w_i \ra_{\vv_{3,2}}\la w_j \,, w_j \ra_{\vv_{3,2}}&=&\la \varphi_{3,2}\circ \ad_{w_i}(w_j) \,, \varphi_{3,2}\circ \ad_{w_i}(w_j) \ra_{\z_{2,3}} \\
&=& \la [\varphi_{3,2}(w_i) \,, \varphi_{3,2}(w_j) ] \,, [\varphi_{3,2}(w_i) \,, \varphi_{3,2}(w_j) ] \ra_{\z_{2,3}}.
\end{eqnarray*}
As the map $\varphi_{3,2}\circ \ad_{w_i}$ is surjective and 
$$\varphi_{3,2}\circ \ad_{w_i}(w_j)=[\varphi_{3,2}(w_i), \varphi_{3,2}(w_j) ]=\ad_{\varphi_{3,2}(w_i)}(\varphi_{3,2}(w_j))
$$ it follows by Proposition~\ref{32surjectiveXX=0} that $\la \varphi_{3,2}(w_i) \,, \varphi_{3,2}(w_i)\ra_{\vv_{2,3}}\not=0$ for all $i=1, \dotso,8$. 

We recall that from Definition~\ref{def:general} it follows that for all $X \in \vv_{r,s}$ with $\la X \,, X \ra_{\vv_{r,s}} \not=0$ and $Y \in \mathfrak{V}_X$:
\begin{eqnarray*}
\la \ad_X(Y) \,, \ad_X(Y) \ra_{\z_{r,s}}=\la X \,, X \ra_{\vv_{r,s}} \la Y \,, Y \ra_{\vv_{r,s}},
\end{eqnarray*} 
hence
\begin{eqnarray}\label{nonisocond32}
&-&\la w_i \,, w_i \ra_{\vv_{3,2}}\la w_j \,, w_j \ra_{\vv_{3,2}}
\\
&=& \la \varphi_{3,2}(w_i)\,, \varphi_{3,2}(w_i) \ra_{\vv_{2,3}}\la  \varphi_{3,2}(w_j)  \,, \varphi_{3,2}(w_j)  \ra_{\vv_{2,3}}.\nonumber
\end{eqnarray}
We obtain the following relations for $w_1$ and $w_4$:
\begin{equation*}
\begin{array}{lllllllllllll}
\sign(\la \varphi_{3,2}(w_1)\,, \varphi_{3,2}(w_1) \ra_{\vv_{2,3}})&=&-\sign(\la \varphi_{3,2}(w_i)\,, \varphi_{3,2}(w_i) \ra_{\vv_{2,3}}), \quad &\text{ for } i=2,3,4, \\  
\sign(\la \varphi_{3,2}(w_1)\,, \varphi_{3,2}(w_1) \ra_{\vv_{2,3}})&=&\sign(\la \varphi_{3,2}(w_i)\,, \varphi_{3,2}(w_i) \ra_{\vv_{2,3}}), \quad &\text{ for } i=5,6, \\
\sign(\la \varphi_{3,2}(w_4)\,, \varphi_{3,2}(w_4) \ra_{\vv_{2,3}})&=&-\sign(\la \varphi_{3,2}(w_i)\,, \varphi_{3,2}(w_i) \ra_{\vv_{2,3}}), \quad &\text{ for } i=1,2,3, \\  
\sign(\la \varphi_{3,2}(w_4)\,, \varphi_{3,2}(w_4) \ra_{\vv_{2,3}})&=&\sign(\la \varphi_{3,2}(w_i)\,, \varphi_{3,2}(w_i) \ra_{\vv_{2,3}}), \quad &\text{ for } i=5,6. 
\end{array}
\end{equation*}
It implies that for $i=2,3$ 
\begin{eqnarray*}
-\sign(\la \varphi_{3,2}(w_i)\,, \varphi_{3,2}(w_i) \ra_{\vv_{2,3}})&=&\sign(\la \varphi_{3,2}(w_1)\,, \varphi_{3,2}(w_1) \ra_{\vv_{2,3}}) \\&=&-\sign(\la \varphi_{3,2}(w_4)\,, \varphi_{3,2}(w_4) \ra_{\vv_{2,3}})\\&=&\sign(\la \varphi_{3,2}(w_i)\,, \varphi_{3,2}(w_i) \ra_{\vv_{2,3}}).
\end{eqnarray*}
Hence $\la \varphi_{3,2}(w_i) \,, \varphi_{3,2}(w_i) \ra_{\vv_{2,3}}=0$ for $i=2,3$. This contradicts~\eqref{nonisocond32}: \\ $\la \varphi_{3,2}(w_i) \,, \varphi_{3,2}(w_i) \ra_{\vv_{2,3}}\not =0$ for $i=1, \dotso,8$, as the map $\ad_{\varphi_{3,2}(w_i)}$ is surjective.

Hence $\n_{3,2}$ is not isomorphic to $\n_{2,3}$.
\end{proof}

%%%%%%%%%%%%%%%%%%%%%%%

\subsection{Special features of the Lie algebra $\n_{3,3}$}\label{sub:33}

%%%%%%%%%%%%%%%%%%%%%%%%

We introduce an integral basis of $\n_{3,3}$ by $\mathfrak B_{\z_{3,3}}=\{Z_1,Z_2,Z_3,Z_4,Z_5,Z_6\}$ and 
\begin{equation*}
\mathfrak B_{\vv_{3,3}}=\Big\{
\begin{array}{lllllll}
&w_1=w, \quad & w_2=J_1w, \quad & w_3=J_2w, \quad & w_4=J_3w, \\
&w_5=J_1J_6w, \quad & w_6=J_6w, \quad & w_7=J_4w, \quad & w_8=J_5w,  
\end{array}
\Big\},
\end{equation*}
$$\text{ for } \quad J_2J_3J_4J_5w=J_1J_2J_5J_6w=J_1J_2J_3w=w,$$
$$\text{ and } \quad \langle w_i, w_i  \rangle_{\mathfrak{v}_{3,3}}=\epsilon_i(4,4), \qquad \la Z_{k}, Z_{k} \ra_{\z_{3,3}}=\epsilon_{k}(3,3).$$

\begin{table}[h]
\caption{Commutation relations on $\n_{3,3}$}
\centering
\begin{tabular}{| c | c | c | c | c | c | c | c | c |} 
\hline
 $[ row \,, col. ]$  & $w_1$ & $w_2$ & $w_5$ & $w_6$ & $w_3$ & $w_4$ & $w_7$ & $w_8$ \\
\hline
$w_1$ & $0$ & $Z_1$ & $0$ & $Z_6$ & $Z_2$ & $Z_3$ & $Z_4$ & $Z_5$\\
\hline
$w_2$ & $-Z_1$ & $0$ & $-Z_6$ & $0$ & $-Z_3$ & $Z_2$ & $-Z_5$ & $Z_4$ \\
\hline
$w_5$ & $0$ & $Z_6$ & $0$ & $Z_1$ & $Z_5$ & $Z_4$ & $Z_3$ & $Z_2$\\
\hline
$w_6$ & $-Z_6$ & $0$ & $-Z_1$ & $0$ & $Z_4$ & $-Z_5$ & $Z_2$ & $-Z_3$ \\
\hline
$w_3$ & $-Z_2$ & $Z_3$  & $-Z_5$ & $-Z_4$ & $0$ & $-Z_1$ & $Z_6$ & $0$ \\
\hline
$w_4$ & $-Z_3$ & $-Z_2$ & $-Z_4$ & $Z_5$ & $Z_1$ & $0$ & $0$ & $-Z_6$ \\
\hline
$w_7$ & $-Z_4$ & $Z_5$ & $-Z_3$ & $-Z_2$ & $-Z_6$ & $0$ & $0$ & $Z_1$ \\
\hline
$w_8$ & $-Z_5$ & $-Z_4$ & $-Z_2$ & $Z_3$ & $0$ & $Z_6$ & $-Z_1$ & $0$ \\
\hline
\end{tabular}\label{33}
\end{table} 

\begin{prop}\label{33surjectiveXX=0}
The linear map $\ad_X \colon \vv_{3,3} \to \z_{3,3}$ is surjective if and only if $\la X \,, X \ra_{\vv_{3,3}}\not=0$ for $X \in \vv_{3,3}$.
\end{prop}
\begin{proof}
We use similar arguments as in the proof of Proposition~\ref{32surjectiveXX=0}.
The matrix $M_X$ that we calculate by using Table~\ref{33} is given by
\begin{eqnarray*}
\begin{pmatrix}
-\lambda_2 & \lambda_1 & -\lambda_6 & \lambda_5 & \lambda_4 & -\lambda_3 & -\lambda_8 & \lambda_7 \\
-\lambda_3 & -\lambda_4 & -\lambda_8 & -\lambda_7 & \lambda_1 & \lambda_2 & \lambda_6 & \lambda_5 \\
-\lambda_4 & \lambda_3 & -\lambda_7 & \lambda_8 & -\lambda_2 & \lambda_1 & \lambda_5 & -\lambda_6 \\
-\lambda_7 & -\lambda_8& -\lambda_4 & -\lambda_3&\lambda_6&\lambda_5&\lambda_1&\lambda_2\\
-\lambda_8&\lambda_7&-\lambda_3&\lambda_4&\lambda_5&-\lambda_6&-\lambda_2&\lambda_1\\
-\lambda_6&\lambda_5&-\lambda_2&\lambda_1&-\lambda_7&\lambda_8&\lambda_3&-\lambda_4
\end{pmatrix}.
\end{eqnarray*}
The determinant of $M_XM_X^T$ has the form
\begin{eqnarray*}
&&(\lambda_1^2+\lambda_2^2+\lambda_3^2+\lambda_4^2-\lambda_5^2-\lambda_6^2-\lambda_7^2-\lambda_8^2)^4\\
&\times&((\lambda_1-\lambda_5)^2+(\lambda_2-\lambda_6)^2+(\lambda_4-\lambda_7)^2+(\lambda_3-\lambda_8)^2) \\
&\times&((\lambda_1+\lambda_5)^2+(\lambda_2+\lambda_6)^2+(\lambda_4+\lambda_7)^2+(\lambda_3+\lambda_8)^2).
\end{eqnarray*}
Thus the map $\ad_X$ is surjective if and only if $\la X, X \ra_{\vv_{3,3}}\neq 0$.
\end{proof}
\begin{theorem}\label{nonauto33}
There does not exist an automorphism $\varphi_{3,3}$ of $\n_{3,3}$ such that the restriction to the center $\varphi \vert_{\z_{3,3}}$ is an anti-isometry.
\end{theorem}
\begin{proof}
By repeating the arguments of the proof of Theorem~\ref{noniso3223}, we obtain equation~\eqref{nonisocond32}. This implies the relations
\begin{equation*}
\begin{array}{lllllll}
\sign(\la \varphi_{3,3}(w_1)\,, \varphi_{3,3}(w_1) \ra_{\vv_{3,3}})&=&-\sign(\la \varphi_{3,3}(w_i)\,, \varphi_{3,3}(w_i) \ra_{\vv_{3,3}}), \quad &\text{ for } i=2,3,4, \\  
\sign(\la \varphi_{3,3}(w_1)\,, \varphi_{3,3}(w_1) \ra_{\vv_{3,3}})&=&\sign(\la \varphi_{3,3}(w_i)\,, \varphi_{3,3}(w_i) \ra_{\vv_{3,3}}), \quad &\text{ for } i=6,7,8, \\
\sign(\la \varphi_{3,3}(w_2)\,, \varphi_{3,3}(w_2) \ra_{\vv_{3,3}})&=&-\sign(\la \varphi_{3,3}(w_i)\,, \varphi_{3,3}(w_i) \ra_{\vv_{3,3}}), \quad &\text{ for } i=1,3,4, \\  
\sign(\la \varphi_{3,3}(w_2)\,, \varphi_{3,3}(w_2) \ra_{\vv_{3,3}})&=&\sign(\la \varphi_{3,3}(w_i)\,, \varphi_{3,3}(w_i) \ra_{\vv_{3,3}}), \quad &\text{ for } i=5,7,8.
\end{array}
\end{equation*}
Thus, for $i=3,4$ 
\begin{eqnarray*}
-\sign(\la \varphi_{3,3}(w_i)\,, \varphi_{3,3}(w_i) \ra_{\vv_{3,3}})&=&\sign(\la \varphi_{3,3}(w_1)\,, \varphi_{3,3}(w_1) \ra_{\vv_{3,3}}) \\&=&-\sign(\la \varphi_{3,3}(w_2)\,, \varphi_{3,3}(w_2) \ra_{\vv_{3,3}})\\&=&\sign(\la \varphi_{3,3}(w_i)\,, \varphi_{3,3}(w_i) \ra_{\vv_{3,3}}).
\end{eqnarray*}
Hence $\la \varphi_{3,3}(w_i), \varphi_{3,3}(w_i) \ra_{\vv_{3,3}}=0$ for $i=3,4$. This contradicts to the fact that  $\la \varphi_{3,3}(w_i), \varphi_{3,3}(w_i) \ra_{\vv_{3,3}}\not =0$ for $i=1, \dotso,8$, as the map $\ad_{\varphi_{3,3}(w_i)}$ is surjective.

Hence there does not exist an automorphism $\varphi_{3,3}$ of $\n_{3,3}$ such that $\varphi \vert_{\z_{3,3}}$ is an anti-isometry.
\end{proof}
\begin{prop}
For the pseudo $H$-type Lie algebras $\n_{11,3}$, $\n_{7,7}$ and $\n_{3,11}$ there exists $X$ in respective $\vv_{r,s}$ such that $\ad_X$ is surjective.
\end{prop}
\begin{proof}
We replace in the proof of Lemma~\ref{noextensionofx0sur} $w_1,w_7 \in \mathfrak B_{\vv_{2,3}}$ by $w_1,w_5 \in \mathfrak B_{\vv_{3,3}}$.
\end{proof}

%%%%%%%%%%%%%%%%%%%%%%%%%%%%

\section{Strongly bracket generating property}\label{strbracksec}

%%%%%%%%%%%%%%%%%%%%%%%%%%%%

In this section we study the bracket generating property of the pseudo $H$-type Lie algebras. For that purpose we state an equivalent definition of the pseudo $H$-type algebras $\n_{r,s}$, which is related to the definition of the strongly bracket generating property.

\begin{defi}\label{def:strongly}
Let $\n_{r,s}=\vv_{r,s} \oplus \z_{r,s}$ be a pseudo $H$-type Lie algebra. We call a vector space $\vv_{r,s}$ strongly bracket generating if 
for any non-zero $v\in \vv_{r,s}$ the linear map 
$
\ad_v=[ v, \cdot ]\colon \vv_{r,s}\to\z_{r,s}
$ is surjective, i.e. $\spn \{ \vv_{r,s} , [v, \vv_{r,s}]\} = \n_{r,s}$ for all $v\in \vv_{r,s}\setminus\{0\}$. We say in this case that the pseudo $H$-type Lie algebra $\n_{r,s}$ has strongly bracket generating property.
\end{defi}

Let $N_{r,s}$ be the Lie group, corresponding to the pseudo $H$-type Lie algebra $\n_{r,s}$ and let $\mathcal H$ be the left translation of the vector space $\vv_{r,s}$. If $\vv_{r,s}$ is strongly bracket generating, then the left invariant distribution $\mathcal H$ is strongly bracket generating in a sense that  
$
\spn \{ \mathcal H , [X, \mathcal H]\}=TN_{r,s}
$
for any smooth non-zero section $X$ of the distribution $\mathcal H$.

Even the strongly bracket generating property seems to be just of interest from a geometrical point of view, it actually has a close relation to an equivalent definition of pseudo $H$-type algebras, which can be seen in the following subsection.

%%%%%%%%%%

\subsection{An  equivalent definition of pseudo $H$-type Lie algebras}\label{sec:equivPsGen}

%%%%%%%%%%

In this subsection we define general $H$-type algebras~\cite{GodoyKorolkoMarkina} and prove that they are equivalent to pseudo $H$-type algebras.

Let $\n=\big(\vv\oplus_{\bot}\z,[\cdot\,,\cdot],\langle\cdot\,,\cdot\rangle_{\n}=\langle\cdot\,,\cdot\rangle_{\vv}+\langle\cdot\,,\cdot\rangle_{\z}\big)$ be an arbitrary $2$-step nilpotent Lie algebra with center $\z$ and a non-degenerate scalar product $\langle\cdot\,,\cdot\rangle_{\n}$. We write $\ad_v\colon\vv\to\z$ for the linear map given by $\ad_vw=[v,w]$. We assume that the restriction of the scalar product $\langle\cdot\,,\cdot\rangle_{\vv}$ onto the subspace $\ker(\ad_v)\subset\vv$ is non-degenerate and denote its orthogonal complement with respect to $\langle\cdot\,,\cdot\rangle_{\vv}$ by $\mathfrak{V}_v$, which is also non-degenerate. Thus the restricted map $\ad_v \colon \mathfrak{V}_v \to \z$ is injective. 

\begin{defi}\label{def:general}\cite{GodoyKorolkoMarkina}
A two-step nilpotent Lie algebra $\n=\big(\vv\oplus_{\bot}\z,[\cdot\,,\cdot],\langle\cdot\,,\cdot\rangle_{\n}=\langle\cdot\,,\cdot\rangle_{\vv}+\langle\cdot\,,\cdot\rangle_{\z}\big)$ is of general $H$-type if 
$
ad_v\colon\mathfrak{V}_v \to \z
$ is a surjective isometry for all $v\in \vv$ with $\langle v, v \rangle_{\vv}=1$ and a surjective anti-isometry for all $v\in\vv$ with $\langle v,v \rangle_{\vv}=-1$. 
\end{defi}

Let $(U,\langle\cdot\,,\cdot\rangle_{U})$, $(V,\langle\cdot\,,\cdot\rangle_{V})$ be two vector spaces with corresponding non-degenerate quadratic forms, written as bi-linear symmetric forms, or scalar products. 

\begin{defi}
A bilinear map $\mu\colon U\times V\to V$ is called a composition of the scalar products $\langle\cdot\,,\cdot\rangle_{U}$ of $U$ and $\langle\cdot\,,\cdot\rangle_{V}$ of $V$ if the equality 
\begin{equation}\label{eq:composition}
\langle\mu(u,v),\mu(u,v)\rangle_{V}=\langle u,u\rangle_{U}\langle v,v\rangle_{V}
\end{equation}
holds for any $u\in U$ and $v\in V$.
\end{defi}

We assume that there is $u_0\in U$ such that $\langle u_0,u_0\rangle_{U}=1$ and $\mu(u_0,v)=v$. This can always be done by normalization procedure of quadratic forms, see~\cite{Lam}. Let us denote by $\mathcal Z$ the orthogonal complement to the non-degenerate space $\spn\{u_0\}$ and by $J$  the restriction of $\mu$ to $\mathcal Z$, thus $J\colon \mathcal Z\times V\to V$. The map $J$ is skew-adjoint in the sense that
$
\langle J(Z,v),v'\rangle_{V}=-\langle v,J(Z,v')\rangle_{V}
$
for any $Z\in \mathcal Z$ and $v,v'\in V$. Therefore, the map $J$ can be used to define a Lie algebra structure on $\n=\mathcal Z\oplus V$ by $\langle J(Z,v),v' \rangle_V = \langle Z, [v,v']\rangle_{\mathcal Z}$. The obtained Lie algebra is a general $H$-type algebra, see~\cite[Theorem 1]{GodoyKorolkoMarkina}.
Now, rephrasing Definition~\ref{def:pseudo} of a pseudo $H$-type algebra we can say that a two-step nilpotent Lie algebra is a pseudo $H$-type algebra if the map $J$ defined by~\eqref{def:J_Z} is the restriction on the center $\mathcal Z=\z$ of a composition of corresponding quadratic forms for vector spaces $V=\vv$, $U=\spn\{u_0\}\oplus_{\bot}\mathcal Z$. 

\begin{theorem}
Definitions~\ref{def:pseudo} and~\ref{def:general} are equivalent.
\end{theorem}
\begin{proof}
Let us prove that Definition~\ref{def:general} implies Definition~\ref{def:pseudo}. It was shown in~\cite[Theorem 1]{GodoyKorolkoMarkina} that any general $H$-type algebra $\n=\big(\vv\oplus_{\bot}\z,[\cdot\,,\cdot],\langle\cdot\,,\cdot\rangle_{\n}=\langle\cdot\,,\cdot\rangle_{\vv}+\langle\cdot\,,\cdot\rangle_{\z}\big)$ defines a composition of the quadratic form $\langle\cdot\,,\cdot\rangle_{\vv}$ and another quadratic form whose restriction on $\z$ coincides with $\langle\cdot\,,\cdot\rangle_{\z}$. Particularly, it implies~\eqref{eq:J_Zcomposition} and therefore a general $H$-type algebra is a pseudo $H$-type algebra.

Now we assume that we are given a pseudo $H$-type algebra $\n=\big(\vv\oplus_{\bot}\z,[\cdot\,,\cdot],\langle\cdot\,,\cdot\rangle_{\n}=\langle\cdot\,,\cdot\rangle_{\vv}+\langle\cdot\,,\cdot\rangle_{\z}\big)$ with center $\z$. Let us fix $v\in \vv$ with $\langle v,v \rangle_{\vv} = \pm 1$. We need to show that $\ad_v\colon \vv\to\z$ is a surjective (anti-)isometry.
The following equation is true 
\begin{equation*}
\langle Z , \ad_v(J_{Z'}v)\rangle_{\z} 
=
\langle Z , [v , J_{Z'}v]\rangle_{\z} = \langle J_Zv , J_{Z'}v\rangle_{\vv} 
= 
\langle Z, Z'\rangle_{\z} \langle v,v \rangle_{\vv} 
= \pm \langle Z , Z' \rangle_{\z},
\end{equation*}
for all $Z, Z' \in \z$
by formula~\eqref{eq_depol}. We use both notations $J(Z,v)$ and $J_Zv$. 
This implies $\ad_{v}(J_{Z'}v)=\pm Z'$ for all $Z' \in \z$ by the non-degenerate property of the scalar product $\langle\cdot\,,\cdot\rangle_{\z}$. Since  
$
\langle J_{Z'}v,w \rangle_{\vv} = \langle Z', [v,w] \rangle_{\z} = \langle Z', 0 \rangle_{\z} =0,
$ for all $w \in \ker(\ad_v)$,
it follows that $J_{Z'}v \in \mathfrak{V}_v=(\ker(\ad_v))^{\bot}$. We showed that $\ad_v$ is surjective. 

To prove that $\ad_v$ is an isometry for $\langle v,v \rangle_{\vv} = 1$ and an anti-isometry for $\langle v,v \rangle_{\vv} = -1$, we exhibit that the maps 
$\ad_v\colon\mathfrak{V}_v\to\z$ and $J_{(\cdot)}v\colon \z\to\mathfrak{V}_v$ are inverse and then equality~\eqref{eq:J_Zcomposition} implies the isometry and anti-isometry properties.
Let us assume that $\langle v,v \rangle_{\vv} = 1$. We proved that $\ad_v\colon\mathfrak{V}_v\to\z$ is bijective, thus the image of $J_{(\cdot)}v$ belongs to $\mathfrak{V}_v$, and $\ad_{v}(J_{(\cdot)}v)=\Id_{\z}$, where $\Id_{\z}$ is the identity map on $\z$. 

We claim that the map $J_{(\cdot)}v\colon \z\to\mathfrak V_v$ is bijective. Indeed if we assume that $J_{(\cdot)}v$ is not surjective, then there is $w\in \mathfrak V_v$ that is not in the image of $J_{(\cdot)}v$. Let $\ad_v(w)=Z\in \z$, then $\ad_v(J_Zv)=Z$ which implies $w=J_Zv$ by injectivity of $\ad_v$ and leads to contradiction. 

If we now assume that $J_{(\cdot)}v$ is not injective, then we find $Z',Z''\in \z$, $Z'\neq Z''$, such that $J_{Z'}(v)=J_{Z''}(v)$. But in this case
$
Z'=\ad_X(J_{Z'}v)=\ad_v(J_{Z''}v)=Z''
$
by bijectivity of $\ad_v$ and we again get a contradiction. The proof for $\langle v,v\rangle_{\vv}=-1$ is analogous and we conclude that $\ad_v$ and $J_{(\cdot)}v$ are inverse maps to each other. Equality~\eqref{eq:J_Zcomposition} becomes 
$$
\langle J_Zv,J_Zv\rangle_{\vv}=\langle Z,Z\rangle_{\z}\ \text{for}\ \langle v,v\rangle_{\vv}=1,
\ \text{and}\ \ 
\langle J_Zv,J_Zv\rangle_{\vv}=-\langle Z,Z\rangle_{\z}\ \text{for}\ \langle v,v\rangle_{\vv}=-1,
$$
that shows the (anti-)isometry property of the map $J_{(\cdot)}v\colon \z\to\mathfrak V_v$ and its inverse $\ad_v\colon\mathfrak{V}_v\to\z$.
\end{proof}

%%%%%%%%%%%%%%%%%%%%%%%%%%%%%%%%%%%%%%%%

\subsection{Bracket generating property of pseudo $H$-type algebras}

%%%%%%%%%%%%%%%%%%%%%%%%%%%%%%%%%%%%%%%%

\begin{theorem}\label{stbracket}
The pseudo $H$-type Lie algebras $\n_{r,s}$ with $r=0$ or $s=0$ have the strongly bracket generating property.
\end{theorem}

\begin{proof}
Let $s=0$. This implies that $\langle v, v \rangle_{\vv_{r,0}} >0$  for all $v \in \vv_{r,0}$ with $v \not =0$. Definition~\ref{def:general} yields that $ad_v$ is surjective, i.e. $\vv_{r,0}$ is strongly bracket generating.

Let now $r=0$. Recall that $\vv_{0,s}$ is a neutral space, i.e. $\langle \cdot \,, \cdot \rangle_{\vv_{0,s}}$ has index $(l,l)$ and we can identify $\vv_{0,s}$ with $\mathbb{R}^{l,l}$. This implies that there exists elements $v\in \vv_{0,s}$, $v \not =0$, with $\langle v,v \rangle_{l,l}=0$. According to Definition~\ref{def:general} we need only to show that $\ad_v\colon \vv_{r,s}\to\z_{r,s}$ is surjective for vectors with $\langle v,v \rangle_{l,l}= 0$, since for all other vectors the adjoint map is surjective. 

We define the orthonormal basis $\{w_1, \dotso,w_{2l}\}$ of $\vv_{0,s}$ with $\la w_i \,, w_i \ra_{\vv_{0,s}}=\epsilon_i(l,l)$ and fix an arbitrary $v\in \vv_{0,s}$ with $\langle v,v \rangle_{l,l}=0$ and $v=\sum_{i=1}^{2l}{\lambda_i w_i}$. We split $v$ in the form $v=v^++v^-$, with $v^+=\sum_{i=1}^{l}{\lambda_i w_i}$, $v^-=\sum_{i=l+1}^{2l}{\lambda_i w_i}$ and $\langle v^+ \,, v^+ \rangle_{l,l}=-\langle v^- \,, v^- \rangle_{l,l} >0$ and $\langle v^+ \,, v^- \rangle_{l,l}=0$. We note that $[w_i \,, w_j]=0$ if $i,j=1, \dotso,l$ or $i,j=l+1, \dotso,2l$ as
\begin{eqnarray*}
\la [w_i \,, w_j] \,, [w_i \,, w_j] \ra_{\z_{0,s}} \geq 0 , \quad \text{ for } i,j=1, \dotso,l, \quad \text{ or } \quad  i,j=l+1,\dotso,2l.
\end{eqnarray*}
Hence $\z_{0,s}=\ad_{w_i}(\vv_{0,s})=\ad_{w_i}(\spn\{w_1, \dotso,w_l\})$ for $i=l+1, \dotso,2l$. It follows that
\begin{eqnarray*}
[v \,, \spn\{w_1, \dotso,w_l\}]=[v^- \,,  \spn\{w_1, \dotso,w_l\}]=\z_{0,s}.
\end{eqnarray*}
Hence $\ad_v$ is surjective, i.e. the pseudo $H$-type algebras $\n_{0,s}=\vv_{0,s}  \oplus \z_{0,s}$, where $s>0$, have strongly bracket generating property.

%If $w\in\vv_{0,s}$ with $\langle w,w \rangle_{l,l}>0$, then $\langle\ad_{v^+}w,\ad_{v^+}w \rangle_{0,s}\geq 0$ by the isometry property of the adjoint map $\ad_{v^+}$. From the other hand $\langle\ad_{v^+}w,\ad_{v^+}w \rangle_{0,s}\leq 0$, since the scalar product $\langle \cdot \,, \cdot \rangle_{0,s}$ is negative definite. We conclude that $\ad_{v^+}w=0$ and therefore $w\in\ker(\ad_{v^+})$ for any vector $w\in\vv_{0,s}$ with
%$\langle w,w \rangle_{l,l}>0$, or in other words
%$$
%\z_{0,s}=\spn\{\ad_{v^+}w\mid\ w\in\vv_{0,s}\quad\text{with}\quad
%\langle w,w \rangle_{l,l}\leq 0\}
%$$
%since the map $\ad_{v^+}\colon \vv_{0,s}\to\z_{0,s}$ is surjective. Thus we conclude 
%$$
%\z_{0,s} = \{\ad_{v^+}w\mid\ w\in \vv_{0,s},\ \langle w,w \rangle_{l,l} \leq 0 \} \subset \{\ad_{v}u\mid\ u \in \vv_{0,s} \} \subset \z_{0,s}.
%$$
%Hence $\ad_v$ is surjective, i.e. the pseudo $H$-type algebras $\n_{0,s}=\vv_{0,s}  \oplus \z_{0,s}$, where $s>0$, have strongly bracket generating property.
\end{proof}

\begin{theorem}\label{notstbracket}
The pseudo $H$-type Lie algebras $\n_{r,s}$ with $r,s\not=0$ do not have the strongly bracket generating property.
\end{theorem}
\begin{proof}
We assume that $\n_{r,s}=\vv_{r,s} \oplus \z_{r,s}$ with $r,s\not=0$ has the strongly bracket generating property, i.e. for all $v \in \vv_{r,s}$: $[v \,, \vv_{r,s}]=\z_{r,s}$ and we show that it contradicts to the presence of nullvectors in the scalar product space $(\z_{r,s},\la \cdot \,, \cdot \ra_{r,s})$. The non-degenerate property of the indefinite scalar-product $\la \cdot \,, \cdot \ra_{r,s}$ implies that for all $v \in \vv_{r,s}$ and for all $Z \in \z_{r,s}$ there exists $v_Z \in \vv_{r,s}$ such that 
\begin{eqnarray*}
\la [v \,, v_Z] \,, Z \ra_{r,s} \not =0 \qquad {\stackrel{\eqref{eq:def_J}}{\Longleftrightarrow}} \qquad \la J_Z (v) \,, v_Z \ra_{\vv_{r,s}} \not =0.
\end{eqnarray*} 
It follows that $J_Z(v)\not = 0$ for all $v \in \vv_{r,s}$ and for all $Z \in \z_{r,s}$, i.e. $\ker \{J_Z\} = \{0\}$ for all $Z \in \z_{r,s}$. But there exist elements $Z_0 \in \z_{r,s}$ such that $\la Z_0 \,, Z_0 \ra_{r,s}=0$ as $r,s \not =0$. This implies that $J_{Z_0}^2=0$ which is equivalent to $\ker \{ J_{Z_0} \} \not = \{0\}$. This is a contradiction, hence the pseudo $H$-type Lie algebras $\n_{r,s}$ with $r,s\not=0$ do not have the strongly bracket generating property.
\end{proof}

%%%%%%%%%%%%%%%%%%%%%%%%%%%

\section{Non-isomorphism properties for pseudo $H$-type groups in general position}\label{sec7}

%%%%%%%%%%%%%%%%%%%%%%%%%%%

In this section we discuss the possible extension of our results to pseudo $H$-type Lie algebras, constructed from non-minimal admissible Clifford modules. Here we need to distinguish two essentially different situations. The first one when the irreducible module is unique and other one when there are two non-equivalent irreducible modules. We introduce new notations.

1. Let the Clifford algebras $\Cl_{r,s}$ admit only one (up to equivalence) irreducible module and we write $\vv_{r,s}$ for the minimal admissible module, that could be a direct sum of two irreducible modules. This situation occurs when $r-s\neq 3 \mod 4$. Any non-minimal admissible $\Cl_{r,s}$-module $\vv$ is isomorphic (and isometric) to the direct sum of minimal admissible modules $\vv_{r,s}$, see~\cite{FurutaniIrina,LawMich}:
$$
\vv=\vv_{r,s}(\mu)\cong\oplus^{\mu}\vv_{r,s}.
$$
Here and further on we use the notation $\vv_{r,s}(\mu)$ for the $\mu$-fold direct sum of minimal admissible modules $\vv_{r,s}$. Thus the argument $\mu$ shows how many equivalent (in the sense of the representation theory) minimal admissible modules contains the sum. The lower index, as previously, indicates the index of the metric of the generating space for the Clifford algebra. 

2. If $r-s=3 \mod 4$, then the Clifford algebra $\Cl_{r,s}$ admits two non-equivalent Clifford modules. We write $\vv_{r,s}^1$ and $\vv_{r,s}^2$ for the minimal admissible modules. Recall, that in this case each of the admissible modules $\vv_{r,s}^l$, $l=1,2$ is either irreducible, or the direct sum of two equivalent irreducible modules, where the representation map is changed appropriately~\cite{Ciatti}. We emphasize that a minimal admissible module $\vv_{r,s}^l$, $l=1,2$, can not be a direct sum of two non-equivalent irreducible modules.  In this case a non-minimal admissible $\Cl_{r,s}$-module $\vv$ is isomorphic to 
$$
\vv=\vv_{r,s}(\mu,\nu)\cong(\oplus^{\mu}\vv_{r,s}^1)\bigoplus(\oplus^{\nu}\vv_{r,s}^2)
$$
for some positive integers $\mu,\nu$ which show the number of equivalent and non-equivalent minimal admissible modules contained in the admissible module $\vv_{r,s}(\mu,\nu)$. To unify the notation we always write $\vv_{r,s}(\mu,\nu)$, where $\nu=0$ if $r-s\neq 3 \mod 4$ and $\nu$ can be different from zero in the case $r-s=3 \mod 4$. According to this new notation we also write $\n_{r,s}(\mu,\nu)$ for a pseudo $H$-type Lie algebra in the case if it is isomorphic to the direct sum $\vv_{r,s}(\mu,\nu)\oplus\z_{r,s}$.

Results of~\cite{FurutaniIrina} imply also that a non-minimal admissible module $(\vv_{r,s}(\mu,\nu),\langle\cdot\,,\cdot\rangle_{\vv_{r,s}(\mu,\nu)})$ of the Clifford algebra $\Cl_{r,s}$ is given as an orthogonal sum of $n$-dimensional minimal admissible modules $\vv_{r,s}=(\vv_{r,s},\langle\cdot\,,\cdot\rangle_{\vv_{r,s}})$, where each scalar product $\langle\cdot\,,\cdot\rangle_{\vv_{r,s}}$ is the restriction of $\langle\cdot\,,\cdot\rangle_{\vv_{r,s}(\mu,\nu)}$ on the corresponding copy of the vector space $\vv_{r,s}$. To describe the Lie bracket on $\n_{r,s}(\mu,\nu)$ we proceed as follows. Let $\{Z_1, \dotso, Z_m\}$ be an orthonormal basis of $\z_{r,s}$. We denote a basis of $j$-term in the sum $\oplus^{\mu}\vv_{r,s}^{l}$, $l=1,2$ by $\{v_{1j}^l, \dotso,v_{nj}^l\}$ with structure constants $(A_{ip}^k)^l_j$. For the sum $\n_{r,s}(\mu,\nu)=\left ( (\oplus_{j=1}^{\mu} (\vv_{r,s}^1)_j)\bigoplus (\oplus_{q=1}^{\nu} (\vv_{r,s}^2)_q) \right ) \oplus\z_{r,s}$ we choose the basis
\begin{equation}\label{eq:basis}
\{\left.v_{ij}^1,v_{pq}^2,Z_k  \right| i,p=1,\dotso,n,\ k=1,\dotso,r+s, \ j=1,\dotso,\mu,\ q=1,\dotso,\nu\}.
\end{equation} 
 The Lie bracket on $\n_{r,s}(\mu,\nu)$ with respect to this basis is given by 
\begin{eqnarray}\label{noncom}
[w_{ij}^{l_1} \,, w_{pq}^{l_2}]=\delta_{l_1l_2} \delta_{jq} \sum_{k=1}^{r+s}(A_{ip}^k)^{l_t}_j Z_k,\quad t=1,2. 
\end{eqnarray}

The bilinear maps $J^l_j \colon \z_{r,s} \times  (\vv_{r,s}^l)_j \to (\vv_{r,s}^l)_j$, $l=1,2$ are defined by a representation of $\Cl_{r,s}$ over $(\vv_{r,s}^l)_j$ and are extended to
$\tilde J \colon \z_{r,s} \times \vv_{r,s}(\mu,\nu) \to \vv_{r,s}(\mu,\nu)$ by 
\begin{eqnarray*}
\tilde J:=\begin{pmatrix} J^1_1 & 0 & \cdots & 0 \\ 0 & \ddots & \ddots & \vdots \\ \vdots & \ddots & \ddots & 0 \\ 0 & \cdots & 0 & J^2_{\nu} \end{pmatrix}.
\end{eqnarray*}
Then the operator $\tilde J\colon \z_{r,s}\to\End(\vv_{r,s}(\mu,\nu))$ satisfies 
$
\la \tilde J_{Z} v, w \ra_{\vv_{r,s}(\mu,\nu)} = \la Z, [v, w] \ra_{\z_{r,s}}, 
$
for all $Z \in \z_{r,s}$, $v,w \in \vv_{r,s}(\mu,\nu)$
and can be extended to the representation of $\Cl_{r,s}$ over $\vv_{r,s}(\mu,\nu)$. We can assume, by a change of coordinates, without loss of generality, that $J^1_1=\dotso=J^1_{\mu}$ and $J^2_1=\dotso=J^2_{\nu}$. If a Clifford algebra $\Cl_{r,s}$ admits only one irreducible representation, then the notation simplifies due to the absence of upper indices $l=1,2$.

We start from general observations where the first one follows easily from the dimension argument.

\begin{prop}
Pseudo $H$-type Lie algebras $\n_{r,s}(\mu_1,0)$ and $\n_{r,s}(\mu_2,0)$ and respectively $\n_{r,s}(0,\mu_1)$ and $\n_{r,s}(0,\mu_2)$ for $r-s\neq 3 \mod 4$ are isomorphic if and only if $\mu_1=\mu_2$. 
\end{prop}

\begin{theorem}
Two $H$-type Lie algebras $\n_{r,s}(\mu_1,\nu_1)$ and $\n_{r,s}(\mu_2,\nu_2)$ for $r-s=3 (\mod 4)$ are isomorphic if and only if $\mu_1=\mu_2$ and $\nu_1=\nu_2$ or $\mu_1=\nu_2$ and $\nu_1=\mu_2$.
\end{theorem}

\begin{proof}
In the first step we show that $\n_{r,s}(\mu,0)$ and $\n_{r,s}(0,\mu)$ are isomorphic for $r-s=3 \mod 4$. Let $J^1\colon \vv_{r,s}^1\oplus \z_{r,s}\to\vv_{r,s}^1$ and $J^2\colon \vv_{r,s}^2\oplus \z_{r,s}\to\vv_{r,s}^2$  be two non-equivalent representations over two minimal admissible modules. Let $\n_{r,s}(1,0)=(\vv_{r,s}^1\oplus \z_{r,s},[\cdot\,,\cdot]^1)$ and $\n_{r,s}(0,1)=(\vv_{r,s}^2\oplus \z_{r,s},[\cdot\,,\cdot]^2)$ be the pseudo $H$-type Lie algebras, where we used the maps $J^1$ and $J^2$ to define the corresponding brackets by~\eqref{def:J_Z}. We can assume that the vector spaces $\vv_{r,s}^1$ and $\vv_{r,s}^2$ are isomorphic under an isomorphism $A\colon \vv_{r,s}^1\to \vv_{r,s}^2$. We define a map $C\colon \z_{r,s}\to \z_{r,s}$ by
\begin{equation}\label{eq:AC}
J^1(v,C(Z))=A^{\tau}\circ J^2(A(v),Z),\quad\text{for any}\quad v\in\vv_{r,s}^1, \ \ Z\in\z_{r,s},
\end{equation}
where $\langle Av,u\rangle_{\vv^2_{r,s}}=\langle v,A^{\tau}u\rangle_{\vv^1_{r,s}}$.
We claim that the map $F=A\oplus C^{\tau}\colon \vv_{r,s}^1\oplus \z_{r,s}\to \vv_{r,s}^2\oplus \z_{r,s}$ is a Lie algebra isomorphism $F\colon \n_{r,s}(1,0)\to \n_{r,s}(0,1)$, where $C^{\tau}$ is the adjoint map to $C$ with respect to the scalar product $\langle \cdot \,, \cdot \rangle_{\z_{r,s}}$. Indeed, the chain of equalities
\begin{eqnarray*}
\langle Z,C^{\tau}([v,w]^1)\rangle_{\z_{r,s}}
&= & 
\langle C(Z),[v,w]^1\rangle_{\z_{r,s}}=
\langle J^1_{C(Z)}v,w \rangle_{\vv^1_{r,s}}=\langle A^{\tau}\circ J^2_{Z}(Av),w \rangle_{\vv^1_{r,s}}
\\
& = & \langle J^2_{Z}(Av),A w \rangle_{\vv^2_{r,s}}
=
\langle Z,[Av,Aw]^2\rangle_{\z_{r,s}}
\end{eqnarray*}
for any $v,w\in\vv^1_{r,s}$, $Z\in\z_{r,s}$ shows that $F([v,w]^1)=C^{\tau}([v,w]^1)=[Av,Aw]^2=[Fv,Fw]^2$.

To show that the Lie algebras $\n_{r,s}(\mu,\nu)$ and $\n_{r,s}(\nu,\mu)$ are isomorphic, we choose the map $A\colon\vv_{r,s}^1\to \vv_{r,s}^2$ to be not only the isomorphism of vector spaces, but also an isometry between the admissible modules. It, particularly, implies that $A^{\tau}=A^{-1}$. The corresponding map $C\colon\z_{r,s}\to\z_{r,s}$ will also be an isometry by Theorem~\ref{Idauto}. We fix an orthonormal basis $Z_1,\ldots, Z_{r+s}$ of $\z_{r,s}$, then the set $C(Z_1),\ldots, C(Z_{r+s})$ also forms an orthonormal basis. We construct an integral basis $v_{11}^2,\ldots, v_{n1}^2$ of $\vv_{r,s}^2$ by using the map $J^2\colon\z_{r,s}\oplus\vv_{r,s}^2\to\vv_{r,s}^2$ and the orthonormal basis $Z_1,\ldots, Z_{r+s}$ as it was done in~\cite{FurutaniIrina}. Then, by making use of the same method, we obtain the integral basis $v_{11}^1,\ldots, v_{n1}^1$ constructed from the orthonormal basis $C(Z_1),\ldots ,C(Z_{r+s})$ and the map $J^1\colon\z_{r,s}\oplus\vv_{r,s}^1\to\vv_{r,s}^1$. By the choice of the map $A$ we get $\prod_{k=1}^{l}J^1_{C(Z_{i_k})}=A^{-1}\circ\Big(\prod_{k=1}^{l} J^2_{Z_{i_k}}\Big)\circ A$ for any choice of orthonormal generators $Z_{i_1},\ldots,Z_{i_l}$ in $\z_{r,s}$. It guarantees that there is a vector $v\in\vv_{r,s}^1$ such that 
$
\langle v,v\rangle_{\vv^1_{r,s}}=\langle Av,Av\rangle_{\vv^2_{r,s}}
$
and $\prod_{k=1}^{l}J^1_{C(Z_{i_k})}v=v$ implies $\prod_{k=1}^{l} J^2_{Z_{i_k}}( Av)=Av$. The method of the construction of the integral basis in~\cite{FurutaniIrina} implies that $v_{ij}^2=Av_{ij}^1$. Hence the structural constants with respect to the basis $\{v_{11}^2, \dotso,v_{n1}^2, Z_1, \dotso, Z_{r+s}\}$ are identical to the structural constants with respect to the basis $\{ v_{11}^1, \dotso,v_{1n}^1,C(Z_1), \dotso, C(Z_{r+s})\}$. More precise, if we write $(A_{ip}^k)^1_1=(A_{ip}^k)^2_1=A_{ip}^k$ in the notation~\eqref{noncom}, then
$$
[v_{i1}^1,v_{p1}^1]=\sum_{k=1}^{r+s}A_{ip}^kC(Z_k)\quad\text{and}\quad
[v_{i1}^2,v_{p1}^2]=\sum_{k=1}^{r+s}A_{ip}^kZ_k.
$$
We can find the exact form of the map $C\colon\z_{r,s}\to\z_{r,s}$. Let $\mathcal Z=\{Z_1, \dotso, Z_{r+s}\}$ be an orthonormal basis for $\z_{r,s}$. Then the volume elements have different actions on their modules, namely $\omega^1(\mathcal Z)=\prod_{k=1}^{r+s}J^1_{Z_k}=\Id$ on $\vv_{rs}^1$ and $\omega^2(\mathcal Z)=\prod_{k=1}^{r+s}J^2_{Z_k}=-\Id$ on $\vv_{rs}^2$,
see~\cite{LawMich}. Let $A\colon\vv_{r,s}^1\to\vv_{r,s}^2$ be an isometry and $C\colon\z_{r,s}\to\z_{r,s}$ be the mapping induced by~\eqref{eq:AC}. Then
$$
\omega^1(C(\mathcal Z))v=\prod_{k=1}^{r+s}J^1_{C(Z_k)}v=A^{-1}\circ\prod_{k=1}^{r+s}J^2_{Z_k}\circ Av=A^{-1}\omega^2(\mathcal Z)Av=A^{-1}(-Av)=-v
$$
for $v\in\vv_{r,s}^1$. Since for $r-s=3\mod 4$ we have $r+s=2(s+2k+1)+1$, $k\in\mathbb Z$, we conclude that $r+s$ is an odd number. Then from 
\begin{equation*}
\prod_{k=1}^{r+s}J^1_{Z_k}=\omega^1(\mathcal Z)=\Id_{\vv_{r,s}^1}=-\omega^1(C(\mathcal Z))=-\prod_{k=1}^{r+s}J^1_{C(Z_k)}=\prod_{k=1}^{r+s}J^1_{-C(Z_k)}
\end{equation*}
we can assume that the map $C\colon\z_{r,s}\to\z_{r,s}$ maps the basis $\mathcal Z=\{Z_1, \dotso, Z_{r+s}\}$ to the basis $-\mathcal Z=\{-Z_1, \dotso, -Z_{r+s}\}$.
We write 
$$\{v_{ij}^1,\, v_{iq}^2,\, Z_k \vert i=1, \dotso,n, \quad j=1, \dotso,\mu, \quad q=1, \dotso,\nu, \quad k=1, \dotso,r+s \}$$ 
for an integral basis of $\n_{r,s}(\mu,\nu)$ where $v_{ij}^1$ is the $i$-th coordinate in the $j$-s copy of the module $\vv_{r,s}^1$ and $v_{iq}^2$ is the $i$-th coordinate in the $q$-s copy of the $n$-dimensional admissible module $\vv_{r,s}^2$.  Analogously, 
$$\{v_{iq}^2,\, v_{ij}^1,\, Z_k \vert i=1, \dotso,n, \quad q=1, \dotso,\mu, \quad j=1, \dotso,\nu, \quad k=1, \dotso,r+s \}$$
is an integral basis of $ \n_{r,s}(\nu,\mu)$.
Recall that in both Lie algebras $\n_{r,s}(\mu,\nu)$ and $\n_{r,s}(\nu,\mu)$ the following relations hold: 
$
[v_{ij}^1, v_{pq}^1]=[v_{ij}^1,v_{pr}^2]=[v_{ij}^2,v_{pq}^2]=0
$ 
for $j\not=q$ and for any $i,p,r$. Moreover
\begin{equation}\label{eq: v1v2}
[v_{ij}^1, v_{pj}^1]=[v_{i1}^1,v_{p1}^1]=-\sum_{k=1}^{r+s}A_{ip}^kZ_k,
\qquad 
[v_{ij}^2,v_{pj}^2]=[v_{i1}^2,v_{p1}^2]=\sum_{k=1}^{r+s}A_{ip}^kZ_k
\end{equation} 
for the above chosen $C\colon\z_{r,s}\to\z_{r,s}$.
The bijective linear map $f\colon\n_{r,s}(\mu,\nu) \to \n_{r,s}(\nu,\mu)$ defined by
\begin{equation*}
\begin{array}{llllll}
&v_{ij}^1 &\mapsto&  v_{ij}^2=A(v_{ij}^1), \qquad &\text{ for } \quad j=1, \dotso,\mu, \\
&v_{ip}^2 &\mapsto& v_{ip}^1=A^{-1}(v_{ip}^2), \qquad &\text{ for } \quad p=1, \dotso,\nu, \\
&Z_k &\mapsto& -Z_k, \qquad &\text{ for } \quad k=1, \dotso,r+s,
\end{array}
\end{equation*}
and $i,p=1,\ldots,n$ induces a Lie algebra homomorphism. Indeed, by~\eqref{eq: v1v2}
\begin{eqnarray*}
f([v_{ij}^1,v_{pj}^1])
&=&f([v_{i1}^1,v_{p1}^1])=f(-\sum_{k=1}^{r+s}A_{ip}^kZ_k)=-\sum_{k=1}^{r+s}A_{ip}^kf(Z_k)
\\
&=&
\sum_{k=1}^{r+s}A_{ip}^kZ_k=[v_{ij}^2,v_{pj}^2]
=
[A(v_{ij}^1),A(v_{pj}^1)]=[f(v_{ij}^1),f(v_{pj}^1)].
\end{eqnarray*}
and, analogously,
\begin{eqnarray*}
f([v_{ij}^2,v_{pj}^2])
&=&f([v_{i1}^2,v_{p1}^2])=f(\sum_{k=1}^{r+s}A_{ip}^kZ_k)=\sum_{k=1}^{r+s}A_{ip}^kf(Z_k)
\\
&=&
-\sum_{k=1}^{r+s}A_{ip}^kZ_k=[v_{ij}^1,v_{pj}^1]
=
[A^{-1}(v_{ij}^2),A^{-1}(v_{pj}^2)]=[f(v_{ij}^2),f(v_{pj}^2)].
\end{eqnarray*}

To show the reverse statement we assume that Lie algebras $\n_{r,s}(\mu_1,\nu_1)$ and $\n_{r,s}(\mu_2,\nu_2)$ are isomorphic for some $\mu_1>\mu_2$ and $\mu_1>\nu_2$. Then there are bijective maps $A_{12}\colon\vv_{r,s}^1\to\vv_{r,s}^2$ and $A_{11}\colon\vv_{r,s}^1\to\vv_{r,s}^1$ of minimal dimensional modules where the map $A_{12}$ induces $C$ by~\eqref{eq:AC} and $C$ induces $A_{11}$ by~\eqref{con}. Then we obtain
$
J^1_{C(Z)}=A_{12}^{\tau}\circ J^2_{Z}\circ A_{12}=A_{11}^{\tau}\circ J^1_{Z}\circ A_{11}$,
that contradicts to the assumption that modules $\vv_{r,s}^1$ and $\vv_{r,s}^2$ are non equivalent.
\end{proof}

%%%%%%%%%%%%%%%%%%%%%%%%%%%

\subsection{Open problems on classification of $H$-type Lie algebras $\n_{r,s}(\mu,\nu)$.}

%%%%%%%%%%%%%%%%%%%%%%%%%%%

The problem of the isomorphism of the pseudo $H$-type Lie algebras $\n_{r,s}(\mu,\nu)$ with different signatures $(r,s)$ turns out to be not so trivial. The increasing of dimension of admissible modules allows more freedom for action of the representation maps and some isomorphic Lie algebras can appear. For instance, it is possible to show, in the above notation, that the Lie algebras
$
\n_{2,1}
$ and $\n_{1,2}(1,1)$ are isomorphic, but the Lie algebras $
\n_{2,1}
$ and $\n_{1,2}(2,0)$ are not isomorphic. Thus we leave the full description of classification of pseudo $H$-type Lie algebras for forthcoming paper. 

%%%%%%%%%%%%%%%%%%%%%%%%%

\subsection{Bracket generating properties}

%%%%%%%%%%%%%%%%%%%%%%%%%

\begin{theorem}
The pseudo $H$-type algebras $\n_{r,s}(\mu,\nu)$ possesses the strongly bracket generating property if only if $r=0$ or $s=0$.
\end{theorem}

\begin{proof}
First we prove that $\n_{r,0}(\mu,\nu)$ is strongly bracket generating, i.e. $[ w \,, \vv_{r,0}(\mu,\nu)]=\z_{r,0}$ for all $w \in \n_{r,0}(\mu,\nu)$, $w \not=0$. We recall that $\vv_{r,0}(\mu,\nu)=\oplus_{j=1}^{\mu} (\vv_{r,0}^1)_j\oplus_{j=1}^{\nu}(\vv_{r,0}^2)_j$. Recall that $\n_{r,0}$ has the strongly bracket generating property for any $r \in \mathbb{N}$ by Theorem~\ref{stbracket}. Thus, we obtain that $[ v \,, (\vv_{r,0}^l)_j]=\z_{r,0}$ for all $v \in (\vv_{r,0}^l)_j\setminus \{0\}$, $l=1,2$, $j=1,\ldots,\mu+\nu$. 

Let $w \in \vv_{r,0}(\mu,\nu)$, $w \not=0$. There is an index $j\in\{1,\ldots, \mu+\nu\}$ such that the orthogonal projection of $w$ to $(\vv_{r,0}^l)_j=:\vv$, $l=1$ or $l=2$ is not vanishing. We obtain
\begin{equation*}
\z_{r,0} \supset [ w,\vv_{r,0}(\mu,\nu)] \supset [w,  \vv]  = \z_{r,0}.
\end{equation*}
Hence $[ w,\vv_{r,0}(\mu,\nu)] =\z_{r,0}$, i.e. $\n_{r,0}(\mu,\nu)$ is strongly bracket generating.

The proof for $\n_{0,r}(\mu,\nu)$ follows analogously. 

We consider the case $r,s \not=0$ and recall that $\n_{r,s}$ has not the strongly bracket generating property by Theorem~\ref{notstbracket}, i.e. there is $v \in \vv_{r,s}^1\setminus \{0\}$ such that $[v\,,\vv_{r,s}]\subsetneq \z_{r,s}$. Then the vector 
$w:=v \oplus \underbrace{0 \oplus \cdots \oplus 0}_{\mu+\nu-1 \text{ times } } \in \vv_{r,0}(\mu,\nu)$ satisfies
$
[w , \vv_{r,0}(\mu,\nu)]= [v, (\vv^1_{r,s})_1] \subsetneq  \z_{r,s}.
$
Hence $\n_{r,0}(\mu,\nu)$ do not have the strongly bracket generating property.
\end{proof}

%%%%%%%%%%%%%%%%%%%%%%%%%%%%

\section{Appendix}\label{appendix}

%%%%%%%%%%%%%%%%%%%%%%%%%%%%

In the tables we indicate by $[r,c]$ that the commutators are calculated as $[row,column]$.
\begin{table}[h!]
{\tiny 
\center\caption{Commutation relations on $\n_{8,0}$}
\begin{tabular}{|c|c| c | c | c | c | c | c | c |c | c | c | c | c | c | c | c |}
\hline
 $[r,c]$&$u_1$&$u_2$ & $u_3$ & $u_4$ & $u_5$ & $u_6$ & $u_7$ & $u_8$ & $u_9$ & $u_{10}$ & $u_{11}$ & $u_{12}$ & $u_{13}$ & $u_{14}$ & $u_{15}$ & $u_{16}$  \\
\hline
$u_1$ & $0$ & $0$ & $0$ & $0$ & $0$ & $0$ & $0$ & $0$ & $Z_1$ & $Z_2$ & $Z_3$ & $Z_4$ & $Z_5$ & $Z_6$ & $Z_7$ & $Z_8$ \\ 
\hline
$u_2$ & $0$ & $0$ & $0$ & $0$ & $0$ & $0$ & $0$ & $0$ & $Z_2$ & $-Z_1$ & $-Z_4$ & $Z_3$ & $-Z_6$ & $Z_5$ & $-Z_8$ & $Z_7$ \\
\hline
$u_3$ & $0$ & $0$ & $0$ & $0$ & $0$ & $0$ & $0$ & $0$ & $Z_3$ & $Z_4$ & $-Z_1$ & $-Z_2$ & $Z_8$ & $Z_7$ & $-Z_6$ & $-Z_5$ \\
\hline
$u_4$ & $0$ & $0$ & $0$ & $0$ & $0$ & $0$ & $0$ & $0$ & $Z_4$ & $-Z_3$ & $Z_2$ & $-Z_1$ & $Z_7$ & $-Z_8$ & $-Z_5$ & $Z_6$ \\
\hline
$u_5$ & $0$ & $0$ & $0$ & $0$ & $0$ & $0$ & $0$ & $0$ & $Z_5$ & $Z_6$ & $-Z_8$ & $-Z_7$ & $-Z_1$ & $-Z_2$ & $Z_4$ & $Z_3$ \\
\hline
$u_6$ & $0$ & $0$ & $0$ & $0$ & $0$ & $0$ & $0$ & $0$ & $Z_6$ & $-Z_5$ & $-Z_7$ & $Z_8$ & $Z_2$ & $-Z_1$ & $Z_3$ & $-Z_4$ \\
\hline
$u_7$ & $0$ & $0$ & $0$ & $0$ & $0$ & $0$ & $0$ & $0$ & $Z_7$ & $Z_8$ & $Z_6$ & $Z_5$ & $-Z_4$ & $-Z_3$ & $-Z_1$ & $-Z_2$ \\
\hline
$u_8$ & $0$ & $0$ & $0$ & $0$ & $0$ & $0$ & $0$ & $0$ & $Z_8$ & $-Z_7$ & $Z_5$ & $-Z_6$ & $-Z_3$ & $Z_4$ & $Z_2$ & $-Z_1$ \\
\hline
$u_9$ & $-Z_1$ & $-Z_2$ & $-Z_3$ & $-Z_4$ & $-Z_5$ & $-Z_6$ & $-Z_7$ & $-Z_8$ & $0$ & $0$ & $0$ & $0$ & $0$ & $0$ & $0$ & $0$\\
\hline
$u_{10}$ & $-Z_2$ & $Z_1$ & $-Z_4$ & $Z_3$ & $-Z_6$ & $Z_5$ & $-Z_8$ & $Z_7$ & $0$ & $0$ & $0$ & $0$ & $0$ & $0$ & $0$ & $0$ \\
\hline
$u_{11}$ & $-Z_3$ & $Z_4$ & $Z_1$ & $-Z_2$ & $Z_8$ & $Z_7$ & $-Z_6$ & $-Z_5$ & $0$ & $0$ & $0$ & $0$ & $0$ & $0$ & $0$ & $0$ \\
\hline
$u_{12}$ & $-Z_4$ & $-Z_3$ & $Z_2$ & $Z_1$ & $Z_7$ & $-Z_8$ & $-Z_5$ & $Z_6$ & $0$ & $0$ & $0$ & $0$ & $0$ & $0$ & $0$ & $0$ \\
\hline
$u_{13}$ & $-Z_5$ & $Z_6$ & $-Z_8$ & $-Z_7$ & $Z_1$ & $-Z_2$ & $Z_4$ & $Z_3$ & $0$ & $0$ & $0$ & $0$ & $0$ & $0$ & $0$ & $0$ \\
\hline
$u_{14}$ & $-Z_6$ & $-Z_5$ & $-Z_7$ & $Z_8$ & $Z_2$ & $Z_1$ & $Z_3$ & $-Z_4$ & $0$ & $0$ & $0$ & $0$ & $0$ & $0$ & $0$ & $0$ \\
\hline
$u_{15}$ & $-Z_7$ & $Z_8$ & $Z_6$ & $Z_5$ & $-Z_4$ & $-Z_3$ & $Z_1$ & $-Z_2$ & $0$ & $0$ & $0$ & $0$ & $0$ & $0$ & $0$ & $0$ \\
\hline
$u_{16}$ & $-Z_8$ & $-Z_7$ & $Z_5$ & $-Z_6$ & $-Z_3$ & $Z_4$ & $Z_2$ & $Z_1$ & $0$ & $0$ & $0$ & $0$ & $0$ & $0$ & $0$ & $0$ \\
\hline 
\end{tabular}
\label{80}
}
\end{table}
\begin{table}[h]
{\tiny
%\begin{sidewaystable}
%\label{Cl08}
%\vspace*{8cm}
\center\caption{Commutation relations on $\n_{0,8}$}
\begin{tabular}{| c | c | c | c | c | c | c | c | c |c | c | c | c | c | c | c | c |} 
\hline
 $[r,c]$  & $v_1$ & $v_2$ & $v_3$ & $v_4$ & $v_5$ & $v_6$ & $v_7$ & $v_8$ & $v_9$ & $v_{10}$ & $v_{11}$ & $v_{12}$ & $v_{13}$ & $v_{14}$ & $v_{15}$ & $v_{16}$  \\
\hline
$v_1$ & $0$ & $0$ & $0$ & $0$ & $0$ & $0$ & $0$ & $0$ & $\tilde{Z}_1$ & $\tilde{Z}_2$ & $\tilde{Z}_3$ & $\tilde{Z}_4$ & $\tilde{Z}_5$ & $\tilde{Z}_6$ & $\tilde{Z}_7$ & $\tilde{Z}_8$ \\ 
\hline
$v_2$ & $0$ & $0$ & $0$ & $0$ & $0$ & $0$ & $0$ & $0$ & $-\tilde{Z}_2$ & $\tilde{Z}_1$ & $\tilde{Z}_4$ & $-\tilde{Z}_3$ & $\tilde{Z}_6$ & $-\tilde{Z}_5$ & $\tilde{Z}_8$ & $-\tilde{Z}_7$ \\
\hline
$v_3$ & $0$ & $0$ & $0$ & $0$ & $0$ & $0$ & $0$ & $0$ & $-\tilde{Z}_3$ & $-\tilde{Z}_4$ & $\tilde{Z}_1$ & $\tilde{Z}_2$ & $-\tilde{Z}_8$ & $-\tilde{Z}_7$ & $\tilde{Z}_6$ & $\tilde{Z}_5$ \\
\hline
$v_4$ & $0$ & $0$ & $0$ & $0$ & $0$ & $0$ & $0$ & $0$ & $-\tilde{Z}_4$ & $\tilde{Z}_3$ & $-\tilde{Z}_2$ & $\tilde{Z}_1$ & $-\tilde{Z}_7$ & $\tilde{Z}_8$ & $\tilde{Z}_5$ & $-\tilde{Z}_6$ \\
\hline
$v_5$ & $0$ & $0$ & $0$ & $0$ & $0$ & $0$ & $0$ & $0$ & $-\tilde{Z}_5$ & $-\tilde{Z}_6$ & $\tilde{Z}_8$ & $\tilde{Z}_7$ & $\tilde{Z}_1$ & $\tilde{Z}_2$ & $-\tilde{Z}_4$ & $-\tilde{Z}_3$ \\
\hline
$v_6$ & $0$ & $0$ & $0$ & $0$ & $0$ & $0$ & $0$ & $0$ & $-\tilde{Z}_6$ & $\tilde{Z}_5$ & $\tilde{Z}_7$ & $-\tilde{Z}_8$ & $-\tilde{Z}_2$ & $\tilde{Z}_1$ & $-\tilde{Z}_3$ & $\tilde{Z}_4$ \\
\hline
$v_7$ & $0$ & $0$ & $0$ & $0$ & $0$ & $0$ & $0$ & $0$ & $-\tilde{Z}_7$ & $-\tilde{Z}_8$ & $-\tilde{Z}_6$ & $-\tilde{Z}_5$ & $\tilde{Z}_4$ & $\tilde{Z}_3$ & $\tilde{Z}_1$ & $\tilde{Z}_2$ \\
\hline
$v_8$ & $0$ & $0$ & $0$ & $0$ & $0$ & $0$ & $0$ & $0$ & $-\tilde{Z}_8$ & $\tilde{Z}_7$ & $-\tilde{Z}_5$ & $\tilde{Z}_6$ & $\tilde{Z}_3$ & $-\tilde{Z}_4$ & $-\tilde{Z}_2$ & $\tilde{Z}_1$ \\
\hline
$v_9$ & $-\tilde{Z}_1$ & $\tilde{Z}_2$ & $\tilde{Z}_3$ & $\tilde{Z}_4$ & $\tilde{Z}_5$ & $\tilde{Z}_6$ & $\tilde{Z}_7$ & $\tilde{Z}_8$ & $0$ & $0$ & $0$ & $0$ & $0$ & $0$ & $0$ & $0$\\
\hline
$v_{10}$ & $-\tilde{Z}_2$ & $-\tilde{Z}_1$ & $\tilde{Z}_4$ & $-\tilde{Z}_3$ & $\tilde{Z}_6$ & $-\tilde{Z}_5$ & $\tilde{Z}_8$ & $-\tilde{Z}_7$ & $0$ & $0$ & $0$ & $0$ & $0$ & $0$ & $0$ & $0$ \\
\hline
$v_{11}$ & $-\tilde{Z}_3$ & $-\tilde{Z}_4$ & $-\tilde{Z}_1$ & $\tilde{Z}_2$ & $-\tilde{Z}_8$ & $-\tilde{Z}_7$ & $\tilde{Z}_6$ & $\tilde{Z}_5$ & $0$ & $0$ & $0$ & $0$ & $0$ & $0$ & $0$ & $0$ \\
\hline
$v_{12}$ & $-\tilde{Z}_4$ & $\tilde{Z}_3$ & $-\tilde{Z}_2$ & $-\tilde{Z}_1$ & $-\tilde{Z}_7$ & $\tilde{Z}_8$ & $\tilde{Z}_5$ & $-\tilde{Z}_6$ & $0$ & $0$ & $0$ & $0$ & $0$ & $0$ & $0$ & $0$ \\
\hline
$v_{13}$ & $-\tilde{Z}_5$ & $-\tilde{Z}_6$ & $\tilde{Z}_8$ & $\tilde{Z}_7$ & $-\tilde{Z}_1$ & $\tilde{Z}_2$ & $-\tilde{Z}_4$ & $-\tilde{Z}_3$ & $0$ & $0$ & $0$ & $0$ & $0$ & $0$ & $0$ & $0$ \\
\hline
$v_{14}$ & $-\tilde{Z}_6$ & $\tilde{Z}_5$ & $\tilde{Z}_7$ & $-\tilde{Z}_8$ & $-\tilde{Z}_2$ & $-\tilde{Z}_1$ & $-\tilde{Z}_3$ & $\tilde{Z}_4$ & $0$ & $0$ & $0$ & $0$ & $0$ & $0$ & $0$ & $0$ \\
\hline
$v_{15}$ & $-\tilde{Z}_7$ & $-\tilde{Z}_8$ & $-\tilde{Z}_6$ & $-\tilde{Z}_5$ & $\tilde{Z}_4$ & $\tilde{Z}_3$ & $-\tilde{Z}_1$ & $\tilde{Z}_2$ & $0$ & $0$ & $0$ & $0$ & $0$ & $0$ & $0$ & $0$ \\
\hline
$v_{16}$ & $-\tilde{Z}_8$ & $\tilde{Z}_7$ & $-\tilde{Z}_5$ & $\tilde{Z}_6$ & $\tilde{Z}_3$ & $-\tilde{Z}_4$ & $-\tilde{Z}_2$ & $-\tilde{Z}_1$ & $0$ & $0$ & $0$ & $0$ & $0$ & $0$ & $0$ & $0$ \\
\hline
\end{tabular}
\label{Cl08}
}
\end{table}

{\tiny
\begin{table}
%\vspace*{8cm}
\center\caption{Commutation relations on $\n_{4,4}$}
\begin{tabular}{| c | c | c | c | c | c | c | c | c |c | c | c | c | c | c | c | c |} 
\hline
 $[r, c]$  & $y_1$ & $y_6$ & $y_7$ & $y_8$ & $y_{13}$ & $y_{14}$ & $y_{15}$ & $y_{16}$ & $y_2$ & $y_{3}$ & $y_{4}$ & $y_{5}$ & $y_{9}$ & $y_{10}$ & $y_{11}$ & $y_{12}$  \\
\hline
$y_1$ & $0$ & $0$ & $0$ & $0$ & $0$ & $0$ & $0$ & $0$ & $Z_1$ & $Z_2$ & $Z_3$ & $Z_4$ & $Z_5$ & $Z_6$ & $Z_7$ & $Z_8$ \\ 
\hline
$y_6$ & $0$ & $0$ & $0$ & $0$ & $0$ & $0$ & $0$ & $0$ & $Z_2$ & $-Z_1$ & $-Z_4$ & $Z_3$ & $Z_6$ & $-Z_5$ & $Z_8$ & $-Z_7$ \\
\hline
$y_7$ & $0$ & $0$ & $0$ & $0$ & $0$ & $0$ & $0$ & $0$ & $Z_3$ & $Z_4$ & $-Z_1$ & $-Z_2$ & $Z_8$ & $Z_7$ & $-Z_6$ & $-Z_5$ \\
\hline
$y_8$ & $0$ & $0$ & $0$ & $0$ & $0$ & $0$ & $0$ & $0$ & $Z_4$ & $-Z_3$ & $Z_2$ & $-Z_1$ & $-Z_7$ & $Z_8$ & $Z_5$ & $-Z_6$ \\
\hline
$y_{13}$ & $0$ & $0$ & $0$ & $0$ & $0$ & $0$ & $0$ & $0$ & $Z_5$ & $-Z_6$ & $-Z_8$ & $Z_7$ & $Z_1$ & $-Z_2$ & $Z_4$ & $-Z_3$ \\
\hline
$y_{14}$ & $0$ & $0$ & $0$ & $0$ & $0$ & $0$ & $0$ & $0$ & $Z_6$ & $Z_5$ & $-Z_7$ & $-Z_8$ & $Z_2$ & $Z_1$ & $-Z_3$ & $-Z_4$ \\
\hline
$y_{15}$ & $0$ & $0$ & $0$ & $0$ & $0$ & $0$ & $0$ & $0$ & $Z_7$ & $-Z_8$ & $Z_6$ & $-Z_5$ & $-Z_4$ & $Z_3$ & $Z_1$ & $-Z_2$ \\
\hline
$y_{16}$ & $0$ & $0$ & $0$ & $0$ & $0$ & $0$ & $0$ & $0$ & $Z_8$ & $Z_7$ & $Z_5$ & $Z_6$ & $Z_3$ & $Z_4$ & $Z_2$ & $Z_1$ \\
\hline
$y_2$ & $-Z_1$ & $-Z_2$ & $-Z_3$ & $-Z_4$ & $-Z_5$ & $-Z_6$ & $-Z_7$ & $-Z_8$ & $0$ & $0$ & $0$ & $0$ & $0$ & $0$ & $0$ & $0$\\
\hline
$y_{3}$ & $-Z_2$ & $Z_1$ & $-Z_4$ & $Z_3$ & $Z_6$ & $-Z_5$ & $Z_8$ & $-Z_7$ & $0$ & $0$ & $0$ & $0$ & $0$ & $0$ & $0$ & $0$ \\
\hline
$y_{4}$ & $-Z_3$ & $Z_4$ & $Z_1$ & $-Z_2$ & $Z_8$ & $Z_7$ & $-Z_6$ & $-Z_5$ & $0$ & $0$ & $0$ & $0$ & $0$ & $0$ & $0$ & $0$ \\
\hline
$y_{5}$ & $-Z_4$ & $-Z_3$ & $Z_2$ & $Z_1$ & $-Z_7$ & $Z_8$ & $Z_5$ & $-Z_6$ & $0$ & $0$ & $0$ & $0$ & $0$ & $0$ & $0$ & $0$ \\
\hline
$y_{9}$ & $-Z_5$ & $-Z_6$ & $-Z_8$ & $Z_7$ & $-Z_1$ & $-Z_2$ & $Z_4$ & $-Z_3$ & $0$ & $0$ & $0$ & $0$ & $0$ & $0$ & $0$ & $0$ \\
\hline
$y_{10}$ & $-Z_6$ & $Z_5$ & $-Z_7$ & $-Z_8$ & $Z_2$ & $-Z_1$ & $-Z_3$ & $-Z_4$ & $0$ & $0$ & $0$ & $0$ & $0$ & $0$ & $0$ & $0$ \\
\hline
$y_{11}$ & $-Z_7$ & $-Z_8$ & $Z_6$ & $-Z_5$ & $-Z_4$ & $Z_3$ & $-Z_1$ & $-Z_2$ & $0$ & $0$ & $0$ & $0$ & $0$ & $0$ & $0$ & $0$ \\
\hline
$y_{12}$ & $-Z_8$ & $Z_7$ & $Z_5$ & $Z_6$ & $Z_3$ & $Z_4$ & $Z_2$ & $-Z_1$ & $0$ & $0$ & $0$ & $0$ & $0$ & $0$ & $0$ & $0$ \\
\hline
\end{tabular}
\label{n44}
%\end{sidewaystable} 
\end{table}
}

\begin{table}[h]
{\tiny
%\vspace*{8cm}
\center\caption{Permutations of the basis of $\n_{8,0}$ by $J_i$}
\begin{tabular}{| c | c | c | c | c | c | c | c | c |} 
\hline
 $J_iu_j$  & $J_1$ & $J_2$ & $J_3$ & $J_4$ &$J_5$ & $J_6$ &$J_7$ & $J_8$ \\
\hline
$u_1$ & $u_{9}$ & $u_{10}$ & $u_{11}$ & $u_{12}$ & $u_{13}$ & $u_{14}$ & $u_{15}$ & $u_{16}$ \\
\hline
$u_2$ & $-u_{10}$ & $u_{9}$ & $u_{12}$ & $-u_{11}$ & $u_{14}$ & $-u_{13}$ & $u_{16}$ & $-u_{15}$ \\
\hline
$u_3$ & $-u_{11}$ & $-u_{12}$ & $u_{9}$ & $u_{10}$ & $-u_{16}$ & $-u_{15}$ & $u_{14}$ & $u_{13}$ \\
\hline
$u_4$ & $-u_{12}$ & $u_{11}$ & $-u_{10}$ & $u_{9}$ & $-u_{15}$ & $u_{16}$ & $u_{13}$ & $-u_{14}$ \\
\hline
$u_5$ & $-u_{13}$ & $-u_{14}$ & $u_{16}$ & $u_{15}$ & $u_{9}$ & $u_{10}$ & $-u_{12}$ & $-u_{11}$ \\
\hline
$u_6$ & $-u_{14}$ & $u_{13}$ & $u_{15}$ & $-u_{16}$ & $-u_{10}$ & $u_{9}$ & $-u_{11}$ & $u_{12}$ \\
\hline
$u_7$ & $-u_{15}$ & $-u_{16}$ & $-u_{14}$ & $-u_{13}$ & $u_{12}$ & $u_{11}$ & $u_{9}$ & $u_{10}$ \\
\hline
$u_8$ & $-u_{16}$ & $u_{15}$ & $-u_{13}$ & $u_{14}$ & $u_{11}$ & $-u_{12}$ & $-u_{10}$ & $u_{9}$ \\
\hline
$u_9$ & $-u_{1}$ & $-u_{2}$ & $-u_{3}$ & $-u_{4}$ & $-u_{5}$ & $-u_{6}$ & $-u_{7}$ & $-u_{8}$ \\
\hline
$u_{10}$ & $u_{2}$ & $-u_{1}$ & $u_{4}$ & $-u_{3}$ & $u_{6}$ & $-u_{5}$ & $u_{8}$ & $-u_{7}$ \\
\hline
$u_{11}$ & $u_{3}$ & $-u_{4}$ & $-u_{1}$ & $u_{2}$ & $-u_{8}$ & $-u_{7}$ & $u_{6}$ & $u_{5}$ \\
\hline
$u_{12}$ & $u_{4}$ & $u_{3}$ & $-u_{2}$ & $-u_{1}$ & $-u_{7}$ & $u_{8}$ & $u_{5}$ & $-u_{6}$ \\
\hline
$u_{13}$ & $u_{5}$ & $-u_{6}$ & $u_{8}$ & $u_{7}$ & $-u_{1}$ & $u_{2}$ & $-u_{4}$ & $-u_{3}$ \\
\hline
$u_{14}$ & $u_{6}$ & $u_{5}$ & $u_{7}$ & $-u_{8}$ & $-u_{2}$ & $-u_{1}$ & $-u_{3}$ & $u_{4}$ \\
\hline
$u_{15}$ & $u_{7}$ & $-u_{8}$ & $-u_{6}$ & $-u_{5}$ & $u_{4}$ & $u_{3}$ & $-u_{1}$ & $u_{2}$ \\
\hline
$u_{16}$ & $u_{8}$ & $u_{7}$ & $-u_{5}$ & $u_{6}$ & $u_{3}$ & $-u_{4}$ & $-u_{2}$ & $-u_{1}$ \\
\hline
\end{tabular}
\label{PermutE}
}
\end{table}

\end{document}